\documentclass{article}
\usepackage{graphicx}
\usepackage{amsmath}
\usepackage{amsthm}
\usepackage{caption}
\usepackage{color,soul}
\providecommand{\keywords}[1]{\textbf{Keywords} #1}
\newtheorem{remark}{Remark}
\newtheorem{result}{Result}
\newtheorem{theorem}{Theorem}
\newtheorem{observation}{Observation}

\begin{document}

\title{Apriori and aposteriori error estimation of Subgrid multiscale stabilized finite element method for coupled unified Stokes-Brinkman/Transport model}

\author{Manisha Chowdhury, B.V. Rathish Kumar\thanks{ Email addresses: chowdhurymanisha8@gmail.com(M.Chowdhury); drbvrk11@gmail.com (B.V.R. Kumar)   } }
      
\date{Indian Institute of Technology Kanpur \\ Kanpur, Uttar Pradesh, India}

\maketitle
\begin{abstract}
In this study, we present a stabilized finite element analysis for completely unified Stokes-Brinkman problems fully coupled with variable coefficient transient Advection-Diffusion-Reaction equation($VADR$). As well we have carried out the stabilized finite element analysis for  Stokes-Brinkman model with interface conditions fully coupled with $VADR$. The viscosity of the fluid, involved in flow problem, depends on the concentration of the solute, whose transport is described by $VADR$ equation. The algebraic subgrid multiscale approach has been employed to arrive at the stabilized coupled variational formulation. For the time discretization the fully implicit Euler scheme has been used. A detailed derivation of both the apriori and aposteriori estimates for the stabilized subgrid multiscale finite element scheme have been presented. Few numerical experiments have been carried out to verify the credibility of the method.
\end{abstract}

 

\section{Introduction}
Solute transport model coupled with fluid flow equations in  porous media region plays a crucial role in modelling physical phenomena of physiological and environmental importance. For several decades many studies have focused on transport model and fluid flow equations separately. For instance \cite{RefT}-\cite{RefZ} have studied various numerical methods for solving advection-diffusion-reaction equation where as \cite{RefO}-\cite{RefP} have focused on studying Stokes-Darcy flow equations and \cite{RefZ}-\cite{RefZ5} and \cite{RefG}-\cite{RefL} have respectively worked with Navier-Stokes fluid flow model and Brinkman model as well. There are many other researches done in this regard. Here we have mentioned only few of them. \vspace{1mm}\\
Contemporary complex problems, such as surface water and ground water pollution due to various human activities, stenosis in human arteries etc, have made it inevitable to model the problems in more effective way and hence study of coupled fluid flow-transport equations is attracting great attention in today's research. [1] presents finite element analysis of  Stokes-Darcy flow coupled with transport equation in which the fluid flow equations are solved through mixed finite element method and transport equation for solute concentration is solved using local discontinuous Galerkin ($LDG$) method. As well another study [3] discusses continuous and discontinuous finite element methods for coupled Navier-Stokes/Darcy-transport problem. Though these are one way or weak coupling in the sense that velocity field obtained after solving flow problem becomes an input data in transport model. Cesmelioglu and Rivière in [2] has introduced a two-way strong coupling of Navier-Stokes/Darcy with transport equation through considering the viscosity of the fluid depending upon concentration of the solute. The strong coupling is proven to be more accurate in modelling problems. In [2] authors prove the existence and uniqueness of the weak solution of the variational form of Navier-Stokes/Darcy-transport model. Further under constrained viscosity consideration in \cite{RefF} authors derive apriori error estimates for a stabilized mixed finite element scheme for the strongly coupled Stokes-Darcy-ADR model though they present numerical results to a one-way coupling problem. \vspace{1mm}\\
In this paper we present Stokes-Brinkman in a completely unified manner strongly coupled with transport equation. The Brinkman equation is an extension to Darcy's law when boundary layer regions have to be taken into account. The Stokes-Brinkman system has an important role in modelling highly heterogeneous porous media problem. As per our knowledge adequate attention is not paid to study numerical method for solving this model. Only few works \cite{RefZ8}-\cite{RefZ11} are available in the literature dealing with Stokes-Brinkman model. Here we have derived algebraic subgrid scale ($ASGS$) stabilized finite element method for the coupled problem using continuous velocities and pressure spaces across the inter-element boundaries. In \cite{RefF} Rui and Zhang have studied a mixed stabilized finite element method for coupled Stokes-Darcy/transport problem, but $ASGS$ approach to study coupled Stokes-Brinkman/$VADR$ model is very new.
Hughes in \cite{RefZ6} has introduced the concept of stabilized multiscale subgrid method for Helmholtz
equation and further developments are going on afterwards. In \cite{RefZ7} Codina presents a study on comparison of stabilized finite element methods viz. $SUPG$, $GLS$, $SGS$, $Taylor-Galerkin$ etc. for solving diffusion-convection-reaction equation and experimentally shows that $SGS$ performs well in compared to other stabilized method. \vspace{1mm}\\
In particular the $ASGS$ approach consists of algebraic approximation of the subscales that arise from the decomposition of the exact solution field into resolvable scale and unresolvable scale, have been used for finite element scheme stabilization. Stabilization parameters are derived following the approach in \cite{RefE}, \cite{RefG} for $ASGS$ method. Apriori error estimates for the current stabilized ASGS finite element method for the unified strongly coupled Stokes-Brinkman/$VADR$ have been derived. Further the aposteriori error estimates following the residual approach have also been carried out. Further more this paper has considered a porous media flow governed by Stokes-Brinkman with interface condition strongly coupled with transport equation and presented the corresponding stabilized formulation of the problem. The interface conditions are taken care of by the standard continuity conditions of normal velocities, normal stresses and concentration at the interface and  the Beavers-Joseph-Saffman condition at the interface allows Stokes fluid to slip in contact with porous medium.  Numerical studies have shown the realization of theoretical order of convergence and the robustness of current stabilized $ASGS$ finite element method for Stokes-Brinkman-$VADR$ tightly coupled system. \vspace{1mm}\\
Organisation of the paper is as follows: Section 2  starts from introducing the model and finishes at Subgrid formulation going through two more subsections presenting weak formulation and semi-discrete formulation. Next section has elaborately described the derivation of apriori and aposteriori error estimations for this subgrid formulation. Section 4 presents stabilization formulation of coupled Stokes-Brinkman/transport equations through interfaces. At last section 5 contains numerical results to verify the numerical performance of the method.       
\section{Model problem}
Let $\Omega \subset R^d$, d=2,3 be an open bounded domain with piecewise smooth boundary $\partial\Omega$. For the sake of simplicity in further calculations, we have considered two dimensional model, but it can be easily extended for three dimensional model. Let us first mention the Stokes-Brinkman flow problem for a fluid as follows: \\
Find $\textbf{u}$: $\Omega$ $\times$ (0,T) $\rightarrow R^2$ and $p$: $\Omega \times$ (0,T) $\rightarrow R$ such that,
\begin{equation}
\begin{split}
- \mu(c) \Delta \textbf{u} + \sigma \textbf{u} + \bigtriangledown p & = \textbf{f}_1 \hspace{2mm} in \hspace{2mm} \Omega \times (0,T) \\
\bigtriangledown \cdot \textbf{u} &= f_2 \hspace{2mm} in \hspace{2mm} \Omega \times (0,T) \\
\textbf{u} &= \textbf{0} \hspace{2mm} on \hspace{2mm} \partial\Omega \times (0,T) \\
\textbf{u} &= \textbf{u}_0 \hspace{2mm} at \hspace{2mm} t=0 \\
\end{split}
\end{equation} 
where \textbf{u}= ($u_1,u_2$) is the velocity of the fluid or solvent, p is the pressure, $\mu(c)$ is the viscosity of the fluid depending on concentration $c$ of the dispersing mass of the solute, $\sigma$ is the inverse of permeability, $\textbf{f}_1$ is the body force, $f_2$ is source term and $\textbf{u}_0$ is the initial velocity. When $\sigma=0$ the flow problem is fully Stokes. For Stokes flow $f_2=0$ too.\vspace{2 mm}\\
This Stokes-Brinkman flow problem is fully-coupled with the following ADR equation with variable coefficients($VADR$), which represents the transportation of solute in the same domain $\Omega$.\\
Find $c$: $\Omega \times$ (0,T) $\rightarrow R$ such that,
\begin{equation}
\begin{split}
\phi \frac{\partial c}{\partial t}- \bigtriangledown \cdot \tilde{\bigtriangledown} c + \textbf{u} \cdot \bigtriangledown c + \alpha c & = g \hspace{2mm} in \hspace{2mm} \Omega \times (0,T) \\
\tilde{\bigtriangledown} c \cdot \textbf{n} &= 0 \hspace{2mm} on \hspace{2mm}\partial \Omega \times (0,T) \\
c & = c_0 \hspace{2mm} at \hspace{2mm} t=0\\
\end{split}
\end{equation}
where the notation, $\tilde{\bigtriangledown}: = (D_1 \frac{\partial}{\partial x}, D_2 \frac{\partial}{\partial y})$ \\
$\phi$ is the porosity,$D_1, D_2$ are variable diffusion coefficients, $\alpha$ is the reaction coefficient and $g$ denotes the source of solute mass, $\textbf{n}$ is the outward normal to $\partial\Omega$ and $c_0$ is the initial concentration of the solute. For purely Stokes flow problem $\phi$ takes value 1.  \vspace{1.0 mm}\\ 
Letting \textbf{U}= (\textbf{u},p,c) the equations all together can be written in the following operator form,
\begin{equation}
M\partial_t \textbf{U} + \mathcal{L} \textbf{U} = \textbf{F}
\end{equation}
where M, a matrix = diag(0,0,0,$\phi$), $\partial_t \textbf{U} = (\frac{\partial \textbf{u}}{\partial t}, \frac{\partial p}{\partial t}, \frac{\partial c}{\partial t})^T$ \\
\[
\mathcal{L} \textbf{U}=
  \begin{bmatrix}
    - \mu(c) \Delta \textbf{u} + \sigma \textbf{u} + \bigtriangledown p\\
    \bigtriangledown \cdot \textbf{u} \\
    - \bigtriangledown \cdot \tilde{\bigtriangledown} c + \textbf{u} \cdot \bigtriangledown c + \alpha c 
  \end{bmatrix}
\]
 and \[
\textbf{F}=
  \begin{bmatrix}
    \textbf{f}_1 \\
    f_2 \\
    g
  \end{bmatrix}
\]
Let us introduce the adjoint $\mathcal{L}^*$ of $\mathcal{L}$ as follows,
\[
\mathcal{L}^* \textbf{U}=
  \begin{bmatrix}
   - \mu(c) \Delta \textbf{u} + \sigma \textbf{u} - \bigtriangledown p\\
    -\bigtriangledown \cdot \textbf{u} \\
    - \bigtriangledown \cdot \tilde{\bigtriangledown} c - \textbf{u} \cdot \bigtriangledown c + \alpha c 
  \end{bmatrix}
\]
Now we impose suitable assumptions, that are necessary to conclude the results further, on the coefficients mentioned above.\\

\textbf{(i)} The fluid viscosity $\mu(c)= \mu \in C^0(R^+; R^+)$, the space of positive real valued functions defined on positive real numbers and we will have two positive real numbers $\mu_l$ and $\mu_u$ such that 
\begin{equation}
0 < \mu_l \leq \mu(x) \leq \mu_u \hspace{2mm} for \hspace{2mm} any \hspace{2mm} x\in R^+
\end{equation}

\textbf{(ii)} $D_1= D_1((x,y),t) \in C^0(R^2 \times (0,T);R)$ and $D_2= D_2((x,y),t) \in C^0(R^2 \times (0,T);R)$ where $ C^0(R^2\times (0,T);R)$ is the space of real valued continuous function defined on $R^2$ for fixed $t \in (0,T)$. Both are bounded quantity that is we can find lower and upper bounds for both of them. \\

\textbf{(iii)} $\sigma$ and $\alpha$ are positive constants.\\

\textbf{(iv)} The spaces of continuous solution $(\textbf{u},p,c)$ are assumed as: \\ $u_1, u_2 \in L^{\infty}(0,T;H^2(\Omega))\bigcap C^{0}(0,T; H_0^1(\Omega))$ and \\
$p \in L^{\infty}(0,T;H^1(\Omega))\bigcap C^{0}(0,T;L^2_0(\Omega)) $,
 $c \in L^{\infty}(0,T;H^2(\Omega))\bigcap C^0(0,T;H^1_0(\Omega))$\\
 
\textbf{(v)} One additional assumption on continuous velocity solution is that $u_1$ and $u_2$ are taken to be bounded functions on $\Omega$.

\subsection{Weak formulation} 
Let us first introduce the spaces as follows,\\
$H^1(\Omega)= \{v \in L^2(\Omega) : \bigtriangledown v \in L^2(\Omega) \} $ \vspace{1mm} \\
Let $V_s=H^1_0(\Omega) = \{ v \in H^1 (\Omega): v=0 \hspace{1 mm} on \hspace{1 mm} \partial \Omega \}$ and $Q_s=L^2(\Omega)$ and J= (0,T)\vspace{1mm}\\
Let $\tilde{\textbf{V}} $ := $L^2(0,T; V_s)\bigcap L^{\infty}(0,T; Q_s)$ \vspace{1 mm} \\
Let us introduce another notation $\textbf{V}_F = V_s \times V_s\times Q_s \times V_s$ \vspace{2mm}\\
The weak formulation of (1)-(2) is to find \textbf{U}= (\textbf{u},p,c): J $ \rightarrow \textbf{V}_F$ such that $\forall$ \textbf{V}=(\textbf{v},q,d) $\in  \textbf{V}_F$
\begin{equation}
 (\frac{\partial c}{\partial t}, d) + a_S(\textbf{u},\textbf{v})- b(\textbf{v},p)+ b(\textbf{u},q)+ a_T(c,d)= l_S^1(\textbf{v})+ l_S^2(q)+l_T(d) 
\end{equation} 
where $a_S(\textbf{u},\textbf{v})= \int_{\Omega} \mu(c) \bigtriangledown \textbf{u}:\bigtriangledown \textbf{v} + \sigma \int_{\Omega} \textbf{u} \cdot \textbf{v}$ \vspace{1 mm}\\
 $b(\textbf{v},q)= \int_{\Omega} (\bigtriangledown \cdot \textbf{v}) q$ \vspace{1 mm} \\
 $a_T(c,d) = \int_{\Omega} \tilde{\bigtriangledown}c \cdot \bigtriangledown d + \int_{\Omega} d \textbf{u} \cdot \bigtriangledown c + \alpha\int_{\Omega}cd $ \vspace{1 mm} \\
 $l_S^1 (\textbf{v})= \int_{\Omega} \textbf{f}_1 \cdot \textbf{v}$, $l_S^2(q)=\int_{\Omega} f_2 q$ and  $l_T(d)= \int_{\Omega} gd$ \vspace{1 mm} \\
Again the above formulation can be written as,\\
Find $\textbf{U} \in \textbf{V}_F$ such that
\begin{equation}
(M\partial_t \textbf{U},\textbf{V}) + B(\textbf{U}, \textbf{V}) = L(\textbf{V})   \hspace{2 mm} \forall \textbf{V} \in \textbf{V}_F
\end{equation}
where $B(\textbf{U}, \textbf{V}) = a_S(\textbf{u},\textbf{v})- b(\textbf{v},p)+ b(\textbf{u},q)+ a_T(c,d)$ \vspace{1mm}\\
 $L(\textbf{V})= l_S^1(\textbf{v})+l_S^2(q)+ l_T(d) $ \vspace{2 mm}\\
\begin{remark}
\cite{RefG} discusses about well posedness of unified Stokes-Darcy equation for positive viscosity coefficient.
\end{remark}

\begin{remark}
The existence of the weak solution of the variational form for coupled Stokes-Darcy/transport equation has been discussed in \cite{RefB}. Under the assumptions [(i)-(iv)] the existence of unique weak solution of the variational form (6) can be established easily following the approach presented in \cite{RefB}, as this model contains only linear terms.
\end{remark}

\subsection{Semi-discrete formulation}
In this section we will introduce the standard Galerkin finite element space discretization for the above variational form (6).\vspace{1 mm} \\
Let the domain $\Omega$ be discretized into finite numbers of subdomains $\Omega_k$ for k=1,2,...,$n_{el}$, where $n_{el}$ is the total number element subdomains. Let $h_k$ be the diameter of each subdomain $\Omega_k$ and h= $\underset{k=1,2,...n_{el}}{max} h_k$ \vspace{1 mm}\\
Let $\tilde{\Omega}= \bigcup_{k=1}^{n_{el}} \Omega_k$ be the union of interior elements.\vspace{1 mm}\\
Let $V_s^h$ and $Q_s^h$ be finite dimensional subspaces of $V_s$ and $Q_s$ respectively. They are taken as follows, \vspace{1 mm} \\
$V_s^h= \{ v \in V_s: v(\Omega_k)= \mathcal{P}^2(\Omega_k)\} $ \vspace{1 mm} \\
$Q_s^h= \{ q \in Q_s : q(\Omega_k)= \mathcal{P}^1(\Omega_k)\}$ \vspace{1 mm}\\
where $\mathcal{P}^1(\Omega_k)$ and $\mathcal{P}^2(\Omega_k)$ denote complete polynomial of order 1 and 2 respectively over each $\Omega_k$ for k=1,2,...,$n_{el}$. \\
Let us consider similar notation $\textbf{V}_F^h$ for corresponding finite dimensional subspace of $\textbf{V}_F$ where $\textbf{V}_F^h=V_s^h \times V_s^h \times Q_s^h \times V_s^h $\vspace{2 mm} \\
Now the Galerkin formulation of the variational form (6) will be as follows:\\
Find $\textbf{U}_h $= $(\textbf{u}_h,p_h,c_h)$: J $ \rightarrow \textbf{V}_F^h$ such that $\forall$ $\textbf{V}_h=(\textbf{v}_h,q_h,d_h)$ $\in \textbf{V}_F^h$
\begin{equation}
(M\partial_t \textbf{U}_h,\textbf{V}_h) + B(\textbf{U}_h, \textbf{V}_h) = L(\textbf{V}_h)   
\end{equation}
where $(M\partial_t \textbf{U}_h,\textbf{V}_h)$= $ (\frac{\partial c_h}{\partial t}, d_h)$ \vspace{1mm}\\
 $B(\textbf{U}_h, \textbf{V}_h) = a_S(\textbf{u}_h,\textbf{v}_h)- b(\textbf{v}_h,p_h)+ b(\textbf{u}_h,q_h)+ a_T(c_h,d_h)$ \vspace{1mm}\\
 $L(\textbf{V}_h)= l_S^1(\textbf{v}_h)+l_S^2(q_h) + l_T(d_h) $                               
 
  \subsection{Subgrid multiscale formulation}
This stabilization method has been introduced to correct the lack of stability that the Galerkin method suffers due to small diffusion coefficient. It involves decomposition of the solution space $\textbf{V}_F$ into the spaces of resolved scales and unresolved scales. The finite element space $\textbf{V}_F^h$ is considered as the space of resolved scales. Then the final form of subgrid formulation will be arrived while the elements of unresolved scales will be expressed in the terms of elements of resolved scales. \vspace{1 mm} \\
Following the procedure described in \cite{RefD} the variational subgrid scale model for this coupled equation will be written as follows,\vspace{1 mm} \\
Find $\textbf{U}_h $= $(\textbf{u}_h,p_h,c_h)$: J $ \rightarrow \textbf{V}_F^h$ such that $\forall$ $\textbf{V}_h=(\textbf{v}_h,q_h,d_h)$ $\in \textbf{V}_F^h$ 
\begin{equation}
(M\partial_t \textbf{U}_h,\textbf{V}_h) + B_{ASGS}(\textbf{U}_h, \textbf{V}_h)  = L_{ASGS}(\textbf{V}_h)  
\end{equation}
where $B_{ASGS}(\textbf{U}_h, \textbf{V}_h)= B(\textbf{U}_h, \textbf{V}_h)+ \sum_{k=1}^{n_{el}} (\tau_k'(M\partial_t \textbf{U}_h + \mathcal{L}\textbf{U}_h-\textbf{d}), -\mathcal{L}^*\textbf{V}_h)_{\Omega_k}- \sum_{k=1}^{n_{el}}((I-\tau_k^{-1}\tau_k')(M\partial_t \textbf{U}_h + \mathcal{L}\textbf{U}_h), \textbf{V}_h)_{\Omega_k}-\sum_{k=1}^{n_{el}} (\tau_k^{-1}\tau_k' \textbf{d}, \textbf{V}_h)_{\Omega_k}$ \vspace{2 mm}\\
$L_{ASGS}(\textbf{V}_h)= L(\textbf{V}_h)+ \sum_{k=1}^{n_{el}}(\tau_k' \textbf{F}, -\mathcal{L}^*\textbf{V}_h)_{\Omega_k}- \sum_{k=1}^{n_{el}}((I-\tau_k^{-1}\tau_k')\textbf{F}, \textbf{V}_h)_{\Omega_k}$  \vspace{1 mm} \\
where the stabilization parameter $\tau_k$ is in matrix form as 
\[
\tau_k= diag(\tau_{1k},\tau_{1k},\tau_{2k},\tau_{3k}) =
  \begin{bmatrix}
    \tau_{1k} I & 0 & 0 \\
    0 & \tau_{2k} & 0 \\
    0 & 0 & \tau_{3k}
  \end{bmatrix}
\]
and 
\[
\tau_k'= (\frac{1}{dt}M+ \tau_k^{-1})^{-1} =
  \begin{bmatrix}
    \tau_{1k}I & 0 & 0 \\
    0 & \tau_{2k} & 0 \\
    0 & 0 & \frac{\tau_{3k} dt}{dt+ \tau_{3k}}
  \end{bmatrix}\\
  = diag (\tau_{1k}',\tau_{1k}',\tau_{2k}',\tau_{3k}') \hspace{2mm}(say)
\]
I is an identity matrix.\vspace{1 mm}\\
$\textbf{d}$= $\sum_{i=1}^{n+1}(\frac{1}{dt}M\tau_k')^i(\textbf{F} -M\partial_t \textbf{U}_h - \mathcal{L}\textbf{U}_h)$\vspace{1 mm}\\
considering $d_i$ for i=1,2,3,4 are components of the matrix $\textbf{d}$ and it can be easily observed that $d_1,d_2,d_3$ are always 0 because of the matrix M. \vspace{2 mm}\\
We have the forms of the stabilization parameters $\tau_{1k}, \tau_{2k}$ for unified Stokes-Darcy problem in \cite{RefG} and $\tau_{3k}$ for ADR equation with variable coefficients in \cite{RefE} and for each k all the coefficients $\tau_{ik}$ coincide with $\tau_{i}$ for i=1,2,3 that is, for each k=1,2,...,$n_{el}$
\begin{equation}
\begin{split}
\tau_{1k} &= \tau_{1}= (c_1^{\textbf{u}} \frac{\mu_{\textbf{u}}}{h^2}+  c_2^{\textbf{u}} \sigma)^{-1} \\
\tau_{2k} &=\tau_{2}=c_1^p \mu_{\textbf{u}} \\
\tau_{3k} & = \tau_{3}= (\frac{9D}{4h^2} + \frac{3U}{2h} + \alpha )^{-1}
\end{split}
\end{equation}
where $c_1^{\textbf{u}},c_2^{\textbf{u}},c_1^p$ are the suitable parameters and $h$ is the mesh size. \vspace{2mm}\\
\begin{remark}
Since we are working with continuous velocities and pressure at the inter-element boundaries, therefore we will not have any jump term in the above stabilized formulation.
\end{remark}

\section{Error estimates}
We start this section with the introduction of the notion of error terms, followed by splitting of those error terms through introducing the projection operator corresponding to each unknown variable. Later we have introduced fully-discrete formulation and then conducted  apriori and aposteriori error estimates.

\subsection{Projection operators : Error splitting}
Let $\textbf{e}=(e_{\textbf{u}},e_p,e_c)$ denote the error where the components are $e_{\textbf{u}}=(e_{u1},e_{u2})= (u_1-u_{1h}, u_2-u_{2h}), e_p= (p-p_h)$ and $e_c=(c-c_h)$.
Here all the remaining notations carry their respective meanings. \vspace{2mm}\\
Let us introduce the projection operator for each of these error components.\vspace{1 mm}\\
(i)For any $\textbf{u} \in H^2(\Omega) \times H^2(\Omega) $ we assume that there exists an interpolation $I^h_{\textbf{u}}:  H^2(\Omega) \times H^2(\Omega) \longrightarrow  V_s^h \times V_s^h $ satisfying \vspace{1mm}\\
(a) $b(\textbf{u}-I^h_{\textbf{u}}\textbf{u}, q_h)=0$ \hspace{2mm} $\forall q_h \in Q_s^h$ and \vspace{1mm}\\
each component of the projection map that is $I^h_{u_1}: H^2(\Omega) \longrightarrow V_s^h$ and $I^h_{u_2}: H^2(\Omega) \longrightarrow V_s^h$ are $L^2$ orthogonal projections, satisfying \vspace{1mm}\\
(b) for any $u_1 \in H^2(\Omega)$ \hspace{1mm} $(u_1-I^h_{u_1}u_1,v_{1h})=0$ \hspace{1mm} $\forall v_{1h} \in V_s^h$ and \vspace{1mm} \\
(c)for any $u_2 \in H^2(\Omega)$ \hspace{1mm} $(u_2-I^h_{u_2} u_2,v_{2h})=0$ \hspace{1mm} $\forall v_{2h} \in V_s^h$ \vspace{2mm}\\
(ii) Let $I^h_p: H^1(\Omega) \longrightarrow Q_s^h$ be the $L^2$ orthogonal projection given by \\ $\int_{\Omega}(p-I^h_pp)q_h=0$  \hspace{1mm} $\forall q_h \in Q_s^h$ and for any $p \in H^1(\Omega)$ \vspace{2mm}\\
(iii)Let $I^h_{c}: H^2(\Omega) \longrightarrow V_s^h$ be the $L^2$ orthogonal projection given by \\ $\int_{\Omega}(c-I^h_c c)d_h=0$ \hspace{1mm} $ \forall d_h \in V_s^h$ and for any $c \in H^2(\Omega)$ \vspace{2mm}\\
Now each component of the error can be split into two parts interpolation part, $E^I$ and auxiliary part, $E^A$ as follows: \vspace{1mm}\\
$e_{u1}=(u_1-u_{1h})=(u_1-I^h_{u1}u_1)+(I^h_{u1}u_1-u_{1h})= E^{I}_{u1}+ E^{A}_{u1}$ \vspace{1mm}\\
Similarly, $e_{u2}=E^{I}_{u2}+ E^{A}_{u2}$,
$e_{p}=E^{I}_{p}+ E^{A}_{p}$, and 
$e_{c}=E^{I}_{c}+ E^{A}_{c}$ \vspace{2mm}\\
Now we put some results using the properties of projection operators and these results will be used in error estimations.
\begin{result}
\begin{equation}
(\frac{\partial}{\partial t} E^I_{c}, d_h)=0 \hspace{2mm} d_h \in V_s^h
\end{equation}
\end{result}
\textit{Proof:} We have $(c-I^h_{c}c,d_h)=0= (E^I_{c},d_h)$ \hspace{1mm} $\forall d_h \in V_s^h$ \vspace{1mm}\\
Therefore
\begin{equation}
\begin{split}
\frac{d}{dt}(E^I_{c},d_h) & =0 \hspace{1mm} \forall d_h \in V_s^h \\
(\frac{\partial}{\partial t} E^I_{c},d_h) + (E^I_{c}, \frac{\partial}{\partial t} d_h) & =0 \hspace{1mm} \forall d_h \in V_s^h \\
(\frac{\partial}{\partial t} E^I_{c},d_h) &=0 \hspace{1mm} \forall d_h \in V_s^h \\
\end{split}
\end{equation}
Since $\frac{\partial}{\partial t} d_h \in V_s^h$, the second term in second equation $(E^I_{c}, \frac{\partial}{\partial t} d_h)=0$ \vspace{2mm}\\
Useful interpolation estimation results \cite{RefQ} are as follows: for any exact solution with regularity upto (m+1)
\begin{equation}
\|v-I^h_v v\|_l = \|E^I_v\|_l \leq C(p,\Omega) h^{m+1-l} \|v\|_{m+1} 
\end{equation}
where l ($\leq m+1$) is a positive integer and C is a constant depending on m and the domain. For l=0 and 1 it  implies standard $L^2(\Omega)$ and $H^1(\Omega)$ norms respectively. For simplicity we will use $\| \cdot \|$ instead of $\| \cdot \|_0$ to denote $L^2(\Omega)$ norm.

\subsection{Fully-discrete form}
Before introducing time discretization, some notations  have been introduced: for $dt$= $\frac{T}{N}$, where $N$ is a positive integer, $t_n= n dt$ and for given $0 \leq \theta \leq 1$,
\begin{equation}
\begin{split}
f^n & = f(\cdot , t_n) \hspace{4 mm} for \hspace{2 mm} 0 \leq n \leq N\\
f^{n,\theta} &= \frac{1}{2} (1 + \theta) f^{(n+1)} + \frac{1}{2} (1- \theta) f^n \hspace{4mm} for \hspace{2mm} 0\leq n \leq N-1
\end{split}
\end{equation}
Later we will see for $\theta=0$ the discretization follows Crank-Nicolson formula and for $\theta=1$ it is backward Euler discretization rule.\vspace{1mm}\\
For sufficiently smooth function $f(t)$, using the Taylor series expansion about t= $t^{n,\theta}$, we will have \vspace{1mm}\\
\begin{equation}
\begin{split}
f^{n+1} & = f(t^{n,\theta})+ \frac{(1-\theta)  dt}{2} \frac{\partial f}{\partial t}(t^{n,\theta}) + \frac{(1-\theta)^2 dt^2}{8} \frac{\partial^2 f}{\partial t^2} (t^{n,\theta}) + \mathcal{O}(dt^3)\\
f^{n} & = f(t^{n,\theta})- \frac{(1+\theta) dt}{2} \frac{\partial f}{\partial t}(t^{n,\theta}) + \frac{(1+\theta)^2 dt^2}{8} \frac{\partial^2 f}{\partial t^2}(t^{n,\theta}) + \mathcal{O}(dt^3)
\end{split}
\end{equation}
We have considered here $t^{n,\theta}- t^n= \frac{(1+\theta) \Delta t}{2}$\\
Multiplying the above first and second sub-equations in (14) by $\frac{1+\theta}{2}$ and $\frac{1-\theta}{2}$ respectively and then adding them we will have the following\\
\begin{equation}
f^{n,\theta} = f(t^{n,\theta}) + \frac{1}{8} (1+\theta)(1-\theta) dt^2 \frac{\partial^2 f}{\partial t^2}(t^{n,\theta}) + \mathcal{O}(dt^3)
\end{equation} 
Let $\textbf{u}^{n,\theta},p^{n,\theta},c^{n,\theta}$ be approximations of $\textbf{u}(\textbf{x},t^{n,\theta}), p(\textbf{x},t^{n,\theta}),c(\textbf{x},t^{n,\theta})$ respectively. Now by Taylor series expansion \cite{RefZ12},we have 
\begin{equation}
\begin{split}
\frac{c^{n+1}-c^n}{dt} & = c_t(\textbf{x},t^{n,\theta}) + TE\mid_{t=t^{n,\theta}} \hspace{5mm} \forall \textbf{x} \in \Omega
\end{split}
\end{equation}
where the truncation error $TE\mid_{t=t^{n,\theta}}$ $\simeq$ $TE^{n,\theta}$ depends upon time-derivatives of the respective variables and $dt$.
\begin{equation}
\begin{split}
\|TE^{n,\theta}\| & \leq
      \begin{cases}
      C' dt \|c_{tt}^{n,\theta}\|_{L^{\infty}(t^n,t^{n+1},L^2)} & if \hspace{1mm} \theta=1 \\
    C'' dt^2 \|c_{ttt}^{n,\theta}\|_{L^{\infty}(t^n,t^{n+1},L^2)} & if \hspace{1mm} \theta=0
      \end{cases}
\end{split}
\end{equation}
Now for backward Euler scheme ($\theta$=1) applying assumption $\textbf{(iv)}$ we will have another property as follows:
\begin{equation}
\begin{split}
\|TE^{n,\theta}\| & \leq C' dt  \|c_{tt}^{n+1}\|_{L^{\infty}(t^n,t^{n+1},L^2)} \\
& \leq \tilde{C} dt
\end{split}
\end{equation}
After introducing all the required definitions finally the fully-discrete formulation of sub-grid form is as follows: \\
For given $\textbf{U}_h^n = (\textbf{u}_h^n,p_h^n,c_h^n)\in \textbf{V}_F^h$ find $\textbf{U}_h^{n+1}= (\textbf{u}_h^{n+1},p_h^{n+1},c_h^{n+1}) \in \textbf{V}_F^h $  such that , $\forall \hspace{1mm} \textbf{V}_h=(\textbf{v}_h,q_h,d_h) \in \textbf{V}_F^h $
\begin{equation}
(M\frac{(\textbf{U}_h^{n+1}-\textbf{U}_h^n)}{dt}, \textbf{V}_h)+ B_{ASGS}(\textbf{U}_h^{n,\theta}, \textbf{V}_h) = L_{ASGS}(\textbf{V}_h) + (TE^{n,\theta},d_h)  
\end{equation}
Again for the exact solution we will have the discrete formulation as follows: \\
For given $\textbf{U}^n = (\textbf{u}^n,p^n,c^n)\in \textbf{V}_F$ find $\textbf{U}^{n+1}= (\textbf{u}^{n+1},p^{n+1},c^{n+1}) \in \textbf{V}_F $  such that , $\forall \hspace{1mm} \textbf{V}_h=(\textbf{v}_h,q_h,d_h) \in \textbf{V}_F^h$
\begin{equation}
(M\frac{(\textbf{U}^{n+1}-\textbf{U}^n)}{dt}, \textbf{V}_h)+ B(\textbf{U}^{n,\theta}, \textbf{V}_h) = L(\textbf{V}_h) + (TE^{n,\theta},d_h)  
\end{equation}

\subsection{Apriori error estimation}
In this section we will find apriori error bound, which depends on the exact solution. Here we first estimate auxiliary error bound and later using that we will find apriori error estimate. Before deriving error estimations let us mention few definitions of norm in which we are going to estimate the errors:
\begin{equation}
\begin{split}
\|f\|_{L^2(H^1)}^2 &= \sum_{n=0}^{N-1} \int_{t^n}^{t^{n+1}} (\int_{\Omega} \mid f^{n,\theta} \mid^2  +  \int_{\Omega} \mid \frac{\partial f}{\partial x}^{n,\theta} \mid^2  +  \int_{\Omega} \mid \frac{\partial f}{\partial y}^{n,\theta} \mid^2 )dt \\
\|f\|_{\tilde{\textbf{V}}}^2 & = \underset{0\leq n \leq N}{max} \|f^n\|^2 + \|f\|_{L^2(H^1)}^2\\
\|f\|_{L^2(L^2)}^2 &= \sum_{n=0}^{N-1} \int_{t^n}^{t^{n+1}} \|f^{n,\theta}\|^2 dt 
\end{split}
\end{equation}
\begin{theorem} (Auxiliary error estimate)
For velocity $\textbf{u}_h=(u_{1h},u_{2h})$, pressure $p_h$ and concentration $c_h$ belonging to $V_s^h \times V_s^h \times Q_s^h \times V_s^h$ satisfying (10), assume dt is sufficiently small and positive, and sufficient regularity of exact solution in equations (1)-(2). Then there exists a constant C, depending upon $\textbf{u}$,p,c , such that
\begin{equation}
\|E^A_{u1}\|^2_{L^2(H^1)} + \|E^A_{u2}\|^2_{L^2(H^1)}+ \|E^A_p\|_{L^2(L^2)}^2  + \|E^A_{c}\|^2_{\tilde{\textbf{V}}} \leq C (h^2+h+ dt^{2r})
\end{equation}
where
\begin{equation}
    r=
    \begin{cases}
      1, & \text{if}\ \theta=1 \\
      2, & \text{if}\ \theta=0
    \end{cases}
  \end{equation}
\end{theorem}
\begin{proof} In first part we will find bound for auxiliary error part of velocity $\textbf{u}$ and concentration c with respect to ${\tilde{\textbf{V}}}$-norm and in the second part we will estimate auxiliary error for pressure term with respect to $Q$ norm and finally combining them we will arrive at the desired result. \vspace{2mm} \\
\textbf{First part} Subtracting (19) from (20) and then simplifying the terms, we have $\forall \hspace{1mm} \textbf{V}_h \in V_s^h \times V_s^h \times Q_s^h \times V_s^h$

\begin{multline}
(M\frac{(\textbf{U}^{n+1}-\textbf{U}^{n+1}_{h})- (\textbf{U}^{n}-\textbf{U}^{n}_{h})}{dt},V_{h}) + B(\textbf{U}^{n,\theta}-\textbf{U}^{n,\theta}_h, \textbf{V}_h)\\
+ \sum_{k=1}^{n_{el}}\tau_k'(M\partial_t (\textbf{U}^n-\textbf{U}^n_h)+ \mathcal{L}(\textbf{U}^{n,\theta}-\textbf{U}_h^{n,\theta}),-\mathcal{L}^* \textbf{V}_h)_{\Omega_k} -\sum_{k=1}^{n_{el}}\tau_k' (\textbf{d},-\mathcal{L}^* \textbf{V}_h)_{\Omega_k}\\
+\sum_{k=1}^{n_{el}}((I-\tau_k^{-1}\tau_k')(M \partial_t(\textbf{U}^{n}-\textbf{U}^{n}_h) + \mathcal{L}(\textbf{U}^{n,\theta}-\textbf{U}^{n,\theta}_h)), -\textbf{V}_h)_{\Omega_k}+ \sum_{k=1}^{n_{el}} (\tau_k^{-1}\tau_k' \textbf{d}, \textbf{V}_h)_{\Omega_k}\\
=(TE^{n,\theta}, d_h)
\end{multline}
where $\textbf{d}$= $(\sum_{i=1}^{n+1}(\frac{1}{dt}M\tau_k')^i)(M\partial_t (\textbf{U}^n-\textbf{U}^n_h) + \mathcal{L}(\textbf{U}^{n,\theta}-\textbf{U}_h^{n,\theta}))$ \vspace{2mm}\\
Let us divide the big expressions into small parts, then using error splitting in each of them and simplifying further, we will have them as follows: \vspace{1mm}\\
Let
\begin{equation}
\begin{split}
I_1 &= (M\frac{(\textbf{U}^{n+1}-\textbf{U}^{n+1}_{h})- (\textbf{U}^{n}-\textbf{U}^{n}_{h})}{dt},V_{h}) \\
& = (\frac{(c^{n+1}-c^{n+1}_{h})- (c^{n}-c^{n}_{h})}{dt},d_{h})\\
 & =  (\frac{(E^{I,n+1}_{c}+E^{A,n+1}_{c})- (E^{I,n}_{c}+E^{A,n}_{c})}{dt},d_{h})\\
 & =  (\frac{E^{A,n+1}_{c}-E^{A,n}_{c}}{dt},d_{h})
\end{split}
\end{equation}
We arrive at the last line after using result 1, deduced in the previous section. \\
\begin{equation}
\begin{split}
I_2 & = B(\textbf{U}^{n,\theta}-\textbf{U}^{n,\theta}_h, \textbf{V}_h) \\
& = \int_{\Omega} \mu(c^n)\bigtriangledown (\textbf{u}^{n,\theta}-\textbf{u}^{n,\theta}_h): \bigtriangledown \textbf{v}_h + \int_{\Omega} \sigma (u_1^{n,\theta}-u^{n,\theta}_{1h})v_{1h}\\
& \quad + \int_{\Omega} \sigma (u_1^{n,\theta}-u^{n,\theta}_{1h})v_{2h}-\int_{\Omega}(\bigtriangledown \cdot \textbf{v}_h)(p^{n,\theta}-p_h^{n,\theta})+ \int_{\Omega}(\bigtriangledown \cdot \textbf{u}^{n,\theta}-\textbf{u}_h^{n,\theta})q_h \\
& \quad + \int_{\Omega} \tilde{\bigtriangledown}(c^{n,\theta}-c^{n,\theta}_h) \cdot \bigtriangledown d_h + \int_{\Omega}d_h \textbf{u}^n \cdot \bigtriangledown (c^{n,\theta}-c^{n,\theta}_h) + \int_{\Omega} \alpha (c^{n,\theta}-c^{n,\theta}_h) d_h \\
& = \int_{\Omega} \mu(c^n)\bigtriangledown E^{I,n,\theta}_{\textbf{u}} : \bigtriangledown \textbf{v}_h + \int_{\Omega} \mu(c^n)\bigtriangledown E^{A,n,\theta}_{\textbf{u}} : \bigtriangledown \textbf{v}_h \\
& \quad + \int_{\Omega} \sigma E^{A,n,\theta}_{u1}v_{1h} + \int_{\Omega} \sigma E^{A,n,\theta}_{u2}v_{2h}- \int_{\Omega}(\bigtriangledown \cdot \textbf{v}_h)(E^{I,n,\theta}_{p} + E^{A,n,\theta}_{p}) \\
& \quad + \int_{\Omega}(\bigtriangledown \cdot  E^{A,n,\theta}_{\textbf{u}}) q_h + \int_{\Omega} \tilde{\bigtriangledown}E^{I,n,\theta}_{c} \cdot \bigtriangledown d_h + \int_{\Omega} \tilde{\bigtriangledown} E^{A,n,\theta}_{c} \cdot \bigtriangledown d_h \\
& \quad + \int_{\Omega}d_h \textbf{u}^n \cdot \bigtriangledown E^{I,n,\theta}_{c} + \int_{\Omega}d_h \textbf{u}^n \cdot \bigtriangledown E^{A,n,\theta}_{c} + \int_{\Omega} \alpha E^{A,n,\theta}_{c} d_h
\end{split}
\end{equation}
Applying various properties of the projection operators we have the final expression of $I_2$ above.
\begin{equation}
\begin{split}
I_3 &= \sum_{k=1}^{n_{el}}(\tau_k'(M\partial_t (\textbf{U}^n-\textbf{U}^n_h)+ \mathcal{L}(\textbf{U}^{n,\theta}-\textbf{U}_h^{n,\theta})),-\mathcal{L}^* \textbf{V}_h)_{\Omega_k} \\
& = \sum_{k=1}^{n_{el}} \{(\tau_1'(- \mu(c) \Delta (E^{I,n,\theta}_{u1}+E_{u1}^{A,n,\theta}) + \sigma (E^{I,n,\theta}_{u1}+E_{u1}^{A,n,\theta}) + \frac{\partial( E^{I,n,\theta}_{p})}{\partial x} +\frac{\partial( E^{A,n,\theta}_{p})}{\partial x} ), \\
& \quad (\mu(c) \Delta v_{1h} - \sigma v_{1h} + \frac{\partial q_h}{\partial x}))_{\Omega_k} +(\tau_1'(- \mu(c) \Delta (E^{I,n,\theta}_{u2}+ E_{u2}^{A,n,\theta}) +  \sigma (E^{I,n,\theta}_{u2}+E_{u2}^{A,n,\theta}) + \\
&\quad  \frac{\partial ( E^{I,n,\theta}_{p}+E_{p}^{A,n,\theta})}{\partial y}),
 (\mu(c) \Delta v_{2h} - \sigma v_{2h} + \frac{\partial q_h}{\partial y}))_{\Omega_k} + (\tau_2'\bigtriangledown \cdot (E^{I,n,\theta}_{\textbf{u}}+E_{\textbf{u}}^{A,n,\theta}), \bigtriangledown \cdot \textbf{v}_h)_{\Omega_k}\\ 
 & \quad + (\tau_3'( \partial_t(E^{I,n}_{c}+E_{c}^{A,n}) - \bigtriangledown \cdot \tilde{\bigtriangledown}(E^{I,n,\theta}_{c}+  E_{c}^{A,n,\theta}) + \textbf{u}^n \cdot \bigtriangledown (E^{I,n,\theta}_{c}+E_{c}^{A,n,\theta}) + \\
& \quad \alpha (E^{I,n,\theta}_{c}+E_{c}^{A,n,\theta})-d_4),\bigtriangledown \cdot \tilde{\bigtriangledown} d_h +\textbf{u}^n \cdot \bigtriangledown d_h - \alpha d_h )_{\Omega_k}\}\\
& = I_3^1 +I_3^2+I_3^3+I_3^4 
\end{split} 
\end{equation}
where $I_3^1,I_3^2,I_3^3,I_3^4$ are four terms of $I_3$ which we will discuss in the later part of the proof.
\begin{equation}
\begin{split}
I_4 & = \sum_{k=1}^{n_{el}}\tau_k'(-\textbf{d},-\mathcal{L}^* \textbf{V}_h)_{\Omega_k} = \sum_{k=1}^{n_{el}}\{ \tau_3' ( d_4, \bigtriangledown \cdot \tilde{\bigtriangledown} d_h +\textbf{u}^n \cdot \bigtriangledown d_h - \alpha d_h )_{\Omega_k}\}
\end{split}
\end{equation}
and since $(1- \tau_1^{-1}\tau_1')=0=(1- \tau_2^{-1}\tau_2')$ the next term will take the following form 
\begin{equation}
\begin{split}
I_5 &= \sum_{k=1}^{n_{el}}((I-\tau_k^{-1}\tau_k')(M \partial_t(\textbf{U}^{n}-\textbf{U}^{n}_h) + \mathcal{L}(\textbf{U}^{n,\theta}-\textbf{U}^{n,\theta}_h)), -\textbf{V}_h)_{\Omega_k}\\
 & = \sum_{k=1}^{n_{el}}((1-\tau_3^{-1}\tau_3') ( \partial_t E^{I,n}_{c} + \partial_t E^{A,n}_{c} - \bigtriangledown \cdot \tilde{\bigtriangledown}(E^{I,n,\theta}_{c}+E_{c}^{A,n,\theta})+ \textbf{u} \cdot \bigtriangledown (E^{I,n,\theta}_{c}+E_{c}^{A,n,\theta}) + \\
 & \quad \alpha (E^{I,n,\theta}_{c}+E_{c}^{A,n,\theta})),-d_h )_{\Omega_k}
\end{split}
\end{equation}
and the last term,
\begin{equation}
\begin{split}
I_6 & = \sum_{k=1}^{n_{el}} (\tau_k^{-1}\tau_k' \textbf{d}, \textbf{V}_h)_{\Omega_k}=\sum_{k=1}^{n_{el}}(\tau_3^{-1}\tau_3'd_4, d_h)_{\Omega_k}
\end{split}
\end{equation}
Now taking all these terms together, (23) becomes
\begin{equation}
I_1+I_2+I_3+I_4+I_5+ I_6 =(TE^{n,\theta}, d_h), \hspace{1mm} \forall \hspace{1mm} \textbf{V}_h \in V_s^h \times V_s^h \times Q_s^h \times V_s^h \\
\end{equation}
This implies
\begin{multline}
I_1 + \int_{\Omega} \mu(c^n)\bigtriangledown E^{A,n,\theta}_{\textbf{u}} : \bigtriangledown \textbf{v}_h  +\int_{\Omega} \tilde{\bigtriangledown}E^{A,n,\theta}_{c} \cdot \bigtriangledown d_h +  \int_{\Omega} \sigma E^{A,n,\theta}_{u1}v_{1h}
+ \int_{\Omega} \sigma E^{A,n,\theta}_{u2}v_{2h} \\ +\int_{\Omega} \alpha E^{A,n,\theta}_{c} d_h 
= \int_{\Omega}(\bigtriangledown \cdot \textbf{v}_h)(E^{I,n,\theta}_{p} + E^{A,n,\theta}_{p})\\
-\int_{\Omega}(\bigtriangledown \cdot  E^{A,n,\theta}_{\textbf{u}}) q_h- \int_{\Omega} \tilde{\bigtriangledown}E^{I,n,\theta}_{c} \cdot \bigtriangledown d_h -\int_{\Omega}d_h \textbf{u}^n \cdot \bigtriangledown (E^{I,n,\theta}_{c}+E^{A,n,\theta}_{c}) \\
 -\int_{\Omega} \mu(c^n)\bigtriangledown E^{I,n,\theta}_{\textbf{u}} : \bigtriangledown \textbf{v}_h-I_3-I_4-I_5-I_6+(TE^{n,\theta}, d_h) \\ \hspace{1mm} \forall \hspace{1mm} \textbf{V}_h \in V_s^h \times V_s^h \times Q_s^h \times V_s^h
\end{multline} 
Now we will treat each term separately to find out the estimate. Before further proceeding let us mention an important consideration: since the above equation holds for all $\textbf{V}_h \in V_s^h \times V_s^h \times Q_s^h \times V_s^h$, therefore in each term we replace $v_{1h},v_{2h},q_h,d_h$ by $E^{A,n,\theta}_{u1},E^{A,n,\theta}_{u2},E^{A,n,\theta}_{p},E^{A,n,\theta}_{c}$ respectively as these auxiliary part of the errors belonging to their respective finite element spaces. From now onwards we will start derivation of each expression after considering the replacements directly.\vspace{2mm}\\
Let us start with $I_1$ as follows:
\begin{equation}
\begin{split}
(\frac{E^{A,n+1}_{c}-E^{A,n}_{c}}{dt},E^{A,n,\theta}_{c}) & = (\frac{E^{A,n+1}_{c}-E^{A,n}_{c}}{dt},\frac{1+\theta}{2} E^{A,n+1}_{c}+ \frac{1-\theta}{2} E^{A,n}_{c}) \\
& =\frac{1+\theta}{2 dt} \|E^{A,n+1}_{c}\|^2 - \frac{1-\theta}{2 dt} \|E^{A,n}_{c}\|^2 -\frac{\theta}{ dt} (E^{A,n}_{c},E^{A,n+1}_{c})\\
&= \frac{1}{2 dt}(\|E^{A,n+1}_{c}\|^2-\|E^{A,n}_{c}\|^2)+ \frac{\theta}{2 dt}(\|E^{A,n+1}_{c}\|^2-\|E^{A,n}_{c}\|^2)^2\\
& \geq \frac{1}{2 dt}(\|E^{A,n+1}_{c}\|^2-\|E^{A,n}_{c}\|^2)
\end{split}
\end{equation}
Let us estimate the remaining terms of $LHS$ as follows:
\begin{equation}
\begin{split}
\int_{\Omega} \mu(c^n)\bigtriangledown E^{A,n,\theta}_{\textbf{u}} : \bigtriangledown  E^{A,n,\theta}_{\textbf{u}} & = \int_{\Omega}\mu(c^n)\{\sum_{i=1}^{2}(\frac{\partial E^{A,n,\theta}_{ui}}{\partial x})^2 + \sum_{i=1}^{2}(\frac{\partial E^{A,n,\theta}_{ui}}{\partial y})^2\}\\
& \geq \mu_l \{ \|\frac{\partial E^{A,n,\theta}_{u1}}{\partial x}\|^2 + \|\frac{\partial E^{A,n,\theta}_{u1}}{\partial y}\|^2 + \|\frac{\partial E^{A,n,\theta}_{u2}}{\partial x}\|^2 + \|\frac{\partial E^{A,n,\theta}_{u2}}{\partial y}\|^2 \}
\end{split}
\end{equation}
and
\begin{equation}
\begin{split}
 \int_{\Omega} \tilde{\bigtriangledown} E^{A,n,\theta}_{c} \cdot \bigtriangledown E^{A,n,\theta}_{c} & = \int_{\Omega}D_1 (\frac{\partial E^{A,n,\theta}_{c}}{\partial x})^2 + \int_{\Omega}D_2 (\frac{\partial E^{A,n,\theta}_{c}}{\partial y})^2\\
 & \geq D_l\{ \|\frac{\partial E^{A,n,\theta}_{c}}{\partial x}\|^2 + \| \frac{\partial E^{A,n,\theta}_{c}}{\partial y}\|^2 \}
\end{split}
\end{equation}
where $D_l$= min $\{ \underset{\Omega}{inf} D_1, \underset{\Omega}{inf} D_2  \}$. \\
Another few terms of $I_2$ can be easily simplified as,
\begin{equation}
\begin{split}
\int_{\Omega} \sigma E^{A,n,\theta}_{ui} E^{A,n,\theta}_{ui} &= \sigma \|E^{A,n,\theta}_{ui}\|^2 \hspace{2mm} for \hspace{1mm} i=1,2\\ 
\int_{\Omega} \alpha E^{A,n,\theta}_{c} E^{A,n,\theta}_{c} & = \alpha \|E^{A,n,\theta}_{c}\|^2 \\
\int_{\Omega}(\bigtriangledown \cdot E^{A,n,\theta}_\textbf{u})(E^{I,n,\theta}_{p} + E^{A,n,\theta}_{p})-\int_{\Omega}(\bigtriangledown \cdot  E^{A,n,\theta}_{\textbf{u}}) E^{A,n,\theta}_{p}  &=  \int_{\Omega}(\bigtriangledown \cdot E^{A,n,\theta}_\textbf{u})E^{I,n,\theta}_{p}
\end{split}
\end{equation}
Combining all these inequalities (32) becomes,
\begin{multline}
 \frac{1}{2 dt}(\|E^{A,n+1}_{c}\|^2-\|E^{A,n}_{c}\|^2)+ \mu_l \{ \|\frac{\partial E^{A,n,\theta}_{u1}}{\partial x}\|^2 + \|\frac{\partial E^{A,n,\theta}_{u1}}{\partial y}\|^2 + \|\frac{\partial E^{A,n,\theta}_{u2}}{\partial x}\|^2 \\
 + \|\frac{\partial E^{A,n,\theta}_{u2}}{\partial y}\|^2 \}+D_l\{ \|\frac{\partial E^{A,n,\theta}_{c}}{\partial x}\|^2 + \| \frac{\partial E^{A,n,\theta}_{c}}{\partial y}\|^2 \} + \sigma \|E^{A,n,\theta}_{u1}\|^2 + \\
 \sigma \|E^{A,n,\theta}_{u2}\|^2 +  \alpha \|E^{A,n,\theta}_{c}\|^2 \\
\leq \int_{\Omega}(\bigtriangledown \cdot E^{A,n,\theta}_\textbf{u})E^{I,n,\theta}_{p}- \int_{\Omega} \tilde{\bigtriangledown}E^{I,n,\theta}_{c} \cdot \bigtriangledown E^{A,n,\theta}_c -\int_{\Omega}E^{A,n,\theta}_c \textbf{u} \cdot \bigtriangledown E^{A,n,\theta}_{c}\\ -\int_{\Omega}E^{A,n,\theta}_c \textbf{u} \cdot \bigtriangledown E^{I,n,\theta}_{c} 
-\int_{\Omega} \mu(c^n)\bigtriangledown E^{I,n,\theta}_{\textbf{u}} : \bigtriangledown E^{A,n,\theta}_\textbf{u}-I_3-I_4-I_5-I_6\\
+(TE^{n,\theta},E^{A,n,\theta}_c) \hspace{80mm}
\end{multline}
Now we will find upper bounds of the terms in the $RHS$ of the above equation.We will use $Cauchy-Schwarz$ and $Young's$ inequality to reach at the desired bounds.Let us start with the first term as follows:\vspace{1mm}\\
\begin{equation}
\begin{split}
\int_{\Omega}(\bigtriangledown \cdot E^{A,n,\theta}_\textbf{u})E^{I,n,\theta}_{p} &= \int_{\Omega} (\frac{\partial E^{A,n,\theta}_{u1}}{\partial x}+\frac{\partial E^{A,n,\theta}_{u2}}{\partial x})E^{I,n,\theta}_{p}\\
& \quad (applying \hspace{1mm} Cauchy-Schwarz \hspace{1mm} inequality) \\
& \leq (\|\frac{\partial E^{A,n,\theta}_{u1}}{\partial x}\| + \|\frac{\partial E^{A,n,\theta}_{u2}}{\partial y}\|) \|E^{I,n,\theta}_{p}\| \\
& \quad (applying \hspace{1mm} Young's \hspace{1mm} inequality \hspace{1mm} for \hspace{1mm} each \hspace{1mm} of \hspace{1mm} the \hspace{1mm} two \hspace{1mm} terms)\\
& \leq \frac{1}{2 \epsilon_1}(\|\frac{\partial E^{A,n,\theta}_{u1}}{\partial x}\|^2 + \|\frac{\partial E^{A,n,\theta}_{u2}}{\partial y}\|^2) + \epsilon_1 \|E^{I,n,\theta}_{p}\|^2 \\
& \leq \frac{1}{2 \epsilon_1}(\|\frac{\partial E^{A,n,\theta}_{u1}}{\partial x}\|^2 + \|\frac{\partial E^{A,n,\theta}_{u2}}{\partial y}\|^2) + \epsilon_1 (\frac{1+\theta}{2}\|E^{I,n+1}_{p}\| + \frac{1-\theta}{2}\|E^{I,n}_{p}\|)^2 \\
& \leq \frac{1}{2 \epsilon_1}(\|\frac{\partial E^{A,n,\theta}_{u1}}{\partial x}\|^2 + \|\frac{\partial E^{A,n,\theta}_{u2}}{\partial y}\|^2) + \epsilon_1 C^2 h^2(\frac{1+\theta}{2}\| p^{n+1}\|_1 + \frac{1-\theta}{2}\| p^n \|_1)^2 \\
\end{split}
\end{equation}
Similarly for each term we will use $Cauchy-Schwarz $ inequality and $ Young's $ inequality  wherever it will be needed, but without mentioning about them now onwards. \\
Proceeding in the same way the second term becomes
\begin{equation}
\begin{split}
- \int_{\Omega} \tilde{\bigtriangledown}E^{I,n,\theta}_{c} \cdot \bigtriangledown E^{A,n,\theta}_c & = -\int_{\Omega}(D_1 \frac{\partial E^{I,n,\theta}_{c}}{\partial x} \frac{\partial E^{A,n,\theta}_{c}}{\partial x}+ D_2 \frac{\partial E^{I,n,\theta}_{c}}{\partial y} \frac{\partial E^{A,n,\theta}_{c}}{\partial y} )\\
& \leq  \frac{D_m}{2 \epsilon_2}(\|\frac{\partial E^{A,n,\theta}_{c}}{\partial x}\|^2 + \|\frac{\partial E^{A,n,\theta}_{c}}{\partial y}\|^2)+  \frac{D_m \epsilon_2}{2} \mid E^{I,n,\theta}_c \mid_1^2\\
& \leq \frac{D_m}{2 \epsilon_2}(\|\frac{\partial E^{A,n,\theta}_{c}}{\partial x}\|^2 + \|\frac{\partial E^{A,n,\theta}_{c}}{\partial y}\|^2) +  \\
& \quad \frac{D_m \epsilon_2}{2} C^2 h^2 (\frac{1+\theta}{2}\| c^{n+1} \|_2+\frac{1-\theta}{2}\| c^n \|_2)^2
\end{split}
\end{equation}
where $D_m$= max $\{ \underset{\Omega}{sup} D_1, \underset{\Omega}{sup} D_2 \}$ \\
Next term,
\begin{equation}
\begin{split}
-\int_{\Omega}E^{A,n,\theta}_c \textbf{u}^n \cdot \bigtriangledown E^{A,n,\theta}_{c} & = -\int_\Omega(u_1^n E^{A,n,\theta}_c \frac{\partial E^{A,n,\theta}_c}{\partial x} + u_2^n E^{A,n,\theta}_c \frac{\partial E^{A,n,\theta}_c}{\partial y})\\
& \leq \frac{1}{2 \epsilon_3}(C_1^n\|\frac{\partial E^{A,n,\theta}_{c}}{\partial x}\|^2 + C_2^n\|\frac{\partial E^{A,n,\theta}_{c}}{\partial y}\|^2)+  \frac{\epsilon_3}{2}(C_1^n+C_2^n) \\
& \quad \| E^{A,n,\theta}_c\|^2
\end{split}
\end{equation}
where $C_1^n$= $\underset{\Omega}{sup}$ $\mid u_1^n \mid$ and $C_2^n$= $\underset{\Omega}{sup}$ $\mid u_2^n \mid$ (applying assumption $\textbf{(v)}$) \\
Similarly the next term 
\begin{equation}
\begin{split}
-\int_{\Omega}E^{A,n,\theta}_c \textbf{u}^n \cdot \bigtriangledown E^{I,n,\theta}_{c} & \leq \frac{1}{2 \epsilon_3}(C_1^n\|\frac{\partial E^{I,n,\theta}_{c}}{\partial x}\|^2 + C_2^n\|\frac{\partial E^{I,n,\theta}_{c}}{\partial y}\|^2)+  \frac{\epsilon_3}{2}(C_1^n+C_2^n) \\
& \quad \| E^{A,n,\theta}_c\|^2 \\
& \leq \frac{C_1^n+C_2^n}{2 \epsilon_3} \| E^{I,n,\theta}_c \|_1^2 + \frac{\epsilon_3}{2}(C_1^n+C_2^n) \| E^{A,n,\theta}_c\|^2 \\
& \leq \frac{C_1^n+C_2^n}{2 \epsilon_3} C^2 h^2 (\frac{1+\theta}{2}\| c^{n+1} \|_2 + \frac{1-\theta}{2} \| c^n \|_2)^2+ \\
& \quad \frac{\epsilon_3}{2}(C_1^n+C_2^n) \| E^{A,n,\theta}_c\|^2 
\end{split}
\end{equation}
The next term,
\begin{equation}
\begin{split}
-\int_{\Omega} \mu(c^n)\bigtriangledown E^{I,n,\theta}_{\textbf{u}} : \bigtriangledown E^{A,n,\theta}_\textbf{u} 
& \leq \mu_u(\|\frac{\partial E^{I,n,\theta}_{u1}}{\partial x}\| \|\frac{\partial E^{A,n,\theta}_{u1}}{\partial x}\|+\|\frac{\partial E^{I,n,\theta}_{u1}}{\partial y}\| \|\frac{\partial E^{A,n,\theta}_{u1}}{\partial y}\| \\
& \quad + \|\frac{\partial E^{I,n,\theta}_{u2}}{\partial x}\| \|\frac{\partial E^{A,n,\theta}_{u2}}{\partial x}\|+ \|\frac{\partial E^{I,n,\theta}_{u2}}{\partial y}\| \|\frac{\partial E^{A,n,\theta}_{u2}}{\partial y}\|) \\
\end{split}
\end{equation}
\begin{equation}
\begin{split}
& \leq \frac{\epsilon_4 \mu_u}{2} \sum_{i=1}^{2} (\|\frac{\partial E^{I,n,\theta}_{ui}}{\partial x} \|^2+ \|\frac{\partial E^{I,n,\theta}_{ui}}{\partial y} \|^2)+\\
& \quad \frac{\mu_u}{2 \epsilon_4} \sum_{i=1}^{2} (\|\frac{\partial E^{A,n,\theta}_{ui}}{\partial x} \|^2+ \|\frac{\partial E^{A,n,\theta}_{ui}}{\partial y} \|^2)\\
& \leq \epsilon_4 \mu_u \sum_{i=1}^{2} C^2 h^2(\frac{1+\theta}{2}\| u_i^{n+1}\|_2+ \frac{1-\theta}{2} \| u_i^n\|_2)^2+\\
& \quad \frac{\mu_u}{2 \epsilon_4} \sum_{i=1}^{2} (\|\frac{\partial E^{A,n,\theta}_{ui}}{\partial x} \|^2+ \|\frac{\partial E^{A,n,\theta}_{ui}}{\partial y} \|^2)
\end{split}
\end{equation}
Now we will find bounds for each remaining term of $I_3$. Before going to further calculations let us mention an important observation:\vspace{2mm}\\
\begin{observation}
According to the choice of the finite element spaces $V_s^h$ and $Q_s^h$, we can clearly say that over each element sub-domain every function belonging to that spaces and their first and second order derivatives all are bounded functions. We can always find positive finite real numbers to bound each of the functions over element sub-domain. We will use this fact for several times further. 
\end{observation}
Let us take the first term of $(-I_3)$ along with earlier mentioned replacements. $I_3$ has four terms and we will find bounds for each of them separately. We have already denoted them by the notations $I_3^1, I_3^2, I_3^3, I_3^4$. Here we start with $I_3^1$,

\begin{multline}
-I_3^1= -\sum_{k=1}^{n_{el}} (\tau_1'( - \mu(c^n) \Delta (E^{I,n,\theta}_{u1}+E_{u1}^{A,n,\theta}) + \sigma (E^{I,n,\theta}_{u1}+E_{u1}^{A,n,\theta}) + \frac{\partial( E^{I,n,\theta}_{p})}{\partial x} + \\
\quad  \frac{\partial( E^{A,n,\theta}_{p})}{\partial x} ),\mu(c^n) \Delta E_{u1}^{A,n,\theta} - \sigma E_{u1}^{A,n,\theta} + \frac{\partial E_{p}^{A,n,\theta}}{\partial x})_{\Omega_k}\\
 =\tau_1' (\mu(c^n) \Delta E^{I,n,\theta}_{u1} - \sigma E^{I,n,\theta}_{u1}- \frac{\partial( E^{I,n,\theta}_{p})}{\partial x} ,\mu(c^n) \Delta E_{u1}^{A,n,\theta} - \sigma E_{u1}^{A,n,\theta} + \frac{\partial E_{p}^{A,n,\theta}}{\partial x})_{\tilde{\Omega}} \\
 \quad + \tau_1' (\mu(c^n) \Delta E^{A,n,\theta}_{u1} - \sigma E^{A,n,\theta}_{u1}- \frac{\partial( E^{A,n,\theta}_{p})}{\partial x},\mu(c^n) \Delta E_{u1}^{A,n,\theta} - \sigma E_{u1}^{A,n,\theta} + \frac{\partial E_{p}^{A,n,\theta}}{\partial x})_{\tilde{\Omega}} \\
\end{multline}
We calculate the bounds for the above two terms separately. Applying $Cauchy-Schwarz$ inequality on each term the first part is as follows:
\begin{multline}
\tau_1' (\mu(c^n) \Delta E^{I,n,\theta}_{u1} - \sigma E^{I,n,\theta}_{u1}- \frac{\partial( E^{I,n,\theta}_{p})}{\partial x} ,\mu(c^n) \Delta E_{u1}^{A,n,\theta} - \sigma E_{u1}^{A,n,\theta} + \frac{\partial E_{p}^{A,n,\theta}}{\partial x})_{\tilde{\Omega}} \\
 \end{multline}
\begin{multline}
 \leq \sum_{k=1}^{n_{el}} \mid \tau_1 \mid (\mu_u^2 \|\Delta E^{I,n,\theta}_{u1}\|_k \|\Delta E^{A,n,\theta}_{u1}\|_k + \sigma \mu_u \|E^{I,n,\theta}_{u1}\|_k \|\Delta E^{A,n,\theta}_{u1}\|_k + 
 \mu_u\\
  \|\frac{\partial E_{p}^{I,n,\theta}}{\partial x}\|_k \| \Delta E_{u1}^{A,n,\theta}\|_k + \sigma \mu_u \| E^{A,n,\theta}_{u1}\|_k \| \Delta E_{u1}^{I,n,\theta}\|_k+ \sigma^2 \|  E^{A,n,\theta}_{u1}\|_k \\
  \| E_{u1}^{I,n,\theta}\|_k + \sigma \| E^{A,n,\theta}_{u1}\|_k \|\frac{\partial E_{p}^{I,n,\theta}}{\partial x} \|_k + \mu_u \|\Delta E^{I,n,\theta}_{u1}\|_k \|\frac{\partial E_{p}^{A,n,\theta}}{\partial x}\|_k + \\
\sigma \|E^{I,n,\theta}_{u1}\|_k \|\frac{\partial E_{p}^{A,n,\theta}}{\partial x}\|_k + \|\frac{\partial E_{p}^{A,n,\theta}}{\partial x}\|_k \|\frac{\partial E_{p}^{I,n,\theta}}{\partial x}\|_k) \\
 \end{multline}
Let $B_{1k},B_{2k},B_{3k}$ be the bounds on $E^{A,n,\theta}_{u1}, \Delta E^{A,n,\theta}_{u1}, \frac{\partial E^{A,n,\theta}_{p}}{\partial x}$ respectively on each element sub domain under the above observation 1.

 \begin{multline}
\leq \mid \tau_1 \mid \sum_{k=1}^{n_{el}} (\mu_u^2  B_{2k} \|\Delta E^{I,n,\theta}_{u1}\|_k  + \sigma \mu_u B_{2k} \|E^{I,n,\theta}_{u1}\|_k  + \mu_u B_{2k}
\mid E_{p}^{I,n,\theta} \mid_{1,k} + \\
 \sigma \mu_u B_{1k} \| \Delta E_{u1}^{I,n,\theta}\|_k+ \sigma^2 B_{1k} \| E_{u1}^{I,n,\theta}\|_k + \sigma B_{1k} \mid E_{p}^{I,n,\theta}\mid_{1,k} + \\
 \mu_u  B_{3k} \|\Delta E^{I,n,\theta}_{u1}\|_k + \sigma B_{3k} \|E^{I,n,\theta}_{u1}\|_k + B_{3k} \mid E_{p}^{I,n,\theta} \mid_{1,k})\\
\leq \mid \tau_1 \mid  \{ (\sum_{k=1}^{n_{el}}(\mu_u^2 B_{2k} + \sigma \mu_u B_{1k} +\mu_u B_{3k} ))C (\frac{1+\theta}{2} \| u_1^{n+1} \|_2 + \frac{1-\theta}{2}\| u_1^n \|_2) + \\
(\sum_{k=1}^{n_{el}}\sigma (\mu_u B_{2k}+\sigma B_{1k}+B_{3k})) C h^2 (\frac{1+\theta}{2} \| u_1^{n+1} \|_2 + \frac{1-\theta}{2}\| u_1^n \|_2)+ \\
(\sum_{k=1}^{n_{el}}(\mu_u B_{2k}+ \sigma B_{1k}+ B_{3k})) C h (\frac{1+\theta}{2} \| p^{n+1} \|_1 + \frac{1-\theta}{2}\| p^n \|_1) \}
\end{multline}
This completes the first part. Now we see that the second part has alike expression with auxiliary error terms in the place of interpolation error terms. Hence proceeding in the same way as above and applying bounds for elements belonging to $V_s^h$ and $Q_s^h$ spaces we will bound the second part as follows:

\begin{multline}
\tau_1' (\mu(c) \Delta E^{A,n,\theta}_{u1} - \sigma E^{A,n,\theta}_{u1}- \frac{\partial( E^{A,n,\theta}_{p})}{\partial x},\mu(c) \Delta E_{u1}^{A,n,\theta} - \sigma E_{u1}^{A,n,\theta} + \frac{\partial E_{p}^{A,n,\theta}}{\partial x})_{\tilde{\Omega}} \\
\leq \mid \tau_1 \mid \sum_{k=1}^{n_{el}} (\mu_u^2 B_{2k}^2 + 2 \sigma \mu_u B_{1k} B_{2k} + \mu_u B_{2k} B_{3k} + \sigma^2 B_{1k}^2 + \sigma B_{1k} B_{3k} +\\
\mu_u B_{3k} B_{2k} + \sigma B_{1k} B_{3k} +B_{3k}^2 )\\
\leq \mid \tau_1 \mid \sum_{k=1}^{n_{el}} M_{1k} \hspace{80mm}
\end{multline}
where $M_{1k}=(\mu_u^2 B_{2k}^2 + 2 \sigma \mu_u B_{1k} B_{2k} + \mu_u B_{2k} B_{3k} + \sigma^2 B_{1k}^2 + \sigma B_{1k} B_{3k} +\mu_u B_{3k} B_{2k} + \sigma B_{1k} B_{3k} +B_{3k}^2$)\\
Combining all these results and putting into (42) we will have
\begin{equation}
\begin{split}
-I_3^1 & \leq  \mid \tau_1 \mid \{ (\sum_{k=1}^{n_{el}}(\mu_u^2 B_{2k} + \sigma \mu_u B_{1k} +\mu_u B_{3k} ))C (\frac{1+\theta}{2} \| u_1^{n+1} \|_2 +  \\
& \quad \frac{1-\theta}{2}\| u_1^n \|_2) + (\sum_{k=1}^{n_{el}}\sigma (\mu_u B_{2k}+\sigma B_{1k}+B_{3k})) C h^2 (\frac{1+\theta}{2} \| u_1^{n+1} \|_2 +  \\
& \quad \frac{1-\theta}{2}\| u_1^n \|_2)+ (\sum_{k=1}^{n_{el}}(\mu_u B_{2k}+ \sigma B_{1k}+ B_{3k})) C h (\frac{1+\theta}{2} \| p^{n+1} \|_1 +  \\
& \quad \frac{1-\theta}{2}\| p^n \|_1) \}
+ \mid \tau_1 \mid \sum_{k=1}^{n_{el}} M_{1k}
\end{split}
\end{equation} 
This completes the derivation of bound on the first term of $(-I_3)$. Now we see that the second term of $I_3$ in (27) is exactly similar to its first term, only the subscripts  are different that is $u_2$ replaces $u_1$ in subscript. Therefore considering the constants $B_{1k}',B_{2k}',B_{3k}'$ as the bounds for $E^{A,n,\theta}_{u2}, \Delta E^{A,n,\theta}_{u2},\frac{\partial E_{p}^{A,n,\theta}}{\partial y} $ respectively on each element sub domain, we can bound the term as follows:
\begin{equation}
\begin{split}
-I_3^2 & \leq  \mid \tau_1 \mid \mid \{ (\sum_{k=1}^{n_{el}}(\mu_u^2 B_{2k}' + \sigma \mu_u B_{1k}' +\mu_u B_{3k}' ))C (\frac{1+\theta}{2} \| u_2^{n+1} \|_2 +  \\
& \quad \frac{1-\theta}{2}\| u_2^n \|_2) + (\sum_{k=1}^{n_{el}}\sigma (\mu_u B_{2k}'+\sigma B_{1k}'+B_{3k}')) C h^2 (\frac{1+\theta}{2} \| u_2^{n+1} \|_2 +  \\
& \quad \frac{1-\theta}{2}\| u_2^n \|_2)+ (\sum_{k=1}^{n_{el}}(\mu_u B_{2k}'+ \sigma^2 B_{1k}'+ B_{3k}')) C h (\frac{1+\theta}{2} \| p^{n+1} \|_1 +  \\
& \quad \frac{1-\theta}{2}\| p^n \|_1) \}
+ \mid \tau_1 \mid \sum_{k=1}^{n_{el}} M_{2k}
\end{split}
\end{equation}
where $M_{2k}=(\mu_u^2 B_{2k}^{'2} + 2 \sigma \mu_u B'_{1k} B'_{2k} + \mu_u B'_{2k} B'_{3k} + \sigma^2 B_{1k}^{'2} + \sigma B'_{1k} B'_{3k} +\mu_u B'_{3k} B'_{2k} + \sigma B'_{1k} B'_{3k} +B_{3k}^{'2}) $ \\
Now we are going to derive bounds for the third term of $I_3$ as follows:
\begin{equation}
\begin{split}
-I_3^3 & = -\tau_2' \sum_{k=1}^{n_{el}}(\bigtriangledown \cdot E^{I,n,\theta}_{\textbf{u}}, \bigtriangledown \cdot E_{\textbf{u}}^{A,n,\theta} )_{\Omega_k} - \tau_2' \sum_{k=1}^{n_{el}} (\bigtriangledown \cdot E^{A,n,\theta}_{\textbf{u}}, \bigtriangledown \cdot E_{\textbf{u}}^{A,n,\theta} )_{\Omega_k} \\
& \leq \mid \tau_2 \mid \sum_{k=1}^{n_{el}} (\|\frac{\partial E^{I,n,\theta}_{u1}}{\partial x}\|_k \|\frac{\partial E^{A,n,\theta}_{u1}}{\partial x}\|_k+ \|\frac{\partial E^{I,n,\theta}_{u2}}{\partial y}\|_k \|\frac{\partial E^{A,n,\theta}_{u2}}{\partial y}\|_k + \|\frac{\partial E^{A,n,\theta}_{u1}}{\partial x}\|_k^2 + \\
& \quad \|\frac{\partial E^{I,n,\theta}_{u1}}{\partial x}\|_k \|\frac{\partial E^{A,n,\theta}_{u2}}{\partial y} \|_k+ \|\frac{\partial E^{A,n,\theta}_{u1}}{\partial x} \|_k \|\frac{\partial E^{I,n,\theta}_{u2}}{\partial y}\|_k + 2\|\frac{\partial E^{A,n,\theta}_{u1}}{\partial x} \|_k \|\frac{\partial E^{A,n,\theta}_{u2}}{\partial y} \|_k + \\
\end{split}
 \end{equation}
\begin{equation}
\begin{split}
& \quad \|\frac{\partial E^{A,n,\theta}_{u2}}{\partial y}\|_k^2)\\
& \leq \mid \tau_2 \mid \sum_{k=1}^{n_{el}} ((B_{4k}+ B_{5k}')(\|\frac{\partial E^{I,n,\theta}_{u1}}{\partial x}\|_k+ \|\frac{\partial E^{I,n,\theta}_{u2}}{\partial y}\|_k)+ \mid \tau_2 \mid C_1 \{ (1+ \epsilon_5) \| \frac{\partial E^{A,n,\theta}_{u1}}{\partial x}\|^2 \\
& \quad + (1+ \epsilon_5) \| \frac{\partial E^{A,n,\theta}_{u2}}{\partial y}\|^2 \} \\
& \leq  \mid \tau_2 \mid ( \sum_{k=1}^{n_{el}} (B_{4k}+B_{5k}')) (\| E^{I,n,\theta}_{u1} \|_1 + \| E^{I,n,\theta}_{u2} \|_1) + C_{\tau_2} C_1 \{ (1+ \epsilon_5) \| \frac{\partial E^{A,n,\theta}_{u1}}{\partial x}\|^2+ \\
& \quad  (1+ \epsilon_5) \| \frac{\partial E^{A,n,\theta}_{u2}}{\partial y}\|^2 \} \\
& \leq  \mid \tau_2 \mid ( \sum_{k=1}^{n_{el}}( B_{4k}+B_{5k}'))C h \{(\frac{1+\theta}{2} \| u_1^{n+1} \|_2 + \frac{1-\theta}{2} \| u_1^{n} \|_2 )+ (\frac{1+\theta}{2} \| u_2^{n+1} \|_2 +  \\
& \quad \frac{1-\theta}{2} \| u_2^{n} \|_2 )\} + C_{\tau_2} C_1 \{ (1+ \epsilon_5) \| \frac{\partial E^{A,n,\theta}_{u1}}{\partial x}\|^2+ (1+ \epsilon_5) \| \frac{\partial E^{A,n,\theta}_{u2}}{\partial y}\|^2 \}
\end{split}
\end{equation}
where the constants $B_{4k},B_{5k},B_{4k}'$ and $B_{5k}'$ are bounds on $ \frac{\partial E^{A,n,\theta}_{u1}}{\partial x} , \frac{\partial E^{A,n,\theta}_{u1}}{\partial y}, \frac{\partial E^{A,n,\theta}_{u2}}{\partial x} $ and $\frac{\partial E^{A,n,\theta}_{u2}}{\partial y}$ respectively on each element sub domain and $C_{\tau_2}$ is the maximum numerical value for $\tau_2$ over $\Omega$. Now we will focus on the fourth term of $I_3$. We will divide $I_3^4$ into three parts $P_1,P_2$ and $P_3$ and then calculate bounds for each of them separately. 

\begin{equation}
\begin{split}
-I_3^4 &= \sum_{k=1}^{n_{el}} \tau_3'( \partial_t(E^{I,n}_{c}+E_{c}^{A,n}),-\bigtriangledown \cdot \tilde{\bigtriangledown} E_{c}^{A,n,\theta} -\textbf{u} \cdot \bigtriangledown E_{c}^{A,n,\theta} + \alpha E_{c}^{A,n,\theta} )_{\Omega_k} + \sum_{k=1}^{n_{el}} \tau_3' \\
& \quad  (\bigtriangledown \cdot \tilde{\bigtriangledown}E^{I,n,\theta}_{c} - \textbf{u} \cdot \bigtriangledown E^{I,n,\theta}_{c}- \alpha E^{I,n,\theta}_{c},\bigtriangledown \cdot \tilde{\bigtriangledown} E_{c}^{A,n,\theta} +\textbf{u} \cdot \bigtriangledown E_{c}^{A,n,\theta} -  \alpha E_{c}^{A,n,\theta} )_{\Omega_k}  \\
& \quad  + \sum_{k=1}^{n_{el}} \tau_3' (\bigtriangledown \cdot \tilde{\bigtriangledown}E^{A,n,\theta}_{c} - \textbf{u} \cdot \bigtriangledown E^{A,n,\theta}_{c}- \alpha E^{A,n,\theta}_{c},\bigtriangledown \cdot \tilde{\bigtriangledown} E_{c}^{A,n,\theta}+ \textbf{u} \cdot \bigtriangledown E_{c}^{A,n,\theta} -  \\
& \quad \alpha E_{c}^{A,n,\theta} )_{\Omega_k}\\
& = P_1+ P_2 + P_3
\end{split}
\end{equation}
Let us start with $P_1$
\begin{equation}
\begin{split}
P_1 & = \sum_{k=1}^{n_{el}} \tau_3'\alpha( \partial_t E^{I,n}_{c}+ \partial_t E_{c}^{A,n}, E_{c}^{A,n,\theta})_{\Omega_k}- \sum_{k=1}^{n_{el}} \tau_3'(\partial_t E^{I,n}_{c}+ \partial_t E_{c}^{A,n},\\
& \quad \bigtriangledown \cdot \tilde{\bigtriangledown} E_{c}^{A,n,\theta} +\textbf{u} \cdot \bigtriangledown E_{c}^{A,n,\theta})_{\Omega_k}\\
& = \alpha\tau_3' \sum_{k=1}^{n_{el}} \int_{\Omega_k} \frac{E^{A,n+1}_{c}-E^{A,n}_{c}}{dt}E_{c}^{A,n,\theta}- \tau_3' \sum_{k=1}^{n_{el}} \int_{\Omega_k} \frac{E^{I,n+1}_{c}-E^{I,n}_{c}}{dt} (D_1 \frac{\partial^2 E_{c}^{A,n,\theta}}{\partial x^2}+\\
\end{split}
 \end{equation}
\begin{equation}
\begin{split}
& \quad D_2 \frac{\partial^2 E_{c}^{A,n,\theta}}{\partial y^2} + (u_1+ \frac{\partial D_1}{\partial x})\frac{\partial E_{c}^{A,n,\theta} }{\partial x} + (u_2+ \frac{\partial D_2}{\partial y})\frac{\partial E_{c}^{A,n,\theta} }{\partial y})-\\
& \quad \tau_3' \sum_{k=1}^{n_{el}} \int_{\Omega_k} \frac{E^{A,n+1}_{u1}-E^{A,n}_{u1}}{dt} (D_1 \frac{\partial^2 E_{c}^{A,n,\theta}}{\partial x^2}+ D_2 \frac{\partial^2 E_{c}^{A,n,\theta}}{\partial y^2} + (u_1+ \frac{\partial D_1}{\partial x})\frac{\partial E_{c}^{A,n,\theta} }{\partial x} + \\
& \quad (u_2+ \frac{\partial D_2}{\partial y})\frac{\partial E_{c}^{A,n,\theta} }{\partial y}) \\
& \leq \frac{\alpha \mid  \tau_3'\mid }{dt} (\sum_{k=1}^{n_{el}} B_{6k}) (\|E^{A,n+1}_{c}\|^2-\|E^{A,n}_{c}\|^2)+\frac{\mid  \tau_3'\mid}{dt} \{\sum_{k=1}^{n_{el}} (D_{1m} B_{7k}+D_{2m}B_{7k}'+D_{u1} \\
& \quad B_{8k}+ D_{u2} B_{8k}')\}(\|E^{I,n+1}_{c}\|+\|E^{I,n}_{c}\|) + \frac{\mid  \tau_3'\mid}{dt} \{\sum_{k=1}^{n_{el}} (D_{1m} B_{7k}+D_{2m}B_{7k}'+D_{u1}B_{8k}+\\
& \quad D_{u2} B_{8k}')\}(\|E^{A,n+1}_{c}\|^2-\|E^{A,n}_{c}\|^2)\\
& \leq \frac{C_{\tau_3} T}{T_0(T_0-C_{\tau_3})}\{ \sum_{k=1}^{n_{el}}(\alpha B_{6k} +D_{B_{1k}})\}(\|E^{A,n+1}_{c}\|^2-\|E^{A,n}_{c}\|^2)+\frac{C_{\tau_3}Ch^2}{(T_0-C_{\tau_3})} \{\sum_{k=1}^{n_{el}} D_{B_{1k}} \} \\
& \quad (\| c^{n+1} \|_2 + \| c^n \|_2)
\end{split}
\end{equation}
where the constants $B_{6k},B_{7k},B_{7k}',B_{8k}$ and $B_{8k}'$ are upper bounds on $E^{A,n,\theta}_{c},\frac{\partial^2 E^{A,n,\theta}_{c}}{\partial x^2}, \frac{\partial^2 E^{A,n,\theta}_{c}}{\partial y^2} ,\\ \frac{\partial E_{c}^{A,n,\theta}}{\partial x}$ and $\frac{\partial E_{c}^{A,n,\theta}}{\partial y}$ respectively on each element sub domain and $D_{1m},D_{2m}, {D}_{u1}, {D}_{u2}$ are maximum of the functions $D_1, D_2, (\frac{\partial D_1}{\partial x}+ u_1),(\frac{\partial D_2}{\partial y}+ u_2) $ respectively over $\Omega$. $C_{\tau_3}$ and $T_0$ are maximum bound of $\tau_3$ and minimum bound of time step $dt$ respectively. At the last line new notation $D_{B_{1k}}$ represents the big sum.

\begin{equation}
\begin{split}
P_2 & = \sum_{k=1}^{n_{el}} \tau_3' (D_1 \frac{\partial^2 E^{I,n,\theta}_{c}}{\partial x^2} + D_2 \frac{\partial^2 E^{I,n,\theta}_{c}}{\partial y^2} + (\frac{\partial D_1}{\partial x}-u_1)\frac{\partial E^{I,n,\theta}_{c}}{\partial x}+ (\frac{\partial D_2}{\partial y}-u_2)\\
& \quad \frac{\partial E^{I,n,\theta}_{c}}{\partial y}-\alpha E^{I,n,\theta}_{c}, D_1 \frac{\partial^2 E^{A,n,\theta}_{c}}{\partial x^2} + D_2 \frac{\partial^2 E^{A,n,\theta}_{c}}{\partial y^2} + (\frac{\partial D_1}{\partial x}+u_1)\frac{\partial E^{A,n,\theta}_{c}}{\partial x}+ \\
& \quad (\frac{\partial D_2}{\partial y}+u_2)\frac{\partial E^{A,n,\theta}_{c}}{\partial y}-\alpha E^{A,n,\theta}_{c} )_{\Omega_k}\\
\end{split}
\end{equation}
Further simplifying and applying the bounds on auxiliary error terms over each sub-domain, $P_2$ becomes
\begin{equation}
\begin{split}
& \leq \sum_{k=1}^{n_{el}} \mid \tau_3' \mid(D_{1m}^2 B_{7k} \|\frac{\partial^2 E^{I,n,\theta}_{c}}{\partial x^2}\|_k + D_{1m}D_{2m} B_{7k}' \|\frac{\partial^2 E^{I,n,\theta}_{c}}{\partial y^2}\|_k + D_{1m} \bar{D}_{u1} B_{7k} \|\frac{\partial E^{I,n,\theta}_{c}}{\partial x}\|_k \\
& \quad D_{1m}\bar{D}_{u2} B_{7k} \|\frac{\partial E^{I,n,\theta}_{c}}{\partial x}\|_k + \alpha D_{1m} B_{7k} \|E^{I,n,\theta}_{c}\|_k + D_{1m}D_{2m} B_{7k}' \|\frac{\partial^2 E^{I,n,\theta}_{c}}{\partial x^2}\|_k + D_{2m}^2 B_{7k}' \\
\end{split}
 \end{equation}
\begin{equation}
\begin{split}
& \quad \|\frac{\partial^2 E^{I,n,\theta}_{c}}{\partial y^2}\|_k + D_{2m} \bar{D}_{u1} B_{7k}' \|\frac{\partial^2 E^{I,n,\theta}_{c}}{\partial y^2}\|_k + D_{2m} \bar{D}_{u2} B_{7k}' \|\frac{\partial E^{I,n,\theta}_{c}}{\partial y}\|_k + \alpha D_{2m} B_{7k}' \|E^{I,n,\theta}_{c}\|_k +\\
& \quad D_{1m} D_{u1} B_{8k} \|\frac{\partial^2 E^{I,n,\theta}_{c}}{\partial x^2}\|_k + D_{2m} D_{u1} B_{8k} \|\frac{\partial^2 E^{I,n,\theta}_{c}}{\partial y^2}\|_k + \bar{D}_{u1} \bar{D}_{u2} B_{8k} \|\frac{\partial E^{I,n,\theta}_{c}}{\partial x}\|_k + D_{u1} \bar{D}_{u2} \\
& \quad B_{8k} \|\frac{\partial E^{I,n,\theta}_{c}}{\partial y}\|_k + \alpha D_{u1}  B_{8k} \|E^{I,n,\theta}_{c}\|_k + D_{1m} D_{u2} B_{8k}'  \|\frac{\partial^2 E^{I,n,\theta}_{c}}{\partial x^2}\|_k + D_{2m} D_{u2} B_{8k}'  \|\frac{\partial^2 E^{I,n,\theta}_{c}}{\partial y^2}\|_k \\
& \quad + \bar{D}_{u1} D_{u2} B_{8k}' \|\frac{\partial E^{I,n,\theta}_{c}}{\partial x}\|_k +  D_{u2} \bar{D}_{u2} B_{8k}' \|\frac{\partial E^{I,n,\theta}_{c}}{\partial y}\|_k + \alpha D_{1m} B_{6k} \|\frac{\partial^2 E^{I,n,\theta}_{c}}{\partial x^2}\|_k + \alpha  D_{2m} B_{6k} \\
& \quad \|\frac{\partial^2 E^{I,n,\theta}_{c}}{\partial y^2}\|_k + \alpha \bar{D}_{u1} B_{6k} \|\frac{\partial E^{I,n,\theta}_{c}}{\partial x}\|_k +
\alpha \bar{D}_{u2} B_{6k} \|\frac{\partial E^{I,n,\theta}_{c}}{\partial y}\|_k + \alpha^2 B_{6k} \| E^{I,n,\theta}_{c}\|_k + \alpha D_{u2} B_{8k}'\\
& \quad \|E^{I,n,\theta}_c\|_k) \\
& \leq \mid \tau_3' \mid \{ \sum_{k=1}^{n_{el}} (D_{1m}^2 B_{7k} +2D_{1m}D_{2m} B_{7k}' +D_{2l}^2 B_{7k}' +D_{2m} \bar{D}_{u1} B_{7k}'+ D_{1m} D_{u1} B_{8k}+\\
& \quad D_{2m} D_{u1} B_{8k} +D_{1m} D_{u2} B_{8k}' +D_{2m} D_{u2} B_{8k}' + \alpha D_{1m} B_{6k} + \alpha D_{2m} B_{6k})\} C(\frac{1+\theta}{2}\| c^{n+1}\|_2 + \\
& \quad \frac{1-\theta}{2} \| c^n \|_2) + \mid \tau_3' \mid \{ \sum_{k=1}^{n_{el}} (D_{1m} \bar{D}_{u1} B_{7k}+ D_{1m} \bar{D}_{u2} B_{7k}+D_{2m} \bar{D}_{u2} B_{7k}'+ \bar{D}_{u1} \bar{D}_{u2} B_{8k} +  \\
& \quad D_{u1}\bar{D}_{u2} B_{8k} + D_{u2} \bar{D}_{u1} B_{8k}'+D_{u2} \bar{D}_{u2} B_{8k}' + \alpha  \bar{D}_{u1} B_{6k}+ \alpha  \bar{D}_{u2} B_{6k})\} C h (\frac{1+\theta}{2}\| c^{n+1}\|_1+ \\
& \quad \frac{1-\theta}{2} \| c^n \|_1)+  \mid \tau_3' \mid \{ \sum_{k=1}^{n_{el}} (\alpha D_{1m} B_{7k} + \alpha D_{2m} B_{7k}' + \alpha D_{u1} B_{8k} + \alpha D_{u2} B_{8k}' + \alpha^2 B_{6k})\} C h^2 \\
& \quad (\frac{1+\theta}{2}\| c^{n+1}\|+\frac{1-\theta}{2} \| c^n \|) \\
& \leq \frac{\mid \tau_3 \mid T}{(T_0-C_{\tau_3})} C \{ (\sum_{k=1}^{n_{el}} D_{B_{2k}})(\frac{1+\theta}{2}\| c^{n+1}\|_2 +\frac{1-\theta}{2} \| c^n \|_2)+h (\sum_{k=1}^{n_{el}} D_{B_{3k}})(\frac{1+\theta}{2}\| c^{n+1}\|_1  \\
& \quad +\frac{1-\theta}{2} \| c^n \|_1) + h^2 (\sum_{k=1}^{n_{el}} D_{B_{4k}})(\frac{1+\theta}{2}\| c^{n+1}\| +\frac{1-\theta}{2} \| c^n \|) \}
\end{split}
\end{equation}
where $D_{B_{2k}}, D_{B_{3k}}$ and $D_{B_{4k}}$ are denoting respectively the summations in which the notations $ \bar{D}_{u1}, \bar{D}_{u2}$ are the maximum of the functions $ (\frac{\partial D_1}{\partial x}- u_1),(\frac{\partial D_2}{\partial y}- u_2) $ respectively over $\Omega$.
The next term is similar to the previous one. Therefore the simplification will be same as above. Hence skipping the calculations we directly put the result as follows: 

\begin{equation}
\begin{split}
P_3 & = \sum_{k=1}^{n_{el}} \tau_3' (\bigtriangledown \cdot \tilde{\bigtriangledown}E^{A,n,\theta}_{c} - \textbf{u} \cdot \bigtriangledown E^{A,n,\theta}_{c}- \alpha E^{A,n,\theta}_{c},\bigtriangledown \cdot \tilde{\bigtriangledown} E_{c}^{A,n,\theta} + \textbf{u} \cdot \bigtriangledown E_{c}^{A,n,\theta} \\
& \quad -  \alpha E_{c}^{A,n,\theta} )_{\Omega_k} \\
\end{split}
 \end{equation}
\begin{equation}
\begin{split}
& \leq \mid \tau_3' \mid \{ \sum_{k=1}^{n_{el}}(D_{1m}^2 B_{7k}^2 + 2 D_{1m} D_{2m} B_{7k} B_{7k}' + D_{1m} \bar{D}_{u1} B_{8k} B_{7k} + D_{1m} \bar{D}_{u2} B_{8k}' B_{7k} +  \\
& \quad\alpha D_{1m} B_{7k} B_{6k} + D_{2m}^2 B_{7k}'^2 + D_{2m} \bar{D}_{u1} B_{8k} B_{7k}' + D_{2m} \bar{D}_{u2} B_{8k}' B_{7k}' + \alpha D_{2m} B_{6k} B_{7k}' + \\
& \quad  D_{1m} D_{u1} B_{8k} B_{7k}+ D_{2m} D_{u1} B_{8k} B_{7k}' + D_{u1} \bar{D}_{u1} B_{7k}^2 + D_{u1} \bar{D}_{u2} B_{8k} B_{8k}'+ \alpha D_{u1}B_{6k}  \\
& \quad  B_{8k}+ D_{1m} D_{u2} B_{8k}' B_{7k} + D_{2m} D_{u2} B_{8k}' B_{7k}' + D_{u2} \bar{D}_{u1} B_{8k} B_{8k}' + D_{u2} \bar{D}_{u2} B_{8k}'^2 + \alpha \\
& \quad ( D_{u2}  B_{8k}' B_{6k} + D_{1m} B_{6k} B_{7k} +D_{2m} B_{6k} B_{7k}' + \bar{D}_{u1} B_{8k} B_{6k} +  \bar{D}_{u2} B_{8k}' B_{6k} + \alpha B_{6k}^2)) \} \\
& \leq \frac{\mid \tau_3 \mid T}{(T_0-C_{\tau_3})} \sum_{k=1}^{n_{el}} D_{B_{5k}} 
\end{split}
\end{equation}
where $D_{B_{5k}}$ is a notation denoting the big sum of the constants.\vspace{1mm}\\
Now combining all the bounds obtained for $P_1,P_2,P_3$ and putting them into the expression of $I^4_3$ we will have

\begin{equation}
\begin{split}
-I^4_3 & \leq \frac{C_{\tau_3} T}{dt(T_0-C_{\tau_3})}\{ \sum_{k=1}^{n_{el}}(\alpha B_{6k} + D_{B_{1k}})\}(\|E^{A,n+1}_{c}\|^2-\|E^{A,n}_{c}\|^2)+ \frac{C_{\tau_3}Ch^2 }{(T_0-C_{\tau_3})}   \{\sum_{k=1}^{n_{el}} D_{B_{1k}}\} \\
& \quad (\| c^{n+1} \|_2 +  \| c^n \|_2)+ \frac{\mid \tau_3 \mid T C}{(T_0-C_{\tau_3})}  \{ (\sum_{k=1}^{n_{el}} D_{B_{2k}})(\frac{1+\theta}{2}\| c^{n+1}\|_2 +\frac{1-\theta}{2} \| c^n \|_2)+ h  \\
& \quad  (\sum_{k=1}^{n_{el}} D_{B_{3k}})(\frac{1+\theta}{2}\| c^{n+1}\|_1 +\frac{1-\theta}{2}  \mid c^n \mid_1)+ h^2 (\sum_{k=1}^{n_{el}} D_{B_{4k}})(\frac{1+\theta}{2} \| c^{n+1}\|+ \frac{1-\theta}{2} \| c^n \|) \}  \\
& \quad   + \frac{\mid \tau_3 \mid T}{(T_0-C_{\tau_3})} \sum_{k=1}^{n_{el}} D_{B_{5k}}
\end{split}
\end{equation}
Finally here the process of finding bound for each term of $I_3$ is completed. Now we will focus on finding bounds for the terms of $I_4$. Before going to derivation let us see the term $d_4$ explicitly. \vspace{1mm} \\
\begin{equation}
\begin{split}
d_4 &= \sum_{i=1}^{n+1}(\frac{1}{dt}\tau_3')^i(\partial_t(c^n-c_{h}^n) - \bigtriangledown \cdot \tilde{\bigtriangledown} (c^{n,\theta}-c_{h}^{n,\theta}) + \textbf{u}^n \cdot \bigtriangledown (c^{n,\theta}-c_{h}^{n,\theta}) + \alpha (c^{n,\theta}-c_{h}^{n,\theta})) \\
& \leq \sum_{i=1}^{\infty}(\frac{1}{dt} \tau_3' )^i (\partial_t (E^{I,n}_{c}+ E_{c}^{A,n})- \bigtriangledown \cdot \tilde{\bigtriangledown} (E^{I,n,\theta}_{c}+ E_{c}^{A,n,\theta}) + \textbf{u}^n \cdot \bigtriangledown (E^{I,n,\theta}_{c}+ E_{c}^{A,n,\theta}) \\
& \quad +  \alpha (E^{I,n,\theta}_{c}+ E_{c}^{A,n,\theta})) \\
& = \frac{\tau_3' }{(dt- \tau_3' )}(\partial_t E^{I,n}_{c}+\partial_t E^{A,n}_{c})-\frac{ \tau_3' }{(dt- \tau_3')}(\bigtriangledown \cdot \tilde{\bigtriangledown} E^{I,n,\theta}_{c} - \textbf{u}^n \cdot \bigtriangledown E^{I,n,\theta}_{c}-\alpha E^{I,n,\theta}_{c}) \\
& \quad -\frac{\tau_3' }{(dt- \tau_3' )}(\bigtriangledown \cdot \tilde{\bigtriangledown} E^{A,n,\theta}_{c} - \textbf{u}^n \cdot \bigtriangledown E^{A,n,\theta}_{c}-\alpha E^{A,n,\theta}_{c})
\end{split}
\end{equation} 
Since $\frac{\tau_3}{dt+ \tau_3} < 1$, which implies $\frac{\tau_3'}{dt} < 1$ and therefore the series $\sum_{i=1}^{\infty}(\frac{1}{dt} \tau_3' )^i$ converges to $\frac{\tau_3' }{(dt- \tau_3' )}$ 
\begin{equation}
\begin{split}
-I_4 & = \sum_{k=1}^{n_{el}} \tau_3' (d_4, \bigtriangledown \cdot \tilde{\bigtriangledown} E_{c}^{A,n,\theta} +\textbf{u} \cdot \bigtriangledown E_{c}^{A,n,\theta} - \alpha E_{c}^{A,n,\theta} )_{\Omega_k}\\
& \leq \frac{\tau_3'^2}{(dt- \tau_3' )}\sum_{k=1}^{n_{el}}(\partial_t E^{I,n}_{c}+\partial_t E^{A,n}_{c},\bigtriangledown \cdot \tilde{\bigtriangledown} E_{c}^{A,n,\theta} +\textbf{u} \cdot \bigtriangledown E_{c}^{A,n,\theta} - \alpha E_{c}^{A,n,\theta} )_{\Omega_k} \\
& \quad - \frac{\tau_3'^2}{(dt-\tau_3' )} \sum_{k=1}^{n_{el}}(\bigtriangledown \cdot \tilde{\bigtriangledown} E^{I,n,\theta}_{c} - \textbf{u} \cdot \bigtriangledown E^{I,n,\theta}_{c}-\alpha E^{I,n,\theta}_{c},\bigtriangledown \cdot \tilde{\bigtriangledown} E_{c}^{A,n,\theta} +  \\
& \quad \textbf{u} \cdot \bigtriangledown E_{c}^{A,n,\theta}- \alpha E_{c}^{A,n,\theta} )_{\Omega_k}  -\frac{ \tau_3'^2}{(dt- \tau_3')}\sum_{k=1}^{n_{el}}(\bigtriangledown \cdot \tilde{\bigtriangledown} E^{A,n,\theta}_{c} - \textbf{u} \cdot \bigtriangledown E^{A,n,\theta}_{c}- \\
& \quad \alpha E^{A,n,\theta}_{c}, \bigtriangledown \cdot \tilde{\bigtriangledown} E_{c}^{A,n,\theta}+ \textbf{u} \cdot \bigtriangledown E_{c}^{A,n,\theta} - \alpha E_{c}^{A,n,\theta} )_{\Omega_k} \\ 
& \leq \frac{\tau_3^2}{dt (dt+ \tau_3)}\{ \sum_{k=1}^{n_{el}}(\alpha B_{6k} +D_{B_{1k}})\} (\|E^{A,n+1}_{c}\|^2- \|E^{A,n}_{c}\|^2)+ \frac{\tau_3^2 Ch^2}{dt (dt+ \tau_3)}   \\
& \quad \{\sum_{k=1}^{n_{el}} D_{B_{1k}}\} (\mid c^{n+1} \mid_2  + \mid c^n \mid_2)+ \frac{ \tau_3}{dt} C \{ (\sum_{k=1}^{n_{el}} D_{B_{2k}})(\frac{1+\theta}{2}\mid c^{n+1}\mid_2 +\frac{1-\theta}{2}  \\
& \quad \mid c^n \mid_2)+ h (\sum_{k=1}^{n_{el}} D_{B_{2k}})(\frac{1+\theta}{2}\mid c^{n+1}\mid_1 +\frac{1-\theta}{2}  \mid c^n \mid_1) + h^2 (\sum_{k=1}^{n_{el}} D_{B_{4k}})(\frac{1+\theta}{2} \\
& \quad  \| c^{n+1}\|+ \frac{1-\theta}{2} \| c^n \|) \} + \frac{\tau_3}{ dt} \sum_{k=1}^{n_{el}} D_{B_{5k}}\\
& \leq \frac{C_{\tau_3}^2}{dt (T_0- C_{\tau_3})}\{ \sum_{k=1}^{n_{el}}(\alpha B_{6k} +D_{B_{1k}})\} (\|E^{A,n+1}_{c}\|^2- \|E^{A,n}_{c}\|^2)+ \frac{C_{\tau_3}^2 Ch^2}{T_0 (T_0- C_{\tau_3})} \\
& \quad \{\sum_{k=1}^{n_{el}} D_{B_{1k}}\} (\| c^{n+1} \|_2  + \| c^n \|_2)+ \frac{ \mid \tau_3 \mid}{T_0} C \{ (\sum_{k=1}^{n_{el}} D_{B_{2k}})(\frac{1+\theta}{2}\| c^{n+1}\|_2 +\frac{1-\theta}{2}\\
& \quad \| c^n \|_2)+ h (\sum_{k=1}^{n_{el}} D_{B_{3k}})(\frac{1+\theta}{2}\| c^{n+1}\|_1 +\frac{1-\theta}{2}  \| c^n \|_1) + h^2 (\sum_{k=1}^{n_{el}} D_{B_{4k}})(\frac{1+\theta}{2} \\
& \quad  \| c^{n+1}\|+ \frac{1-\theta}{2} \| c^n \|) \} + \frac{\mid \tau_3 \mid}{ T_0} \sum_{k=1}^{n_{el}} D_{B_{5k}}\\
\end{split}
\end{equation}
This completes finding the bounds for $I_4$. \vspace{1mm} \\
Now we will find bounds for  $I_5$ and $I_6$ in similar manner as many terms of $I_5, I_6$ coincide with the terms of $I_3$ and $I_4$. \vspace{1mm}\\
\begin{equation}
\begin{split}
- I_5 & = (1-\tau_3^{-1}\tau_3') \sum_{k=1}^{n_{el}} (\partial_t E^{A,n}_{c}, E_{c}^{A,n,\theta} )_{\Omega_k} + (1-\tau_3^{-1}\tau_3') \sum_{k=1}^{n_{el}}(\textbf{u}^n \cdot \bigtriangledown E^{I,n,\theta}_{c}- \bigtriangledown \cdot \tilde{\bigtriangledown} E^{I,n,\theta}_{c}, \\
& \quad E_{c}^{A,n,\theta} )_{\Omega_k} + (1-\tau_3^{-1}\tau_3') \sum_{k=1}^{n_{el}} (\textbf{u}^n \cdot \bigtriangledown E^{A,n,\theta}_{c}- \bigtriangledown \cdot \tilde{\bigtriangledown} E^{A,n,\theta}_{c} + \alpha E^{A,n,\theta}_{c} ,E^{A,n,\theta}_{c})_{\Omega_k} \\
& = Q_1 + Q_2 + Q_3
\end{split}
\end{equation}
where 
\begin{equation}
\begin{split}
Q_1 & = (1-\tau_3^{-1}\tau_3')\sum_{k=1}^{n_{el}} (\partial_t E_{c}^{A,n}, E_{c}^{A,n,\theta})_{\Omega_k} \\
&= \frac{\tau_3}{dt+\tau_3} \sum_{k=1}^{n_{el}} (\partial_t E_{c}^{A,n}, E_{c}^{A,n,\theta})_{\Omega_k} \\
& \leq \frac{C_{\tau_3}}{dt(T_0-C_{\tau_3})} (\sum_{k=1}^{n_{el}} B_{6k})(\|E_{c}^{A,n+1}\|^2 - \|E_{c}^{A,n}\|^2)
\end{split}
\end{equation}
\begin{equation}
\begin{split}
Q_2 & = (1-\tau_3^{-1}\tau_3')\sum_{k=1}^{n_{el}}(\textbf{u}^n \cdot \bigtriangledown E^{I,n,\theta}_{c}- \bigtriangledown \cdot \tilde{\bigtriangledown} E^{I,n,\theta}_{c}, E_{c}^{A,n,\theta} )_{\Omega_k} \\
& \leq \frac{C_{\tau_3}}{(T_0-C_{\tau_3})} \{\sum_{k=1}^{n_{el}} (\bar{D}_{u1} + \bar{D}_{u2} )B_{6k}\}  \| E^{I,n,\theta}_{c} \|_1\\
& \leq \frac{\mid \tau_3 \mid Ch}{(T_0-C_{\tau_3})} \{\sum_{k=1}^{n_{el}} (\bar{D}_{u1} + \bar{D}_{u2} )B_{6k}\} (\frac{1+\theta}{2} \| c^{n+1} \|_1 + \frac{1-\theta}{2} \| c^n \|_1)
\end{split}
\end{equation}
and 
\begin{equation}
\begin{split}
Q_3 & = (1-\tau_3^{-1}\tau_3') \sum_{k=1}^{n_{el}} (\textbf{u}^n \cdot \bigtriangledown E^{A,n,\theta}_{c}- \bigtriangledown \cdot \tilde{\bigtriangledown} E^{A,n,\theta}_{c} + \alpha E^{A,n,\theta}_{c} ,E^{A,n,\theta}_{c})_{\Omega_k} \\
& \leq \frac{C_{\tau_3}}{(T_0-C_{\tau_3})} \sum_{k=1}^{n_{el}}(D_{1m}B_{7k}+D_{2m}B_{7k}'+\bar{D}_{u1}B_{8k}+\bar{D}_{u2}B_{8k}'+\alpha B_{6k})B_{6k}
\end{split}
\end{equation}
Combining all these results we will have
\begin{equation}
\begin{split}
-I_5 & \leq \frac{C_{\tau_3}}{(T_0-C_{\tau_3})}  (\sum_{k=1}^{n_{el}}\frac{B_{6k}}{dt})(\|E_{c}^{A,n+1}\|^2 - \|E_{c}^{A,n}\|^2)+ \\
& \quad \frac{\mid \tau_3 \mid}{(T_0-C_{\tau_3})} \{(\sum_{k=1}^{n_{el}} (\bar{D}_{u1} + \bar{D}_{u2} )B_{6k}) Ch (\frac{1+\theta}{2} \| c^{n+1} \|_1 + \frac{1-\theta}{2} \| c^n \|_1) \\
& \quad +\sum_{k=1}^{n_{el}} (D_{1m}B_{7k}+D_{2m}B_{7k}'+\bar{D}_{u1}B_{8k}+\bar{D}_{u2}B_{8k}'+\alpha B_{6k})B_{6k} \}
\end{split}
\end{equation}
This completes finding the bound for each term of $I_5$. Now we focus on deriving bounds of $I_6$
\begin{equation}
\begin{split}
-I_6 & = -\sum_{k=1}^{n_{el}} \tau_3^{-1}\tau_3' (d_4, E^{A,n,\theta}_{c})_{\Omega_k} \\
& \leq  \frac{\mid \tau_3^{-1} \mid \tau_3'^2}{(dt-\mid \tau_3' \mid)} \{ (\partial_t E^{I,n}_{c}+\partial_t E^{A,n}_{c},E^{A,n,\theta}_{c})_{\Omega_k}-(\bigtriangledown \cdot \tilde{\bigtriangledown} E^{I,n,\theta}_{c} - \textbf{u} \cdot \bigtriangledown E^{I,n,\theta}_{c}- \\
& \quad \alpha E^{I,n,\theta}_{c}, E^{A,n,\theta}_{c})_{\Omega_k} -(\bigtriangledown \cdot \tilde{\bigtriangledown} E^{A,n,\theta}_{c} - \textbf{u} \cdot \bigtriangledown E^{A,n,\theta}_{c}-\alpha E^{A,n,\theta}_{c},E^{A,n,\theta}_{c})_{\Omega_k} \} \\
& \leq \frac{ \tau_3}{(dt+\tau_3)} \{ (\sum_{k=1}^{n_{el}}
\frac{B_{6k}}{dt})(\|E^{A,n+1}_{c}\|^2-\|E^{A,n}_{c}\|^2) + (\sum_{k=1}^{n_{el}} (D_{1m}+D_{2m})B_{6k}) \\
& \quad  \mid E^{I,n,\theta}_{c}\mid_2 + (\sum_{k=1}^{n_{el}} (\bar{D}_{u1} + \bar{D}_{u2})B_{6k})\mid E^{I,n,\theta}_{c}\mid_1 +\sum_{k=1}^{n_{el}}(D_{1m} B_{7k} +D_{2m} B_{7k}' +  \\
& \quad \bar{D}_{u1} B_{8k} + \bar{D}_{u2}B_{8k}'+ \alpha B_{6k} )B_{6k} \} \\
& \leq \frac{C_{\tau_3}}{(T_0-C_{\tau_3})} \{ (\sum_{k=1}^{n_{el}}
\frac{B_{6k}}{dt})(\|E^{A,n+1}_{c}\|^2-\|E^{A,n}_{c}\|^2)+ \frac{ \mid \tau_3 \mid } {(T_0-C_{\tau_3})} (\sum_{k=1}^{n_{el}} C(D_{1m}+   \\
& \quad D_{2m})B_{6k})(\frac{1+\theta}{2}\| c^{n+1} \|_2 +\frac{1-\theta}{2} \| c^n \|_2)+ C h(\sum_{k=1}^{n_{el}}  (\bar{D}_{u1} + \bar{D}_{u2})B_{6k}) (\frac{1+\theta}{2} \| c^{n+1} \|_1 \\
& \quad +\frac{1-\theta}{2} \| c^n \|_1)+ \sum_{k=1}^{n_{el}}( D_{1m} B_{7k} +D_{2m} B_{7k}' + \bar{D}_{u1} B_{8k} + \bar{D}_{u2}B_{8k}' + \alpha B_{6k} )B_{6k} \}
\end{split}
\end{equation}
Again
\begin{equation}
(TE^{n,\theta}, E^{A,n,\theta}_c) \leq \frac{\epsilon_6}{2} \|TE^{n,\theta}\|^2 + \frac{1}{2 \epsilon_6} \|E^{A,n,\theta}_c\|
\end{equation}
Finally  we have completed finding bounds for each of the terms in the right hand side of (37). Now we explain the further proceeding in language as follows: \vspace{1mm} \\
First we put all the bounds, obtained for each of the terms in the right hand side of (37). Then we  take out few common terms in the left hand side and consequently we have left 3 types of terms in the right hand side. One type will be few constant terms multiplied by $h^2$, other type will be another few constant terms multiplied by h and the remaining constant terms will be free of h. Now we multiply both sides by 2 and taking integration over $(t^n,t^{n+1})$ for n=0,1,...,$(N-1)$ to both the sides. Finally we have (37) as follows:

\begin{multline}
\{1-\frac{2C_{\tau_3}(T+C_{\tau_3})}{T_0-C_{\tau_3}} \sum_{k=1}^{n_{el}}(\alpha B_{6k}+ D_{B_{1k}})- \frac{4C_{\tau_3}}{T_0-C_{\tau_3}} (\sum_{k=1}^{n_{el}} B_{6k})\} \sum_{n=0}^{N-1} \int_{t^n}^{t^{n+1}} (\|E^{A,n+1}_{c}\|^2-\|E^{A,n}_{c}\|^2) \\
\quad + (\mu_l-\frac{1}{ \epsilon_1}-\frac{\mu_u}{ \epsilon_4}) \sum_{n=0}^{N-1} \int_{t^n}^{t^{n+1}} (\|\frac{\partial E^{A,n,\theta}_{u1}}{\partial x}\|^2 + \|\frac{\partial E^{A,n,\theta}_{u2}}{\partial y}\|^2 )dt+ (\mu_l-\frac{\mu_u}{ \epsilon_4}) \sum_{n=0}^{N-1} \int_{t^n}^{t^{n+1}} (\|\frac{\partial E^{A,n,\theta}_{u1}}{\partial y}\|^2+\\
\quad  \|\frac{\partial E^{A,n,\theta}_{u2}}{\partial x}\|^2 )dt  + (D_l-\frac{D_m}{ \epsilon_2}-\frac{C_1^n}{ \epsilon_3})\sum_{n=0}^{N-1} \int_{t^n}^{t^{n+1}} \|\frac{\partial E^{A,n,\theta}_{c}}{\partial x}\|^2 dt + (D_l-\frac{D_m}{ \epsilon_2}-\frac{C_2^n}{ \epsilon_3})\sum_{n=0}^{N-1} \int_{t^n}^{t^{n+1}} \| \frac{\partial E^{A,n,\theta}_{c}}{\partial y}\|^2 dt\\
\quad   + 
 \sigma \sum_{n=0}^{N-1}\int_{t^n}^{t^{n+1}} (\|E^{A,n,\theta}_{u1}\|^2 + \|E^{A,n,\theta}_{u2}\|^2) dt +  (\alpha- 2\epsilon_3(C_1^n + C_2^n)-\frac{1}{\epsilon_6}) \sum_{n=0}^{N-1}\int_{t^n}^{t^{n+1}}\|E^{A,n,\theta}_{c}\|^2 dt \\
\leq 2C h^2 \sum_{n=0}^{N-1}\int_{t^n}^{t^{n+1}} [C \mu_u \epsilon_4 \sum_{i=1}^2 (\frac{1+\theta}{2} \| u_i^{n+1} \|_2 + \frac{1-\theta}{2} \| u_i^n \|_2)^2 + C \epsilon_1 (\frac{1+\theta}{2} \| p^{n+1} \|_1 + \frac{1-\theta}{2} \| p^n \|_1)^2 + \\
\quad  C(\frac{D_m \epsilon_2}{2} + \frac{C_1^n+C_2^n}{2 \epsilon_3}) (\frac{1+\theta}{2} \| c^{n+1} \|_2 + \frac{1-\theta}{2} \| c^n \|_2)^2 + \frac{\mid \tau_3 \mid (T_0+\mid \tau_3 \mid)}{T_0(T_0-C_{\tau_3})} (\sum_{k=1}^{n_{el}} D_{B_{1k}})(\| c^{n+1} \|_2 + \\
\quad    \| c^n \|_2)+ \mid \tau_1 \mid (\sum_{k=1}^{n_{el}} \sigma (\mu_u B_{2k} + \sigma B_{1k} + B_{3k}))(\frac{1+\theta}{2} \| u_1^{n+1} \|_2 + \frac{1-\theta}{2} \| u_1^{n} \|_2)+ \mid \tau_1 \mid (\sum_{k=1}^{n_{el}} \sigma (\mu_u B_{2k}' + \\
\quad   \sigma  B_{1k}' + B_{3k}'))(\frac{1+\theta}{2} \| u_2^{n+1} \|_2 + \frac{1-\theta}{2} \| u_2^{n} \|_2)+ (\sum_{k=1}^{n_{el}} D_{B_{4k}})\frac{\mid \tau_3 \mid}{T_0}(\frac{1+\theta}{2} \| c^{n+1} \|+ \frac{1-\theta}{2} \| c^{n} \|) ] dt \\
\quad   + 2 C h \sum_{n=0}^{N-1} \int_{t^n}^{t^{n+1}}[ \mid \tau_1 \mid (\sum_{k=1}^{n_{el}}(\mu_u(B_{2k}+ B_{2k}')+ \sigma (B_{1k}+B_{1k}')+ (B_{3k}+B_{3k}'))(\frac{1+\theta}{2}\| p^{n+1} \|_1 + \frac{1-\theta}{2} \| p^n \|_1) \\
\quad   + \mid \tau_2 \mid (\sum_{k=1}^{n_{el}}B_{4k}) (\frac{1+\theta}{2} \| u_1^{n+1} \|_1 + \frac{1-\theta}{2} \| u_1^n \|_1) + \mid \tau_2 \mid (\sum_{k=1}^{n_{el}}B_{5k}')(\frac{1+\theta}{2} \| u_2^{n+1} \|_1 +\frac{1-\theta}{2} \| u_2^n \|_1) \\
\quad + \{ (\frac{\mid \tau_3 \mid T }{(T_0-C_{\tau_3})}+\frac{\mid \tau_3 \mid}{T_0})(\sum_{k=1}^{n_{el}} D_{B_{3k}})+\frac{4\mid \tau_3 \mid}{T_0-C_{\tau_3}} (\sum_{k=1}^{n_{el}}(\bar{D}_{u1}+\bar{D}_{u2})B_{6k})\}(\frac{1+\theta}{2} \| c^{n+1} \|_1 +\frac{1-\theta}{2} \\
\quad \| c^n \|_1)] dt +2C \sum_{n=0}^{N-1} \int_{t^n}^{t^{n+1}} [ \mid \tau_1 \mid (\sum_{k=1}^{n_{el}}(\mu_u^2 B_{2k} + \sigma \mu_u B_{1k} +\mu_u B_{3k} ))(\frac{1+\theta}{2} \| u_1^{n+1} \|_2 +\frac{1-\theta}{2}\| u_1^n \|_2)+\\
\quad \mid \tau_1 \mid (\sum_{k=1}^{n_{el}}(\mu_u^2 B_{2k}' + \sigma \mu_u B_{1k}' +\mu_u B_{3k}' ))(\frac{1+\theta}{2} \| u_2^{n+1} \|_2 +\frac{1-\theta}{2}\| u_2^n \|_2)+ (\frac{\mid \tau_3 \mid T }{(T_0-C_{\tau_3})}+\frac{\mid \tau_3 \mid}{T_0}) \\  
\quad (\sum_{k=1}^{n_{el}}D_{B_{2k}})(\frac{1+\theta}{2} \| c^{n+1} \|_1 +\frac{1-\theta}{2}\| c^n \|_1) + \mid \tau_1 \mid (\sum_{k=1}^{n_{el}}(M_{1k}+M_{2k}))+(\frac{\mid \tau_3 \mid T }{(T_0-C_{\tau_3})}+\frac{\mid \tau_3 \mid}{T_0}) \\
\quad  \{(\sum_{k=1}^{n_{el}} D_{B_{5k}}) +(\sum_{k=1}^{n_{el}} (D_{1m}+D_{2m})B_{6k})  (\frac{1+\theta}{2}\| c^{n+1} \|_2 +\frac{1-\theta}{2} \| c^n \|_2)+ 2( \sum_{k=1}^{n_{el}}( D_{1m} B_{7k} +D_{2m} B_{7k}'\\
\quad  + \bar{D}_{u1} B_{8k} + \bar{D}_{u2}B_{8k}' + \alpha B_{6k} )B_{6k})\} ]dt + \epsilon_6 \sum_{n=0}^{N-1} \int_{t^n}^{t^{n+1}}  \|TE^{n,\theta}\|^2 dt \hspace{10mm}
\end{multline}
We can choose the values of the arbitrary parameters in such a manner that we can make all the coefficients in the left hand side positive. In order to satisfy such condition it is inevitable to choose $h$ small. Now after taking minimum of all the coefficients in left hand side, let us divide both the sides with that minimum, which turns out to be a positive real number. Applying assumption $\textbf{(iv)}$ it can be seen that $\|u_i^n\|_2$ for $i=1,2$, $\|p^n\|_1$ and $\|c^n\|_2$ are bounded for $n=0,1,2,...,N$. Now by applying initial condition on $c$ we will have $\|E^{A,0}_{c}\|=0$. \vspace{1mm}\\
After performing all these intermediate steps we will finally arrive at the following expression since $\tau_1$ and $\tau_3$ are of order $h^2$:
\begin{multline}
\|E^{A,N}_{c}\|^2 +\sum_{n=0}^{N-1}  \int_{t^n}^{t^{n+1}} (\|\frac{\partial E^{A,n,\theta}_{u1}}{\partial x}\|^2 + \|\frac{\partial E^{A,n,\theta}_{u1}}{\partial y}\|^2 + \|E^{A,n,\theta}_{u1} \|^2 )dt+ \\
 \quad \sum_{n=0}^{N-1}  \int_{t^n}^{t^{n+1}} (\|\frac{\partial E^{A,n,\theta}_{u2}}{\partial x}\|^2 + \|\frac{\partial E^{A,n,\theta}_{u2}}{\partial y}\|^2 + \|E^{A,n,\theta}_{u2} \|^2 )dt + \\
 \quad \sum_{n=0}^{N-1}  \int_{t^n}^{t^{n+1}} (\|\frac{\partial E^{A,n,\theta}_{c}}{\partial x}\|^2 + \|\frac{\partial E^{A,n,\theta}_{c}}{\partial y}\|^2 + \|E^{A,n,\theta}_{c} \|^2 )dt \leq C(T,\textbf{u},p,c) (h^2+h+dt^{2r})
\end{multline}
This implies
\begin{equation}
\|E^{A}_{u1}\|_{L^2(H^1)}^2+ \|E^{A}_{u2}\|_{L^2(H^1)}^2+ \|E^{A}_{c}\|_{\tilde{\textbf{V}}}^2 \leq C(T,\textbf{u},p,c) (h^2+h+dt^{2r})
\end{equation}
where
\begin{equation}
    r=
    \begin{cases}
      1, & \text{if}\ \theta=1 \\
      2, & \text{if}\ \theta=0
    \end{cases}
  \end{equation}
We have used the fact that $\sum_{n=0}^{N-1} \int_{t^n}^{t^{n+1}} M dt \leq M T  $ and the property of $TE$ given in (17). This completes the first part of the proof. \vspace{2mm}\\
\textbf{Second part} Using this above result we are going to estimate auxiliary error part of pressure. We will use inf-sup condition to find estimate for $E_p^A$. Applying Galerkin orthogonality only for variational form of Stokes-Darcy flow problem we have obtained \\
\begin{equation}
\begin{split}
 a_S (\textbf{u}-\textbf{u}_h, \textbf{v}_h)- b(\textbf{v}_h, p-p_h) & =0 \\ 
b(\textbf{v}_h, p-I_hp)+ b(\textbf{v}_h, I_hp-p_h) & =  a_S(E^I_\textbf{u}, \textbf{v}_h)+  a_S(E^A_\textbf{u}, \textbf{v}_h)
\end{split}
\end{equation}
Assuming the inclusion $\bigtriangledown \cdot V_s^h \subset Q_s^h$ and the property of the $L^2$ orthogonal projection of $I^h_p$ we have 
\begin{equation}
b(\textbf{v}_h, p-I_hp)= \int_{\Omega}(p-I_hp)(\bigtriangledown \cdot \textbf{v}_h)=0
\end{equation}
Now according to inf-sup condition we will have the following expression
\begin{equation}
\begin{split}
\|I_hp-p_h\|_{L^2(L^2)}^2& = \|E_p^A\|_{L^2(L^2)}^2 \\
& = \sum_{n=0}^{N-1} \int_{t^n}^{t^{n+1}} \|E_p^{A,n,\theta}\|^2 dt \\
& \leq  \sum_{n=0}^{N-1} \int_{t^n}^{t^{n+1}} \underset{\textbf{v}_h}{sup} \frac{b(\textbf{v}_h, E_p^{A,n,\theta})}{\|\textbf{v}_h\|_1} dt
\end{split}
\end{equation}
Now from (75)
\begin{equation}
\begin{split}
\sum_{n=0}^{N-1} \int_{t^n}^{t^{n+1}} b(\textbf{v}_h, E_p^{A,n,\theta}) dt & = \sum_{n=0}^{N-1} \int_{t^n}^{t^{n+1}} \{ a_S(E^{I,n,\theta}_\textbf{u}, \textbf{v}_h)+a_S(E^{A,n,\theta}_\textbf{u}, \textbf{v}_h)\} dt\\
& = \sum_{n=0}^{N-1} \int_{t^n}^{t^{n+1}} \{ \int_{\Omega} \mu(c^n) \bigtriangledown E^{I,n,\theta}_{u1} \cdot \bigtriangledown v_{1h} +  \int_{\Omega} \mu(c^n) \bigtriangledown E^{I,n,\theta}_{u2} \cdot \bigtriangledown v_{2h}+ \\
& \quad \sigma \int_{\Omega}  (E^{I,n,\theta}_{u1} v_{1h} +E^{I,n,\theta}_{u2} v_{2h})+  \int_{\Omega} \mu(c^n) \bigtriangledown E^{A,n,\theta}_{u1} \cdot \bigtriangledown v_{1h} + \\
& \quad \int_{\Omega} \mu(c^n) \bigtriangledown E^{A,n,\theta}_{u2} \cdot \bigtriangledown v_{2h} +  \sigma \int_{\Omega}  (E^{A,n,\theta}_{u1} v_{1h} +E^{A,n,\theta}_{u2} v_{2h}) \} dt\\
& \leq (\mu_u + \sigma) \sum_{n=0}^{N-1} \int_{t^n}^{t^{n+1}} \{(\|E^{I,n,\theta}_{u1}\|_1+\|E^{A,n,\theta}_{u1}\|_1) \|v_{1h}\|_1 + (\|E^{I,n,\theta}_{u2}\|_1 + \\
& \quad \|E^{A,n,\theta}_{u2}\|_1) \|v_{2h}\|_1 \} dt \\
& \leq (\mu_u + \sigma) \sum_{n=0}^{N-1} \int_{t^n}^{t^{n+1}} \{(\|E^{I,n,\theta}_{u1}\|_1 + \|E^{I,n,\theta}_{u2}\|_1 +\|E^{A,n,\theta}_{u1}\|_1+\|E^{A,n,\theta}_{u2}\|_1)\\
& \quad (\|v_{1h}\|_1+\|v_{2h}\|_1)\}dt\\
& \leq (\mu_u + \sigma)(\|E^{A}_{u1}\|_{L^2(H^1)} + \|E^{A}_{u2}\|_{L^2(H^1)})(\|v_{1h}\|_1+\|v_{2h}\|_1)+ \\
& \quad (\mu_u + \sigma) \sum_{n=0}^{N-1} \int_{t^n}^{t^{n+1}}(\|E^{I,n,\theta}_{u1}\|_1+\|E^{I,n,\theta}_{u2}\|_1)(\|v_{1h}\|_1+\|v_{2h}\|_1)dt\\
& \leq (\mu_u + \sigma)(\|E^{A}_{u1}\|_{L^2(H^1)}^2 + \|E^{A}_{u2}\|_{L^2(H^1)}^2+ \|E^{A}_{c}\|_{\tilde{\textbf{V}}}^2) \|\textbf{v}_h\|_1 + (\mu_u + \sigma)Ch \\
& \quad \sum_{n=0}^{N-1} \int_{t^n}^{t^{n+1}} \{\sum_{i=1}^{2}(\frac{1+\theta}{2}\mid u_i^{n+1}\mid_2 + \frac{1-\theta}{2}\mid u_i^{n}\mid_2)dt\} \|\textbf{v}_h\|_1 \\
& \leq C'(T,\textbf{u},p,c) (h^2+h+dt^{2r}) \|\textbf{v}_h\|_1
\end{split}
\end{equation}
Using this above result into (77), we will have the estimate for the pressure term
\begin{equation}
\|I_hp-p_h\|_{L^2(L^2)}^2 \leq C'(T,\textbf{u},p,c) (h^2+h+dt^{2r})
\end{equation}
Now combining the results obtained in the first and second part we have finally arrived at the auxiliary error estimate as follows
\begin{equation}
\|E^A_{u1}\|^2_{L^2(H^1)} + \|E^A_{u2}\|^2_{L^2(H^1)}+ \|E^A_p\|_{L^2(L^2)}^2  + \|E^A_{c}\|^2_{\tilde{\textbf{V}}} \leq \bar{C}(T,\textbf{u},p,c) (h^2+h+ dt^{2r})
\end{equation}
where
\begin{equation}
    r=
    \begin{cases}
      1, & \text{if}\ \theta=1 \\
      2, & \text{if}\ \theta=0
    \end{cases}
  \end{equation}
This completes the proof.
\end{proof}
\begin{theorem}(Apriori error estimate)
Assuming the same condition as in the previous theorem, 
\begin{equation}
\|u_1-u_{1h}\|_{L^2(H^1)}^2+ \|u_2-u_{2h}\|_{L^2(H^1)}^2+\|p-p_h\|_{L^2(L^2)}^2 + \|c-c_h\|^2_{\tilde{\textbf{V}}} \leq C'' (h^2+h+ dt^{2r})
\end{equation}
where C' depends on T, $\textbf{u}$,p,c and
\begin{equation}
    r=
    \begin{cases}
      1, & \text{if}\ \theta=1 \\
      2, & \text{if}\ \theta=0
    \end{cases}
  \end{equation}
\end{theorem}
\begin{proof}
By applying triangle inequality, the interpolation inequalities and the result of the previous theorem we will have,
\begin{multline}
\|u_1-u_{1h}\|_{L^2(H^1)}^2+ \|u_2-u_{2h}\|_{L^2(H^1)}^2 +\|p-p_h\|_{L^2(L^2)}^2 +\|c-c_h\|^2_{\tilde{\textbf{V}}} \\
 = \|E^I_{u1}+E^A_{u1}\|_{L^2(H^1)}^2 + \|E^I_{u2}+E^A_{u2}\|_{L^2(H^1)}^2 +  \|E^I_{p}+E^A_{p}\|_{L^2(L^2)}^2+ \|E^I_{c}+E^A_{c}\|_{\tilde{\textbf{V}}}^2 \\
 \leq \bar{C} (\|E^I_{u1}\|_{L^2(H^1)}^2 + \|E^I_{u2}\|_{L^2(H^1)}^2+\|E^I_{p}\|_{L^2(L^2)}^2 +  \|E^I_{c}\|_{\tilde{\textbf{V}}}^2+ 
 \|E^A_{u1}\|_{L^2(H^1)}^2+\|E^A_{u2}\|_{L^2(H^1)}^2\\
  \quad + \|E^A_{p}\|_{L^2(L^2)}^2+\|E^A_{c}\|_{\tilde{\textbf{V}}}^2) \hspace{80mm}\\
 \leq C''(T,\textbf{u},p,c)(h^2+h+ dt^{2r}) \hspace{70mm}
\end{multline}
This completes apriori error estimation.
\end{proof}

\subsection{Aposteriori error estimation}
In this section we are going to derive residual based aposteriori error estimation. \\
We have $B(\textbf{V},\textbf{V})= a_S(\textbf{v},\textbf{v})+a_T(d,d) \geq \mu_l (\|v_1\|_1^2+\|v_2\|_1^2)+ D_{\alpha} \|d\|_1^2$ \hspace{1mm} $\forall \textbf{V} \in \textbf{V}_F$ \\
Now we substitute the errors $e_{u1},e_{u2},e_c$ into the relation we will similarly have
\begin{multline}
\mu_l (\|e_{u1}\|_1^2+\|e_{u2}\|_1^2)+ D_{\alpha} \|e_c\|_1^2  \leq a_S(e_\textbf{u},e_\textbf{u})+a_T(e_c,e_c)\\
\end{multline} 
By adding few terms in both sides the above equation becomes
\begin{multline}
\underbrace{(\frac{\partial e_{c}}{\partial t},e_{c})+\mu_l (\|e_{u1}\|_1^2+\|e_{u2}\|_1^2)+\sigma \|e_p\|^2+ D_{\alpha} \|e_c\|_1^2}_\textit{LHS} \\
 \leq \underbrace{(\frac{\partial e_{c}}{\partial t},e_{c})+a_S(e_\textbf{u},e_\textbf{u})+a_T(e_c,e_c)+ (e_p,e_p)+
 b(e_\textbf{u},e_p)-b(e_\textbf{u},e_p)}_\textit{RHS}
\end{multline} 
Now first we will find a lower bound of $LHS$ and then upper bound for $RHS$ and finally combining them we will get aposteriori error estimate. To find the lower bound the $LHS$ can be written as
\begin{multline}
LHS =  (\frac{e_{c}^{n+1}-e_{c}^n}{dt},e_{c}^{n,\theta})+ 
\mu_l \sum_{i=1}^{2}(\|e_{ui}^{n,\theta}\|^2+ \|\frac{\partial e_{ui}^{n,\theta}}{\partial x}\|^2 +  \|\frac{\partial e_{ui}^{n,\theta}}{\partial y}\|^2)+ \\     \sigma \|e_p^{n,\theta}\|^2 + D_{\alpha}(\|e_{c}^{n,\theta}\|^2+
 \|\frac{\partial e_{c}^{n,\theta}}{\partial x}\|^2 + \|\frac{\partial e_{c}^{n,\theta}}{\partial y}\|^2) \hspace{20mm}
\end{multline}
Using the same argument done in (33) we have
\begin{equation}
\begin{split}
(\frac{e_{c}^{n+1}-e_{c}^n}{dt},e_{c}^{n,\theta}) & \geq \frac{1}{2 dt}(\|e_{c}^{n+1}\|^2-\|e_{c}^n\|^2) \\
\end{split}
\end{equation}
Hence
\begin{multline}
\frac{1}{2 dt}(\|e_{c}^{n+1}\|^2-\|e_{c}^n\|^2)+\mu_l \sum_{i=1}^{2}(\|e_{ui}^{n,\theta}\|^2+ \|\frac{\partial e_{ui}^{n,\theta}}{\partial x}\|^2 + \|\frac{\partial e_{ui}^{n,\theta}}{\partial y}\|^2)+ \\
\sigma \|e_p^{n,\theta}\|^2+
 D_{\alpha}(\|e_{c}^{n,\theta}\|^2+
 \|\frac{\partial e_{c}^{n,\theta}}{\partial x}\|^2 + \|\frac{\partial e_{c}^{n,\theta}}{\partial y}\|^2) \leq LHS \leq RHS
\end{multline}
Now our aim is to find upper bound for $RHS$ through dividing it into two broad parts by splitting errors in each of the terms as follows:

\begin{multline}
RHS = \{ (\frac{e_c^{n+1}-e_c^n}{dt},E^{I,n,\theta}_c)+ a_S(e^{n,\theta}_\textbf{u},E^{I,n,\theta}_\textbf{u})+a_T(e^{n,\theta}_c,E^{I,n,\theta}_{c})\\
-  b(E^{I,n,\theta}_\textbf{u},e^{n,\theta}_p)+ 
b(e^{n,\theta}_\textbf{u},E^{I,n,\theta}_{p})\} + 
\{ (\frac{e^{n+1}_{c}-e_c^n}{ dt},E^{A,n,\theta}_{c})\\
 + a_S(e^{n,\theta}_\textbf{u},E^{A,n,\theta}_\textbf{u})
+a_T(e^{n,\theta}_c,E^{A,n,\theta}_{c})-  b(E^{A,n,\theta}_\textbf{u},e^{n,\theta}_p)+ b(e^{n,\theta}_\textbf{u} ,E^{A,n,\theta}_{p})\}
\end{multline}
In the expression of $RHS$ the first under brace part is first part and second one is second part. Before proceeding further let us introduce the residuals corresponding to each equations \\
\[
\textbf{R}^h=
  \begin{bmatrix}
 \textbf{f}_1-( - \mu(c) \Delta \textbf{u}_h + \sigma \textbf{u}_h + \bigtriangledown p_h) \\
  f_2  -\bigtriangledown \cdot \textbf{u}_h \\
 g-(\frac{\partial c_h}{\partial t} - \bigtriangledown \cdot \tilde{\bigtriangledown} c_h + \textbf{u} \cdot \bigtriangledown c_h + \alpha c_h )
  \end{bmatrix}
\]
This column vector $\textbf{R}^h$ has four components $R_1^h, R_2^h, R_3^h$ and $R_4^h$ denoting four rows respectively. Let us start finding bound for the first part as follows: for all $\textbf{v}=(v_1,v_2) \in V_s \times V_s$
\begin{multline}
 \int_{\Omega} \mu(c^n)\bigtriangledown e_\textbf{u}^{n,\theta}: \bigtriangledown \textbf{v} + \int_{\Omega} \sigma e_\textbf{u}^{n,\theta} \cdot \textbf{v} - \int_{\Omega} (\bigtriangledown \cdot \textbf{v}) e_p^{n,\theta} \\
= \{ \int_{\Omega} \mu(c^n)\bigtriangledown \textbf{u}^{n,\theta}: \bigtriangledown \textbf{v} + \int_{\Omega} \sigma \textbf{u}^{n,\theta} \cdot \textbf{v} - \int_{\Omega} (\bigtriangledown \cdot \textbf{v}) p^{n,\theta} \} -\{ \int_{\Omega} \mu(c^n)\bigtriangledown \textbf{u}_h^{n,\theta}: \bigtriangledown \textbf{v} + \hspace{2mm}\\
  \int_{\Omega} \sigma \textbf{u}_h^{n,\theta} \cdot \textbf{v} - \int_{\Omega} (\bigtriangledown \cdot \textbf{v}) p_h^{n,\theta} \} \hspace{50mm}\\
= \int_{\Omega}( - \mu(c^n) \Delta \textbf{u}^{n,\theta} + \sigma \textbf{u}^{n,\theta} + \bigtriangledown p^{n,\theta}) \cdot \textbf{v} - \int_{\Omega}( - \mu(c^n) \Delta \textbf{u}_h^{n,\theta} + \sigma \textbf{u}_h^{n,\theta} + \bigtriangledown p_h^{n,\theta}) \cdot \textbf{v} \hspace{10mm}\\
= (R_1^{h,n,\theta},v_1)+ (R_2^{h,n,\theta},v_2) \hspace{80 mm}
\end{multline}
Similarly $\int_{\Omega} (\bigtriangledown \cdot e_\textbf{u}^{n,\theta})q = \int_{\Omega} R_3^{h,n,\theta} q$ \hspace{2mm} $\forall q \in Q_s$
\begin{multline}
\int_{\Omega} (\frac{ e_c^{n+1}-e_c^{n}}{dt} d + \tilde{\bigtriangledown} e_c^{n,\theta} \cdot \bigtriangledown d + d \textbf{u} \cdot \bigtriangledown e_c^{n,\theta} + \alpha e_c^{n,\theta} d)= \int_{\Omega} R_4^{h,n,\theta} d \hspace{3 mm} \forall d \in V_s \hspace{30 mm}
\end{multline}
Now substituting $v_1, v_2 ,q,d$ in the above expressions by $E^{I,n,\theta}_{u1},E^{I,n,\theta}_{u2}, E^{I,n,\theta}_{p}, E^{I,n,\theta}_{c}$ respectively, we will have the first part of the $RHS$ as,

\begin{equation}
\begin{split}
First \hspace{1mm} part \hspace{1mm} of \hspace{1mm} RHS & = \int_{\Omega}\{ R_1^{h,n,\theta} E^{I,n,\theta}_{u1} +  R_2^{h,n,\theta} E^{I,n,\theta}_{u2} + R_3^{h,n,\theta} E^{I,n,\theta}_{p} + R_4^{h,n,\theta} E^{I,n,\theta}_{c}\}\\
& \leq \|R_1^{h,n,\theta}\| \|E^{I,n,\theta}_{u1}\| + \|R_2^{h,n,\theta}\| \|E^{I,n,\theta}_{u2}\| + \|R_3^{h,n,\theta}\| \|E^{I,n,\theta}_{p}\|+  \\
& \quad \|R_4^{h,n,\theta}\| \|E^{I,n,\theta}_{c}\| \hspace{10mm}( by \hspace{1mm} Cauchy-Schwarz \hspace{1mm} inequality)\\
& \leq \|R_1^{h,n,\theta}\| h \| u_1^{n,\theta}\|_1 + \|R_2^{h,n,\theta}\| h \| u_2^{n,\theta}\|_1 + \|R_3^{h,n,\theta}\|  \| p^{n,\theta}\| + \\
& \quad \|R_4^{h,n,\theta}\| h \| c^{n,\theta}\|_1\\
& \leq C_2(\|R_1^{h,n,\theta}\| h \| e_{u1}^{n,\theta}\|_1 + \|R_2^{h,n,\theta}\| h \| e_{u2}^{n,\theta}\|_1 + \\
& \quad \|R_3^{h,n,\theta}\|  \| e_{p}^{n,\theta}\| + \|R_4^{h,n,\theta}\| h \| e_{c}^{n,\theta}\|_1)\\
& \leq C_2 \frac{h^2}{2\epsilon_1}(\|R_1^{h,n,\theta}\|^2+\|R_2^{h,n,\theta}\|^2+\|R_3^{h,n,\theta}\|^2+\|R_4^{h,n,\theta}\|^2)+ \\
& \quad \frac{\epsilon_1}{2}( \| e_{u1}^{n,\theta}\|_1^2+  \| e_{u2}^{n,\theta}\|_1^2+\| e_{p}^{n,\theta}\|^2 + \| e_{c}^{n,\theta}\|_1^2)
 \hspace{3mm}( by \hspace{1mm} Young's \hspace{1mm} inequality) \\
\end{split}
\end{equation}
This completes finding bound for first part of $RHS$. Now we are going to estimate remaining second part of $RHS$. For that we will use subgrid formulation (8). Subtracting (8) from the variational finite element formulation satisfied by the exact solution we have $\forall \textbf{V}_h \in \textbf{V}_F^h $
\begin{multline}
\int_{\Omega} \frac{e_{c}^{n+1}-e_c^n}{ dt} d_{h}+ \int_{\Omega} \mu(c^n)\bigtriangledown e_\textbf{u}^{n,\theta}: \bigtriangledown \textbf{v}_h + \int_{\Omega} \sigma e_\textbf{u}^{n,\theta} \cdot \textbf{v}_h - \int_{\Omega} (\bigtriangledown \cdot \textbf{v}_h) e_p^{n,\theta} + \int_{\Omega} (\bigtriangledown \cdot e_\textbf{u}^{n,\theta}) q_h  \\
 + \int_{\Omega} \tilde{\bigtriangledown} e_c^{n,\theta} \cdot \bigtriangledown d_h + \int_{\Omega} d_h \textbf{u} \cdot \bigtriangledown e_c^{n,\theta} + \int_{\Omega} \alpha e_c^{n,\theta} d_h \\
 = \sum_{k=1}^{n_{el}} \{(\tau_k'(\textbf{R}^{h,n,\theta}+\textbf{d}), -\mathcal{L}^* \textbf{V}_h)_{\Omega_k} - ((I-\tau_k^{-1}\tau_k) \textbf{R}^{h,n,\theta}, \textbf{V}_h)_{\Omega_k} + (\tau_k^{-1}\tau_k \textbf{d}, \textbf{V}_h)_{\Omega_k} \} \\
 +(TE^{n,\theta},d_h)\hspace{80mm}\\
= \sum_{k=1}^{n_{el}} \{ \tau_1'(R_1^{h,n,\theta}, \mu(c) \Delta v_{1h}-\sigma v_{1h}+ \frac{\partial q_h}{\partial x})_{k} +\tau_1'(R_2^{h,n,\theta}, \mu(c) \Delta v_{2h}-\sigma v_{2h}+ \frac{\partial q_h}{\partial y})_{k} \\
\quad +\tau_2'(R_3^{h,n,\theta}, \bigtriangledown \cdot \textbf{v}_h)_{k} + \tau_3'(R_4^{h,n,\theta}+d_4, \bigtriangledown \cdot \tilde{\bigtriangledown} d_h + \textbf{u} \cdot \bigtriangledown d_h - \alpha d_h)_{k} \\
\quad + (1-\tau_3^{-1}\tau_3')(R_4^{h,n,\theta},d_{h})_k +\tau_3^{-1}\tau_3' (d_4,d_{h})_k\} +(TE^{n,\theta},d_h) \hspace{5mm}
\end{multline}
Here $(\cdot,\cdot)_k $ in simple form denotes $(\cdot,\cdot)_{\Omega_k}$ \vspace{1mm}\\
Now substituting $\textbf{V}_h$ by $(E^{A,n,\theta}_{u1}, E^{A,n,\theta}_{u2}, E^{A,n,\theta}_{p}, E^{A,n,\theta}_{c})$ in the above equation we will get the second part of $RHS$ as follows

\begin{multline}
 (\frac{e_{c}^{n+1}-e_c^n}{dt},E^{A,n,\theta}_{c})
 + a_S(e_\textbf{u}^{n,\theta},E^{A,n,\theta}_\textbf{u})
+a_T(e_c^{n,\theta},E^{A,n,\theta}_{c} e_{c})-b(E^{A,n,\theta}_\textbf{u},e_p^{n,\theta})+ b(e_\textbf{u}^{n,\theta},E^{A,n,\theta}_{p}) \\
 = \sum_{k=1}^{n_{el}} \{ \tau_1'(R_1^{h,n,\theta}, \mu(c) \Delta E^{A,n,\theta}_{u1}-\sigma E^{A,n,\theta}_{u1}+ \frac{\partial E^{A,n,\theta}_{p}}{\partial x})_k + \tau_2'(R_3^{h,n,\theta}, \bigtriangledown \cdot E^{A,n,\theta}_\textbf{u})_k +\\
\quad \tau_1'(R_2^{h,n,\theta},\mu(c) \Delta E^{A,n,\theta}_{u2} -\sigma E^{A,n,\theta}_{u2}+ \frac{\partial E^{A,n,\theta}_{p} }{\partial y})_k+ (1-\tau_3^{-1}\tau_3')(R_4^{h,n,\theta},E^{A,n,\theta}_{c})_k + \\
\quad \tau_3'(R_4^{h,n,\theta}+d_4, \bigtriangledown \cdot \tilde{\bigtriangledown} E^{A,n,\theta}_{c} + \textbf{u} \cdot \bigtriangledown E^{A,n,\theta}_{c} - \alpha E^{A,n,\theta}_{c} )_k + 
\tau_3^{-1}\tau_3' (d_4,E^{A,n,\theta}_{c})_k\} +(TE^{n,\theta},E^{A,n,\theta}_{c})
\end{multline}
Now we will bound each of the term starting with 4th term of the right hand side of the above equation.
\begin{equation}
\begin{split}
\sum_{k=1}^{n_{el}}(1-\tau_3^{-1}\tau_3')(R_4^{h,n,\theta},E^{A,n,\theta}_{c})_k & \leq \frac{\mid \tau_3 \mid}{T_0-C_{\tau_3}} (\sum_{k=1}^{n_{el}}  B_{6k}) \|R_4^{h,n,\theta}\|
\end{split}
\end{equation}
We have obtained this using Cauchy-Schwarz inequality and then imposing bound on auxiliary error corresponding to $u_1$ over each sub-domain $\Omega_k$. Before proceeding further let us look into the form of the column vector $\textbf{d}$ which has components $d_1,d_2,d_3$ and $d_4$ \vspace{1mm} \\
$\textbf{d}$= $\sum_{i=1}^{n+1}(\frac{1}{dt}M\tau_k')^i(\textbf{F} -M\partial_t \textbf{U}_h - \mathcal{L}\textbf{U}_h)=\sum_{i=1}^{n+1}(\frac{1}{dt}M\tau_k')^i \textbf{R}^h$ \vspace{2mm}\\
Hence clearly $d_1 =0$, $d_2=0$, $d_3= 0$ and $d_4=(\sum_{i=1}^{n+1}(\frac{1}{dt} \tau_3')^i) R_4^{h,n,\theta}$ \vspace{2mm} \\
Now we can bound the last term as follows
\begin{equation}
\begin{split}
\sum_{k=1}^{n_{el}}\tau_3^{-1}\tau_3' (d_4,E^{A,n,\theta}_{c})_k & = \tau_3^{-1}\tau_3' (\sum_{i=1}^{n+1}(\frac{1}{dt} \tau_3')^i) \sum_{k=1}^{n_{el}} (R_4^{h,n,\theta},E^{A,n,\theta}_{c})_k \\
& \leq \tau_3^{-1}\tau_3' (\sum_{i=1}^{\infty}(\frac{1}{dt} \tau_3')^i) \sum_{k=1}^{n_{el}} (R_4^{h,n,\theta},E^{A,n,\theta}_{c})_k \\
& \leq \frac{ \mid \tau_3 \mid}{(T_0-C_{\tau_3}}) (\sum_{k=1}^{n_{el}} B_{6k}) \|R_4^{h,n,\theta}\|
\end{split}
\end{equation}
Now it will be easy enough to bound the remaining terms of the right hand side.
\begin{equation}
\begin{split}
\sum_{k=1}^{n_{el}} \tau_1'(R_1^{h,n,\theta}, \mu(c) \Delta E^{A,n,\theta}_{u1}-\sigma E^{A,n,\theta}_{u1}+ \frac{\partial E^{A,n,\theta}_{p}}{\partial x})_k &  \\
 \leq \mid \tau_1'\mid (\sum_{k=1}^{n_{el}}(\mu_u B_{2k} + \mid \sigma \mid B_{1k}+ B_{3k})) \|R_1^{h,n,\theta}\| \hspace{7 mm} \\
 \leq \mid\tau_1 \mid (\sum_{k=1}^{n_{el}} \bar{B}_{1k}) \|R_1^{h,n,\theta}\|  \hspace{40 mm}\\
\sum_{k=1}^{n_{el}} \tau_1'(R_2^{h,n,\theta}, \mu(c) \Delta E^{A,n,\theta}_{u2}-\sigma E^{A,n,\theta}_{u2}+ \frac{\partial E^{A,n,\theta}_{p}}{\partial y})_k & \\
 \leq \mid \tau_1'\mid (\sum_{k=1}^{n_{el}}(\mu_u B_{2k}' + \mid \sigma \mid B_{1k}'+ B_{3k}')) \|R_2^{h,n,\theta}\| \hspace{7 mm} \\
 \leq  \mid \tau_1 \mid (\sum_{k=1}^{n_{el}} \bar{B}_{2k}) \|R_2^{h,n,\theta}\| \hspace{2mm}(say) \hspace{30 mm}\\
\end{split}
\end{equation}
where $\bar{B}_{1k}$= $(\mu_u B_{2k} + \mid \sigma \mid B_{1k}+ B_{3k})$ and $\bar{B}_{2k}$= $(\mu_u B_{2k}' + \mid \sigma \mid B_{1k}'+ B_{3k}') $
and
\begin{multline}
\sum_{k=1}^{n_{el}} \tau_3'(R_4^{h,n,\theta}+d_4, \bigtriangledown \cdot \tilde{\bigtriangledown} E^{A,n,\theta}_{c} + \textbf{u} \cdot \bigtriangledown E^{A,n,\theta}_{c} - \alpha E^{A,n,\theta}_{c})_k \\
=\sum_{k=1}^{n_{el}} \tau_3'(R_4^{h,n,\theta},\bigtriangledown \cdot \tilde{\bigtriangledown} E^{A,n,\theta}_{c} + \textbf{u} \cdot \bigtriangledown E^{A,n,\theta}_{c} - \alpha E^{A,n,\theta}_{c})_k) + \sum_{k=1}^{n_{el}} \tau_3'(d_4, \bigtriangledown \cdot \tilde{\bigtriangledown} E^{A,n,\theta}_{c} \\
+ \textbf{u} \cdot \bigtriangledown E^{A,n,\theta}_{c} - \alpha E^{A,n,\theta}_{c})_k \hspace{70 mm}\\
\leq (\mid \tau_3'\mid + \frac{\tau_3'^2}{dt-\mid \tau_3'\mid})(\sum_{k=1}^{n_{el}}(D_{1m} B_{7k}+D_{2m} B_{7k}' + D_{u1} B_{8k}+D_{u2} B_{8k}' + \mid \alpha \mid B_{6k} )) \|R_4^{h,n,\theta}\| \hspace{21mm} \\
\leq \frac{ \mid \tau_3 \mid T_0}{T_0-C_{\tau_3}}(\sum_{k=1}^{n_{el}} \bar{B}_{4k}) \|R_4^{h,n,\theta}\| \hspace{2mm} (say) \hspace{65mm}
\end{multline}
where $\bar{B}_{4k}$= $(D_{1m} B_{7k}+D_{2m} B_{7k}' + D_{u1} B_{8k}+D_{u2} B_{8k}' + \mid \alpha \mid B_{6k} )$.\\
Finally
\begin{equation} 
\begin{split}
\sum_{k=1}^{n_{el}} \tau_2'(R_3^{h,n,\theta}, \bigtriangledown \cdot E^{A,n,\theta}_{\textbf{u}})_k & = \sum_{k=1}^{n_{el}} \tau_2'(\bigtriangledown \cdot e^{n,\theta}_{\textbf{u}}, \bigtriangledown \cdot E^{A,n,\theta}_{\textbf{u}})_k-\sum_{k=1}^{n_{el}} \tau_2'(\bigtriangledown \cdot E^{I,n,\theta}_{\textbf{u}}, \bigtriangledown \cdot E^{A,n,\theta}_{\textbf{u}})_k\\
& \leq C_{\tau_2}(1+\epsilon_2+\epsilon_3)\| e^{n,\theta}_{u1}\|_1^2+ C_{\tau_2}(1+ \frac{1}{\epsilon_2}+\frac{1}{\epsilon_3})\| e^{n,\theta}_{u2}\|_1^2 \\
& \quad C_{\tau_2} h^2 \sum_{i=1}^2(\frac{1+\theta}{2} \|u_i^{n+1}\|_2 + \frac{1-\theta}{2} \|u_i^{n}\|_2)\\
& \leq C_{\tau_2}\{(1+\epsilon_2+\epsilon_3)\| e^{n,\theta}_{u1}\|_1^2+ (1+ \frac{1}{\epsilon_2}+\frac{1}{\epsilon_3})\| e^{n,\theta}_{u2}\|_1^2+ C_2 h^2\}
\end{split}
\end{equation}
Applying assumption $\textbf{(iv)}$ on $\|u_i^n\|_2$ for $n=0,1,2,...,N$ we have the constant $C_2$ in the last line of (100).
Now this completes finding bounds for each term in the $RHS$ of (89). Therefore our next work is to combine all the results into equation (89). Putting common terms all together in the left hand side and multiplying them by $2$ and  then integrating both sides over $(t^n,t^{n+1})$ for $n=0,...,(N-1)$ , we will finally have
\begin{multline}
\|e_c^N\|^2+ \sum_{n=0}^{N-1} \int_{t^n}^{t^{n+1}} \{(2\mu_l-\epsilon_1-2C_{\tau_2}(1+\epsilon_2+\epsilon_3)) (\|e_{u1}^{n,\theta}\|^2+ \|\frac{\partial e_{u1}^{n,\theta}}{\partial x}\|^2 + \|\frac{\partial e_{u1}^{n,\theta}}{\partial y}\|^2) \\
+(2\mu_l-\epsilon_1-2C_{\tau_2}(1+ \frac{1}{\epsilon_2}+\frac{1}{\epsilon_3})) (\|e_{u2}^{n,\theta}\|^2+ \|\frac{\partial e_{u2}^{n,\theta}}{\partial x}\|^2 + \|\frac{\partial e_{u2}^{n,\theta}}{\partial y}\|^2)+\\
(2\sigma-\epsilon_1) \|e_p^{n,\theta}\|^2 +(2 D_{\alpha}-\epsilon_1)(\|e_{c}^{n,\theta}\|^2+
 \|\frac{\partial e_{c}^{n,\theta}}{\partial x}\|^2 + \|\frac{\partial e_{c}^{n,\theta}}{\partial y}\|^2)\} dt\\
 \leq \frac{C_2 h^2}{ \epsilon_1} \sum_{n=0}^{N-1} \int_{t^n}^{t^{n+1}}(\|R_1^{h,n,\theta}\|^2+\|R_2^{h,n,\theta}\|^2+\|R_3^{h,n,\theta}\|^2+\|R_4^{h,n,\theta}\|^2)dt +  \hspace{17 mm}\\
 + 2 \mid \tau_1 \mid \{(\sum_{k=1}^{n_{el}} \bar{B}_{1k}) \sum_{n=0}^{N-1} \int_{t^n}^{t^{n+1}} \|R_1^{h,n,\theta}\| +  (\sum_{k=1}^{n_{el}} \bar{B}_{2k}) \sum_{n=0}^{N-1} \int_{t^n}^{t^{n+1}} \|R_2^{h,n,\theta}\|\} dt  \\
  + \frac{4 \mid \tau_3 \mid}{T_0-C_{\tau_3}} (\sum_{k=1}^{n_{el}}(B_{6k}+\bar{B}_{4k}) \sum_{n=0}^{N-1} \int_{t^n}^{t^{n+1}} \|R_4^{h,n,\theta}\| dt +  \sum_{n=0}^{N-1} \int_{t^n}^{t^{n+1}} C_{\tau_2}C_2 h^2  \hspace{5 mm} \\
 +(\sum_{k=1}^{n_{el}}B_{6k})\sum_{n=0}^{N-1} \int_{t^n}^{t^{n+1}}\|TE^{n,\theta}\| dt \hspace{60 mm} \\
 \leq \bar{C}(\textbf{R}^h)(h^2+ dt^{2}) \hspace{75 mm}
\end{multline}
Choose the arbitrary parameters in such a way that all the coefficients in the left hand side can be made positive.
Then taking minimum over the coefficients in the left hand side let us divide both sides by them. Using backward Euler time discretisation scheme and its associated property (18) and the fact that $\tau_1, \tau_3$ are of order $h^2$, we have arrived at the above $aposteriori$ estimate (101), which does not depend upon exact solution. It shows that the method is second order accurate in space.

\section{Coupling of the Stokes-Brinkman/Transport equations through interfaces}
This section presents a brief study on implementing stabilized $ASGS$ method on coupled Stokes-Brinkman/Transport model with interface conditions. Here we have considered the domain $\Omega$ be partitioned into two sub-domains viz. $\Omega^S$ and $\Omega^B$ where the fluid flow in $\Omega^S$ is governed by the Stokes equation and in $\Omega^B$ the porous media flow obeys the Brinkman model. Let $\Gamma$ denote the interface and $\partial \Omega^S$ and $\partial\Omega^B$ be the boundaries of $\Omega^S$ and $\Omega^B$ respectively. Now $\Gamma^l=\partial\Omega^l \setminus \Gamma$ ($l=S,B$). \vspace{1mm}\\
Let us first mention here the system of fluid flow and mass transport equations in $\Omega^S$: Find $\textbf{u}^S$: $\Omega$ $\times$ (0,T) $\rightarrow R^2$ , $p^S$: $\Omega \times$ (0,T) $\rightarrow R$ and $c^S$: $\Omega \times$ (0,T) $\rightarrow R$ such that,
\begin{equation}
\begin{split}
- \mu^S(c) \Delta \textbf{u}^S + \bigtriangledown p^S & = \textbf{f}_1^S \hspace{2mm} in \hspace{2mm} \Omega^S \times (0,T) \\
\bigtriangledown \cdot \textbf{u}^S &= 0 \hspace{2mm} in \hspace{2mm} \Omega^S \times (0,T) \\
\frac{\partial c^S}{\partial t}- \bigtriangledown \cdot \tilde{\bigtriangledown} c^S + \textbf{u}^S \cdot \bigtriangledown c^S + \alpha c^S & = g^S \hspace{2mm} on \hspace{2mm} \Omega^S \times (0,T) \\
\textbf{u}^S  &= \textbf{0} \hspace{2mm} on \hspace{2mm} \partial\Omega^S \times (0,T) \\
\textbf{u}^S &= \textbf{u}_0^S \hspace{2mm} at \hspace{2mm} t=0 \\
\tilde{\bigtriangledown} c^S \cdot \textbf{n} &= 0 \hspace{2mm} in \hspace{2mm}\partial \Omega^S \times (0,T) \\
c^S & = c_0^S \hspace{2mm} at \hspace{2mm} t=0\\
\end{split}
\end{equation}
and the same set of equations in $\Omega^B$ is: Find $\textbf{u}^B$: $\Omega$ $\times$ (0,T) $\rightarrow R^2$ , $p^B$: $\Omega \times$ (0,T) $\rightarrow R$ and $c^B$: $\Omega \times$ (0,T) $\rightarrow R$ such that,
\begin{equation}
\begin{split}
- \mu^B(c) \Delta \textbf{u}^B + \bigtriangledown p^B & = \textbf{f}_1^B \hspace{2mm} in \hspace{2mm} \Omega^B \times (0,T) \\
\bigtriangledown \cdot \textbf{u}^B &= f_2^B \hspace{2mm} in \hspace{2mm} \Omega^B \times (0,T) \\
\phi \frac{\partial c^B}{\partial t}- \bigtriangledown \cdot \tilde{\bigtriangledown} c^B + \textbf{u}^B \cdot \bigtriangledown c^B + \alpha c^B & = g^B \hspace{2mm} on \hspace{2mm} \Omega^S \times (0,T) \\
\textbf{u}^B &= 0 \hspace{2mm} on \hspace{2mm} \partial\Omega^B \times (0,T) \\
\textbf{u}^B &= \textbf{u}_0^B \hspace{2mm} at \hspace{2mm} t=0 \\
\tilde{\bigtriangledown} c^B \cdot \textbf{n} &= 0 \hspace{2mm} in \hspace{2mm}\partial \Omega^B \times (0,T) \\
c^B & = c_0^B \hspace{2mm} at \hspace{2mm} t=0\\
\end{split}
\end{equation}
where $(\textbf{u}^S,p^S)$ and $(\textbf{u}^B,p^B)$ are the pairs of Stokes and Brinkman velocities and pressure respectively. As well $c^S$ and $c^B$ denote concentration of the solute in $\Omega^S$ and $\Omega^B$ respectively. $\mu^S(c)$ and $\mu^B(c)$ are the Stokes dynamic viscosity and Brinkman effective viscosity respectively Now the interface conditions on $\Gamma$ are as follows:
\begin{equation}
\textbf{u}^S \cdot \textbf{n}^S + \textbf{u}^B \cdot \textbf{n}^B = 0
\end{equation}
\begin{equation}
-\mu^S(c) \partial_n \textbf{u}^S \cdot \textbf{n}^S + p^S = -\mu^B(c) \partial_n \textbf{u}^B \cdot \textbf{n}^B + p^B
\end{equation}
\begin{equation}
\mu^S(c) \partial_n \textbf{u}^S \cdot \textbf{t} + \frac{\alpha}{\sqrt{\sigma}} \textbf{u}^S \cdot \textbf{t} =0
\end{equation}
\begin{equation}
c^S= c^B
\end{equation}
\begin{equation}
\tilde{\bigtriangledown}c^S \cdot \textbf{n}^S + \tilde{\bigtriangledown}c^B \cdot \textbf{n}^B =0
\end{equation}
where $\textbf{n}^S$ and $\textbf{n}^B$ are the outward normals to $\Omega^S$ and $\Omega^B$ respectively. It is quite obvious to observe that $\textbf{n}^S=-\textbf{n}^B$. Here $(104),(107),(108)$ represent continuity conditions of normal velocities and concentration. Where as $(105)$ enforces continuity of normal stresses and $(106)$ is the Beavers-Joseph-Saffman condition $\cite{RefG}$. \vspace{1mm}\\
Let us mention the corresponding spaces for both the sub-problems here: Let $V^l=H_0^1(\Omega^l)$ and $Q^l=L^2(\Omega^l)$ (for $l=S,B$). $V_h^l$ and  $Q_h^l$ be the corresponding finite dimensional subspaces of  $V_h^l$ and $Q_h^l$ respectively for $l=S,B$. Following the previous steps the stabilized formulation for Stokes sub-problem is to find $\textbf{U}_h^S=(\textbf{u}_h^S,p_h^S,c_h^S)$: J $ \rightarrow $ $V_h^S \times V_h^S \times Q_h^S \times V_h^S$ such that $\forall \textbf{V}_h^S =(\textbf{v}_h^S,q_h^S,d_h^S) $ $\in$ $V_h^S \times V_h^S \times Q_h^S \times V_h^S$
\begin{equation}
(M\partial_t \textbf{U}^S_h,\textbf{V}^S_h) + B^S_{ASGS}(\textbf{U}^S_h, \textbf{V}^S_h)  = L^S_{ASGS}(\textbf{V}^S_h) + \int_{\Gamma}(\mu^S(c) \partial_n \textbf{u}^S - p^S \textbf{n}^S) \cdot \textbf{v}_h^S d\Gamma
\end{equation}
and for Brinkman sub-problem the stabilized formulation is to find $\textbf{U}_h^B=(\textbf{u}_h^B,p_h^B,c_h^B)$: J $ \rightarrow $ $V_h^B \times V_h^B \times Q_h^B \times V_h^B$ such that $\forall \textbf{V}_h^B =(\textbf{v}_h^B,q_h^B,d_h^B) $ $\in$ $V_h^B \times V_h^B \times Q_h^B \times V_h^B$
\begin{equation}
(M\partial_t \textbf{U}^B_h,\textbf{V}^B_h) + B^B_{ASGS}(\textbf{U}^B_h, \textbf{V}^B_h)  = L^B_{ASGS}(\textbf{V}^B_h) + \int_{\Gamma}(\mu^B(c) \partial_n \textbf{u}^B - p^B \textbf{n}^B) \cdot \textbf{v}_h^B d\Gamma
\end{equation}
Now applying the interface conditions it is easy to conclude that the stabilized $ASGS$ formulation for coupled Stokes-Brinkman/Transport model is to find $\textbf{U}_h^S=(\textbf{u}_h^S,p_h^S,c_h^S)$: J $ \rightarrow $ $V_h^S \times V_h^S \times Q_h^S \times V_h^S$ and $\textbf{U}_h^B=(\textbf{u}_h^B,p_h^B,c_h^B)$: J $ \rightarrow $ $V_h^B \times V_h^B \times Q_h^B \times V_h^B$ such that  $\forall \textbf{V}_h^S =(\textbf{v}_h^S,q_h^S,d_h^S) $ $\in$ $V_h^S \times V_h^S \times Q_h^S \times V_h^S$ and $\forall \textbf{V}_h^B =(\textbf{v}_h^B,q_h^B,d_h^B) $ $\in$ $V_h^B \times V_h^B \times Q_h^B \times V_h^B$ 
\begin{multline}
(M\partial_t \textbf{U}^S_h,\textbf{V}^S_h)+(M\partial_t \textbf{U}^B_h,\textbf{V}^B_h) + B^S_{ASGS}(\textbf{U}^S_h, \textbf{V}^S_h)+ B^B_{ASGS}(\textbf{U}^B_h, \textbf{V}^B_h)\\ + \frac{\alpha}{\sqrt{\sigma}}(\textbf{u}_h^S \cdot \textbf{t},\textbf{v}_h^S \cdot \textbf{t})= L^S_{ASGS}(\textbf{V}^S_h)+ L^B_{ASGS}(\textbf{V}^B_h)
\end{multline}

\begin{remark}
These terms in $(111)$ do not much differ from the general form of stabilized $ASGS$ formulation in $(8)$. Only one term consisting of the tangential component of Stokes velocity is extra in the coupled stabilized formulation $(111)$. Hence $apriori$ and $aposteriori$ error estimate results will be almost similar; only coefficients $C''$ for $apriori$ and $\bar{C}(\textbf{R}^h)$ for $aposteriori$ may sightly differ though the convergence rates for both that error estimations will be the same.
\end{remark}

\section{Numerical Experiment}
In this section we present a comparative study between standard Galerkin method and stabilized algebraic subgrid scale(ASGS) method as well as we have verified the convergence rate established theoretically  under stabilized method in the previous sections. We have considered three different models to work with: Coupled Stokes/Transport Model, Coupled Brinkman/Transport Model and Coupled Stokes-Brinkman with interface/Transport Model. \vspace{1mm}\\
For simplicity we have considered bounded square domain $\Omega$= (0,1) $\times$ (0,1). We have taken continuous piecewise linear finite element(P1) space into account for approximating velocity, pressure and concentration too. The expression of concentration dependent viscosity is taken from \cite{RefR}, which establishes that viscosity of a solvent depends upon concentration of the solute of a electrolyte solution. The proposed expression for viscosity is $\mu(c)=0.954 e^{27.93 \times 0.028 c}$. \vspace{1mm}\\
 Let us mention here the exact solutions for all three cases as follows: \vspace{1mm}\\
$\textbf{u}=(t sin^2(\pi x) sin(\pi y) cos(\pi y), -t sin(\pi x) cos(\pi x) sin^2(\pi y))$, \vspace{1mm} \\
 $p=t sin(2 \pi x) sin(2 \pi y)$ and $c= t x y (x-1)(y-1)$ \vspace{1mm} \\
Now in the following we mention the general expressions of the coefficients involved in the equations: \\
The reaction coefficients $\alpha=0.01$\\
The diffusion coefficients: $D_1=t^2(sin(\pi x))^4(sin(2 \pi y))^2$, $D_2$=$t^2 (sin(2 \pi x))^2$ $ (sin(\pi y))^4$\\
The stabilization parameters: $\tau_1=(4 \frac{\mu_l}{h^2}+ \sigma)^{-1}$, $\tau_2=4 \mu_l h $ and $\tau_3=19(\frac{9}{4 h^2}+ \frac{3}{2h} + \alpha )^{-1}$ , where $\mu_l= 0.954 e^{27.93 \times 0.028 \times 0.0625}$ \vspace{2mm}\\
\textbf{1.Coupled Stokes/ Transport Model} 
In this case $\sigma=0$ and porosity $\phi=1$. Hence the expression for stabilization parameters are changed accordingly.\vspace{1mm}\\
Table 1 and table 2 present the error in $V$ norm (which is standard norm on the space $\textbf{V}$, introduced in the section 2.1 and is defined in the initial part of section 3.3) and order of convergence under  Galerkin method and ASGS method respectively for this case. These tables are clearly showing that both the methods perform equally well. \vspace{2mm}\\
\textbf{2.Coupled Brinkman/ Transport Model}
In this case $\sigma$ is non-zero and in particular we consider $\sigma=1$. The value of porosity is taken to be 2. In Brinkman flow problem we will deal with effective viscosity $\mu_B$. According to \cite{RefS} the effective viscosity $\mu_B$ and viscosity $\mu$ is related through $\sigma^2 = \frac{\mu_B}{\mu}$. Hence both are considered same since $\sigma=1$. Here the stabilization parameters take the general form. \vspace{1mm}\\
Table 3 and table 4 present the error in $V$ norm and order of convergence under Galerkin method and ASGS method respectively. The tables represent that both methods perform equally well.\vspace{2mm}\\ 
\textbf{3.Coupled Stokes-Brinkman/ Transport Model} Here $(\textbf{u},p,c)$ take values $(\textbf{u}^S,p^S,c^S)$ in $\Omega^S$ and $(\textbf{u}^B,p^B,c^B)$ in $\Omega^B$. As mentioned in the previous case $\mu^S(c)=\mu^B(c)$ for $\sigma=1$ on $\Omega^B$. The stabilization parameters on  $\Omega^S$: $\tau_1^S=\frac{h^2}{4 \mu_l}$, $\tau_2^S=4 \mu_l h $ and $\tau_3^S=19(\frac{9}{4 h^2}+ \frac{3}{2h} + \alpha )^{-1}$ and on $\Omega^B$: $\tau_1^B=(4 \frac{\mu_l}{h^2}+ \sigma)^{-1}$, $\tau_2^B=4 \mu_l h $ and $\tau_3^B=19(\frac{9}{4 h^2}+ \frac{3}{2h} + \alpha )^{-1}$ ,for given $\mu_l= 0.954 e^{27.93 \times 0.028 \times 0.0625}$. The value of porosity $\phi$ is 2 on $\Omega^B$. \vspace{1mm}\\
Table 5 and table 6 show the error in $V$ norm and order of convergence under Galerkin method and $ASGS$ method respectively for coupled Stokes-Brinkman/ Transport model. These tables represent that stabilized $ASGS$ method   performs well, whereas the convergence rate under the Galerkin method oscillates.

\begin{table}[]
    \centering
    \begin{tabular}{||c c c||}
    \hline 
    Grid & Error in $H^1$ norm &  Order of convergence \\
    \hline \hline
      10$\times$ 10      &   0.950341        &                       \\
      \hline
      20 $\times$ 20     &  0.27489   & 1.78959 \\
      \hline
      40 $\times$ 40      &  0.0635241   & 2.11348  \\
      \hline
      80 $\times$ 80     &  0.0190773   & 1.73544  \\
      \hline
      160 $\times$ 160    &  0.0053094   & 1.84524  \\
      \hline
    \end{tabular}
\caption{Error and Order of convergence obtained in $V$ norm under Galerkin method for Stokes/Transport Model}
\end{table}

\begin{table}[]
    \centering
    \begin{tabular}{||c c c||}
    \hline 
    Grid & Error in $H^1$ norm &  Order of convergence \\
    \hline \hline
      10 $\times$ 10     &   0.200567        &                       \\
      \hline
      20 $\times$ 20     &  0.0661861  &  1.59948 \\
      \hline
      40 $\times$ 40     &  0.0162986   & 2.02178  \\
      \hline
      80 $\times$ 80     &  0.00434506   & 1.9073  \\
      \hline
      160 $\times$ 160    &  0.00113881   & 1.93185  \\
      \hline
    \end{tabular}
\caption{Error and Order of convergence obtained in $V$ norm under ASGS method for Stokes/Transport Model}
\end{table}

\begin{table}[]
    \centering
    \begin{tabular}{||c c c||}
    \hline 
    Grid & Error in $H^1$ norm &  Order of convergence \\
    \hline \hline
      10 $\times$ 10     &   0.953771        &                       \\
      \hline
      20 $\times$ 20     &  0.275123  &  1.79357 \\
      \hline
      40 $\times$ 40     &  0.0635331   & 2.1145  \\
      \hline
      80 $\times$ 80     &  0.0190837  & 1.73516  \\
      \hline
      160 $\times$ 160    &  0.00531838   & 1.84328  \\
      \hline
    \end{tabular}
\caption{Error and Order of convergence obtained in $V$ norm under Galerkin method for Brinkman/Transport Model}
\end{table}

\begin{table}[]
    \centering
    \begin{tabular}{||c c c||}
    \hline 
    Grid & Error in $H^1$ norm &  Order of convergence \\
    \hline \hline
      10 $\times$ 10     &   0.201123        &                       \\
      \hline
      20 $\times$ 20     &  0.0661698  &   1.60383 \\
      \hline
      40 $\times$ 40      &  0.0162947   &  2.02177 \\
      \hline
      80 $\times$ 80     &  0.00435187  &  1.90469 \\
      \hline
      160 $\times$ 160    &  0.0011484   &  1.92201  \\
      \hline
    \end{tabular}
\caption{Error and Order of convergence obtained in $V$ norm under ASGS method for Brinkman/Transport Model }
\end{table}

\begin{table}[]
    \centering
    \begin{tabular}{||c c c||}
    \hline 
    Grid & Error in $H^1$ norm &  Order of convergence \\
    \hline \hline
      10 $\times$ 10     &   0.000603299        &                       \\
      \hline
      20 $\times$ 20     &  0.00057062  &  0.080342 \\
      \hline
      40 $\times$ 40     &  0.000139967   & 2.02744  \\
      \hline
      80 $\times$ 80     &  3.04023$e^{-5}$   &  2.20284  \\
      \hline
      160 $\times$ 160    &  2.06788$e^{-5}$   &  0.556032  \\
      \hline
    \end{tabular}
\caption{Error and Order of convergence obtained in $V$ norm under Galerkin method for coupled Stokes-Brinkman/Transport Model}
\end{table}

\begin{table}[]
    \centering
    \begin{tabular}{||c c c||}
    \hline 
    Grid & Error in $H^1$ norm &  Order of convergence \\
    \hline \hline
      10 $\times$ 10     &   0.000479849        &                       \\
      \hline
      20  $\times$ 20    &  0.00015364  &  1.64303 \\
      \hline
      40 $\times$ 40     &  4.25356$e^{-5}$   & 1.85281  \\
      \hline
      80  $\times$ 80    &  1.1035$e^{-5}$   & 1.94664  \\
      \hline
      160 $\times$ 160    &  2.80343$e^{-6}$   & 1.97685  \\
      \hline
    \end{tabular}
\caption{Error and Order of convergence obtained in $V$ norm under ASGS method for coupled Stokes-Brinkman/Transport Model}
\end{table}

\begin{figure}
\centering
\begin{minipage}{.49\textwidth}
  \centering
  \includegraphics[width=\textwidth]{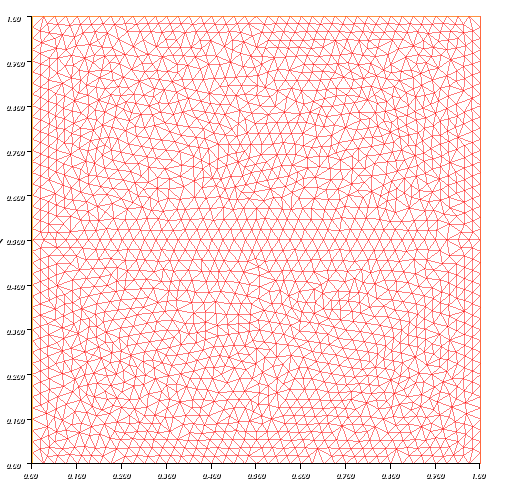} 
\end{minipage}
\begin{minipage}{.49\textwidth}
  \centering
  \includegraphics[width=\textwidth]{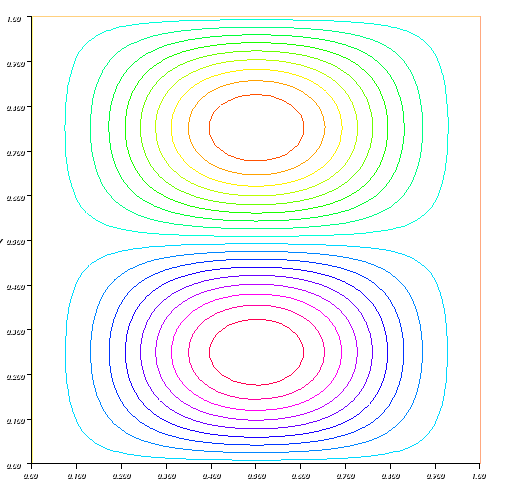}  
\end{minipage}
\caption{Mesh for 40 $\times$ 40 grid points and Horizontal velocity plot respectively for both Coupled Stokes/Transport Model and Coupled Brinkman/Transport Model }
\end{figure}

\begin{figure}
\centering
\begin{minipage}{.49\textwidth}
  \centering
  \includegraphics[width=\textwidth]{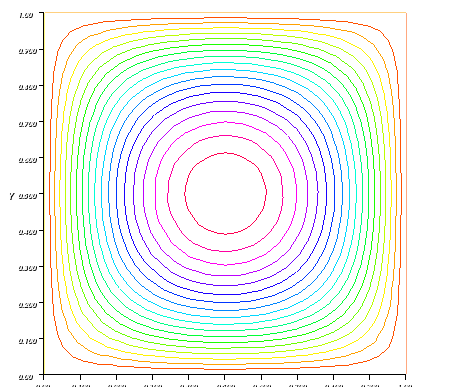} 
  \caption{Concentration plot for both Coupled Stokes/Transport Model and Coupled Brinkman/Transport Model}
\end{minipage}
\begin{minipage}{.49\textwidth}
  \centering
  \includegraphics[width=\textwidth]{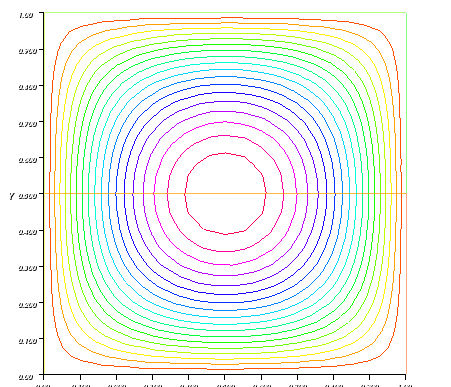}  
  \caption{Concentration plot for both Coupled Stokes-Brinkman/Transport Model with interface}
\end{minipage}
\end{figure}

\begin{figure}
\centering
\begin{minipage}{.49\textwidth}
  \centering
  \includegraphics[width=\textwidth]{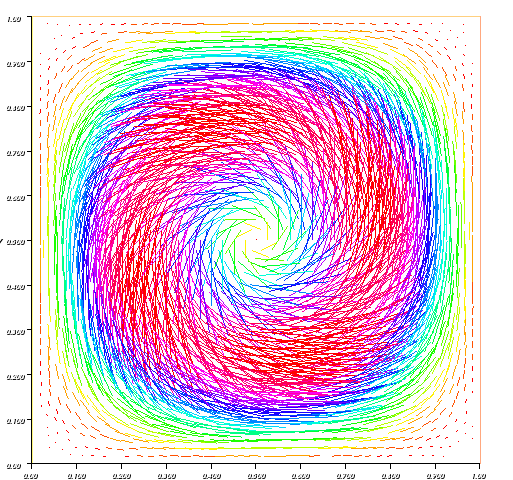} 
\end{minipage}
\begin{minipage}{.49\textwidth}
  \centering
  \includegraphics[width=\textwidth]{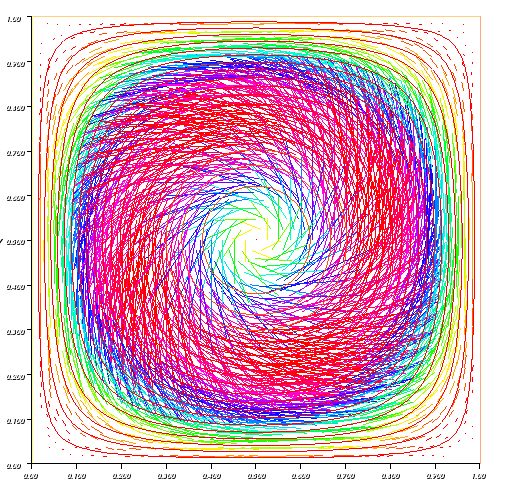}  
\end{minipage}
\caption{Velocity plot and Velocity concentration plot respectively for both Coupled Stokes/Transport Model and Coupled Brinkman/Transport Model}
\end{figure}

\begin{figure}
\centering
\begin{minipage}{.49\textwidth}
  \centering
  \includegraphics[width=\textwidth]{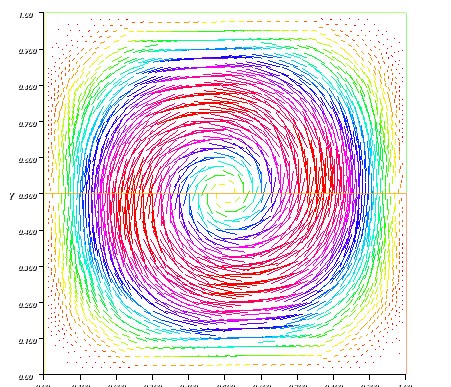} 
\end{minipage}
\begin{minipage}{.49\textwidth}
  \centering
  \includegraphics[width=\textwidth]{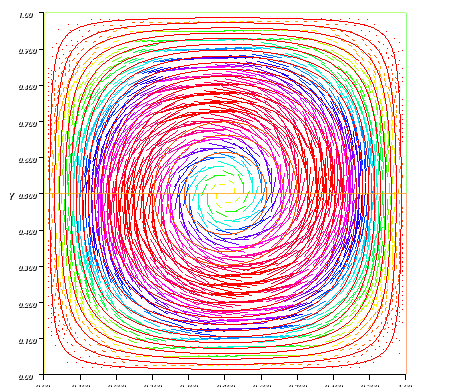}  
\end{minipage}
\caption{Velocity plot and Velocity concentration plot respectively for  Coupled Stokes-Brinkman/Transport Model}
\end{figure}

\begin{figure}
\centering
\begin{minipage}{.8\textwidth}
  \centering
  \includegraphics[width=\textwidth]{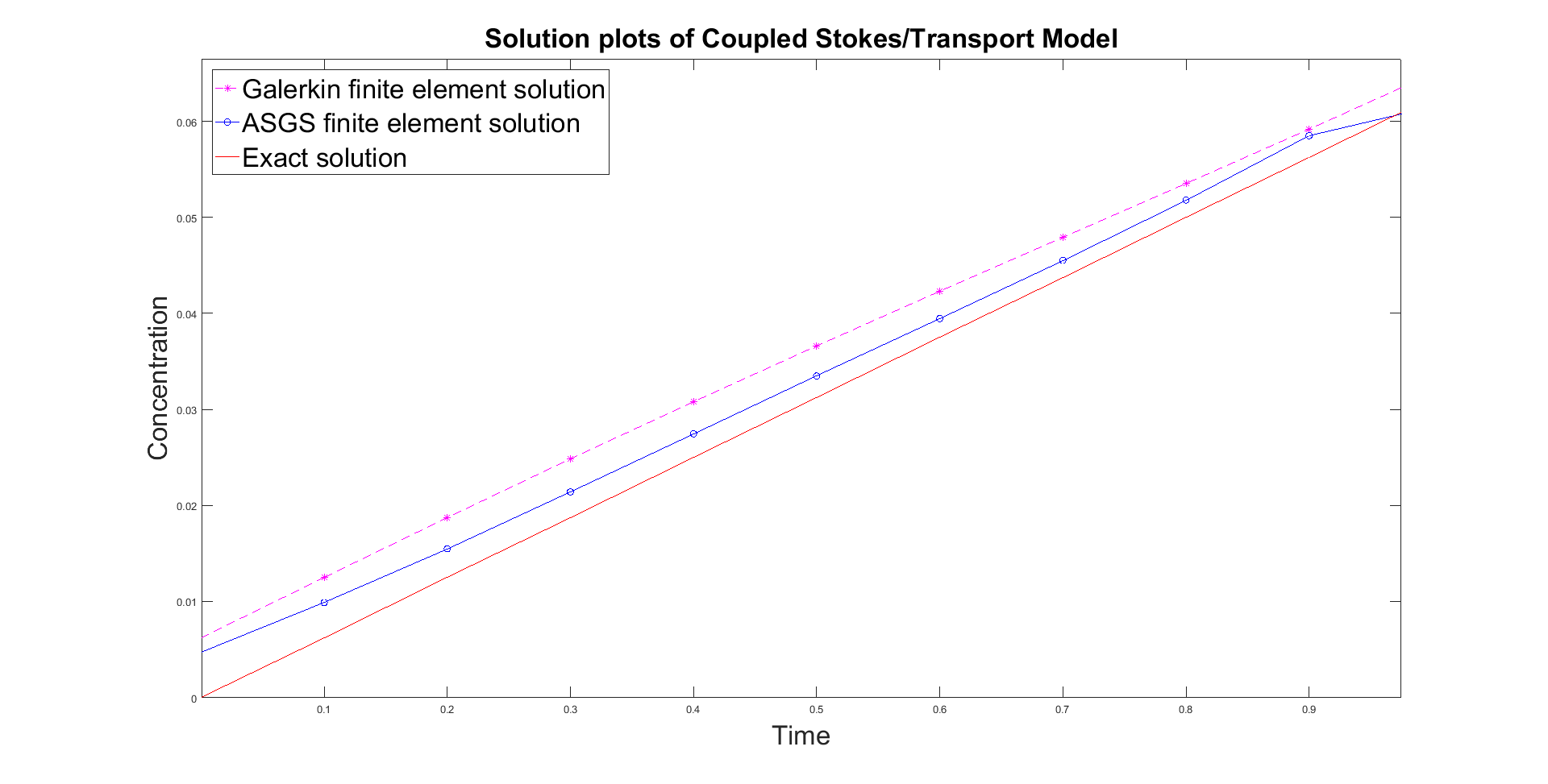} 
\end{minipage}
\begin{minipage}{.8\textwidth}
  \centering
  \includegraphics[width=\textwidth]{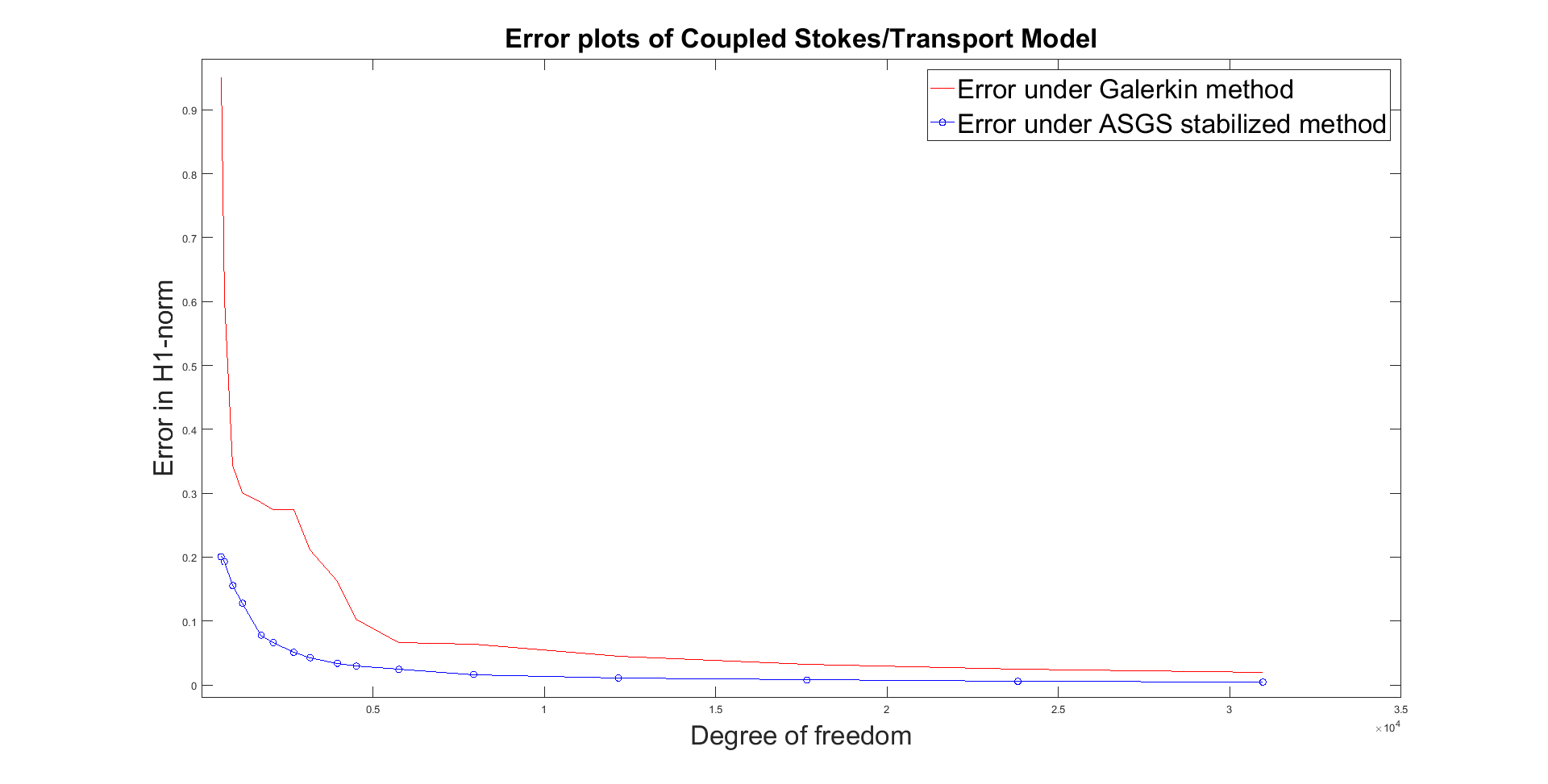}  
\end{minipage}
\begin{minipage}{.8\textwidth}
  \centering
  \includegraphics[width=\textwidth]{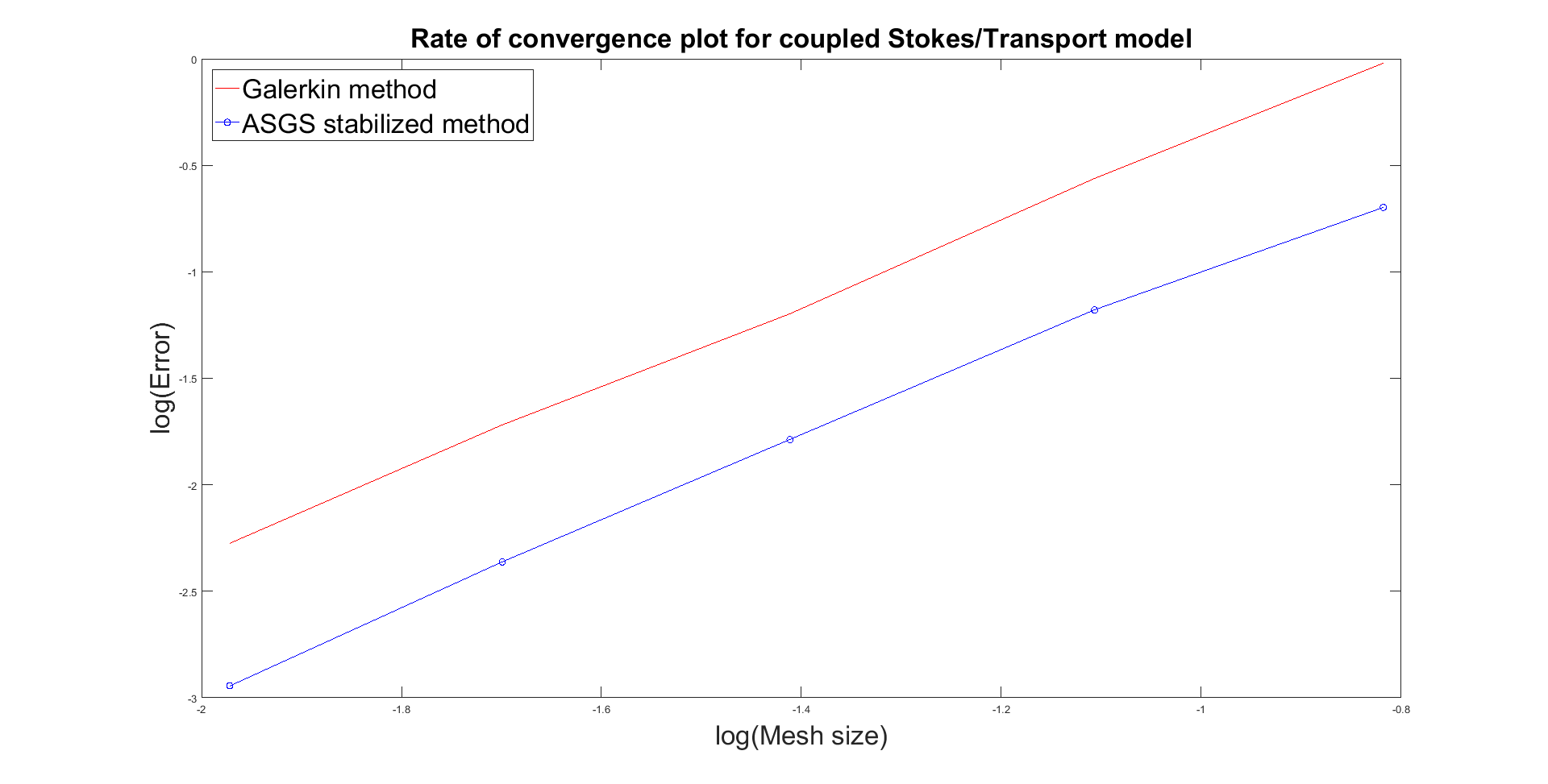}  
\end{minipage}
\caption{Comparison of solution plots at an arbitrary point (0.5,0.5) (upper); Error plot in $H1$-norm with respect to degrees of freedom (middle); Order of convergence plot (lower) for Coupled Stokes/Transport Model}
\end{figure}

\begin{figure}
\centering
\begin{minipage}{.8\textwidth}
  \centering
  \includegraphics[width=\textwidth]{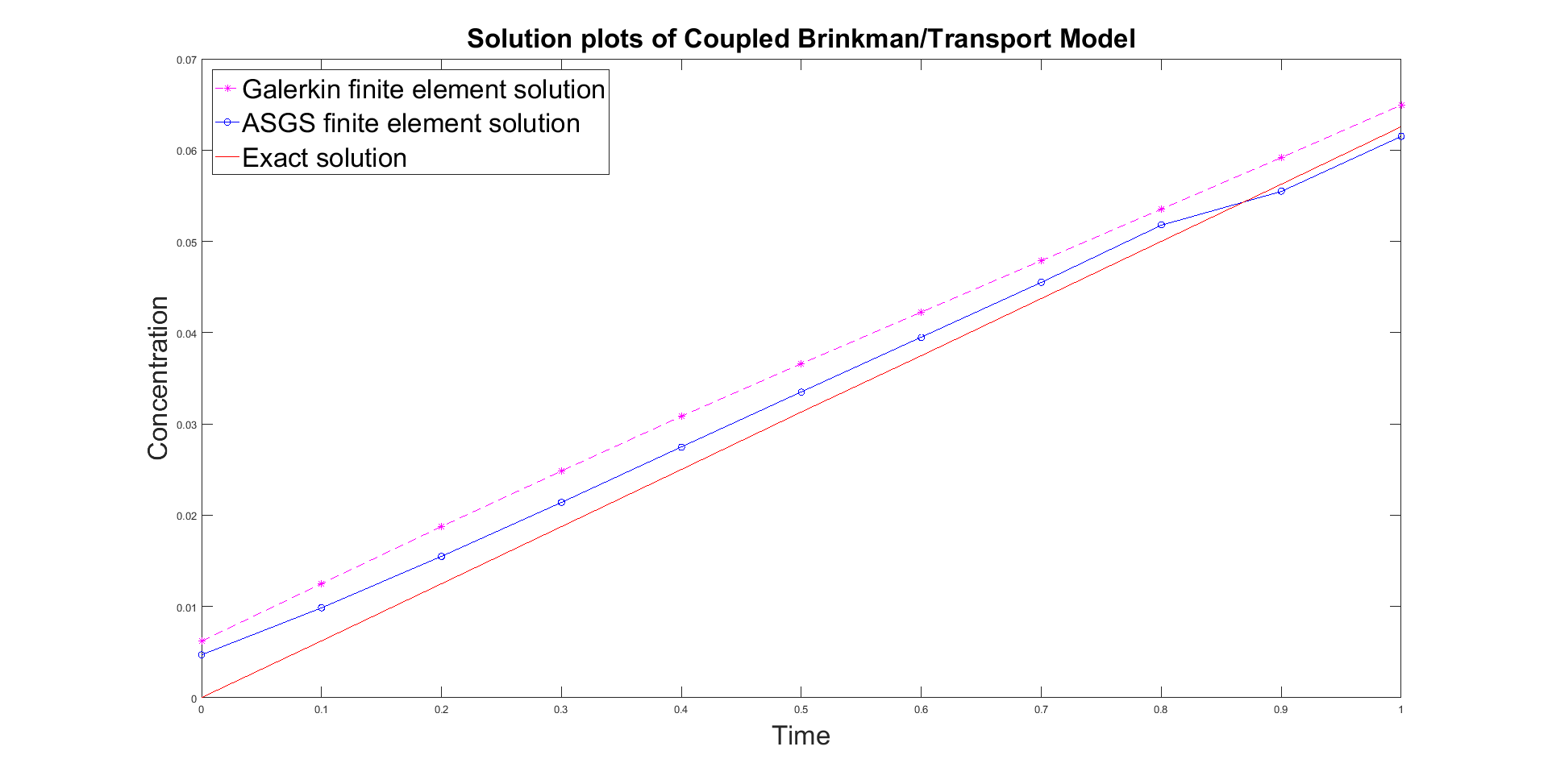} 
\end{minipage}
\begin{minipage}{.8\textwidth}
  \centering
  \includegraphics[width=\textwidth]{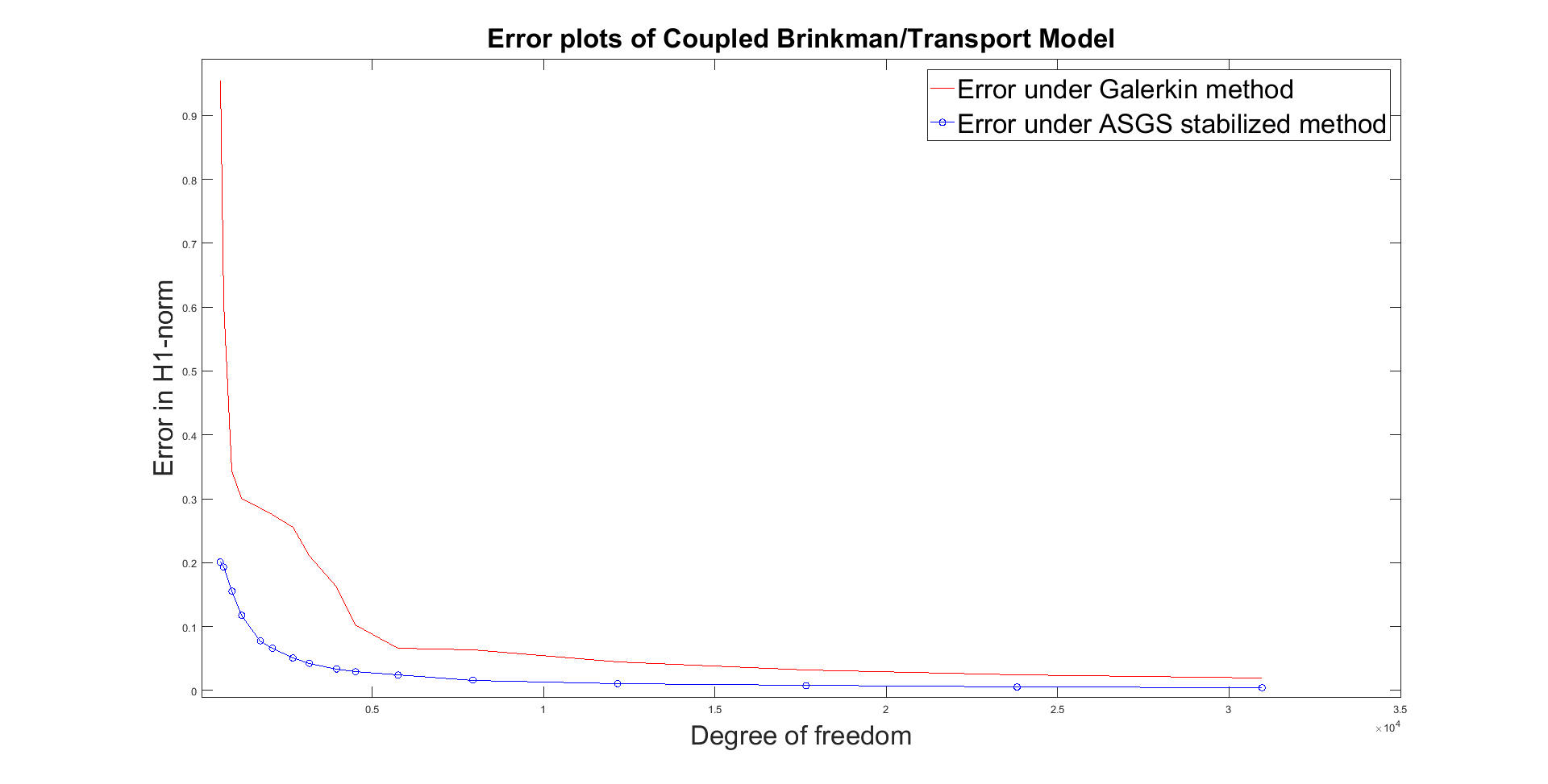}  
\end{minipage}
\begin{minipage}{.8\textwidth}
  \centering
  \includegraphics[width=\textwidth]{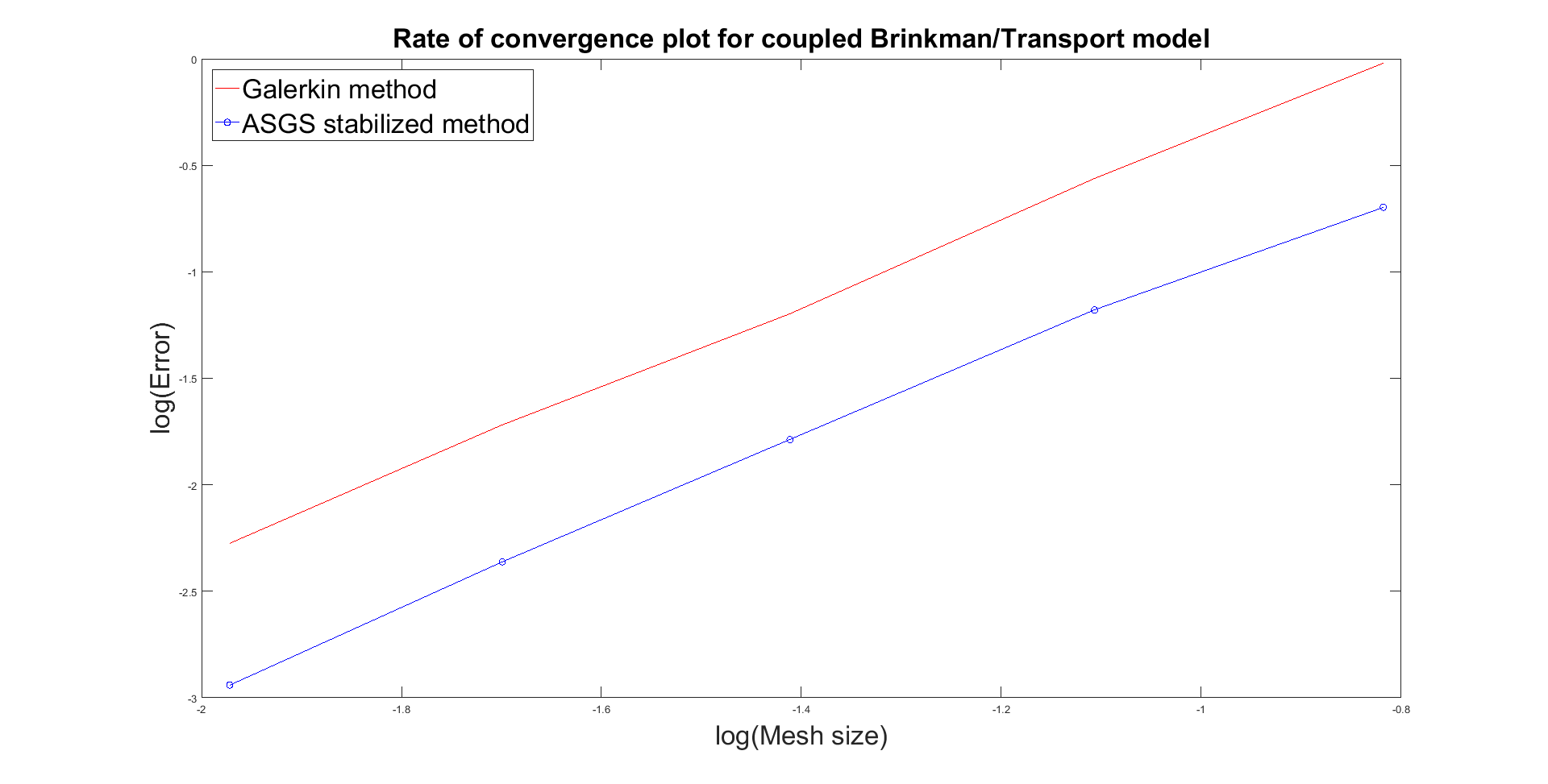}  
\end{minipage}
\caption{Comparison of solution plots at an arbitrary point (0.5,0.5) (upper); Error plot in $H1$-norm with respect to degrees of freedom (middle); Order of convergence plot (lower) for Coupled Brinkman/Transport Model}
\end{figure}

\begin{figure}
\centering
\begin{minipage}{.8\textwidth}
  \centering
  \includegraphics[width=\textwidth]{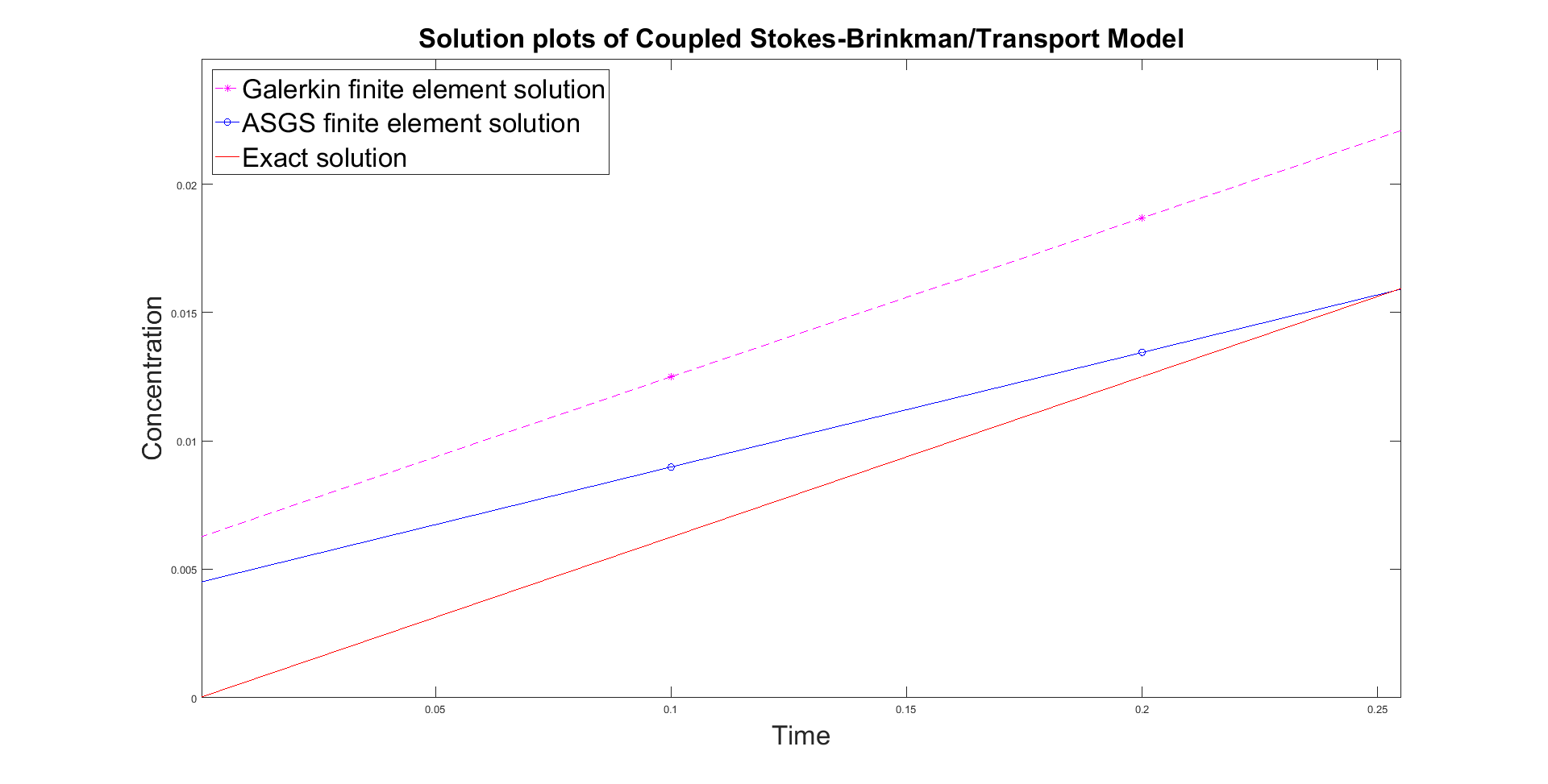} 
\end{minipage}
\begin{minipage}{.8\textwidth}
  \centering
  \includegraphics[width=\textwidth]{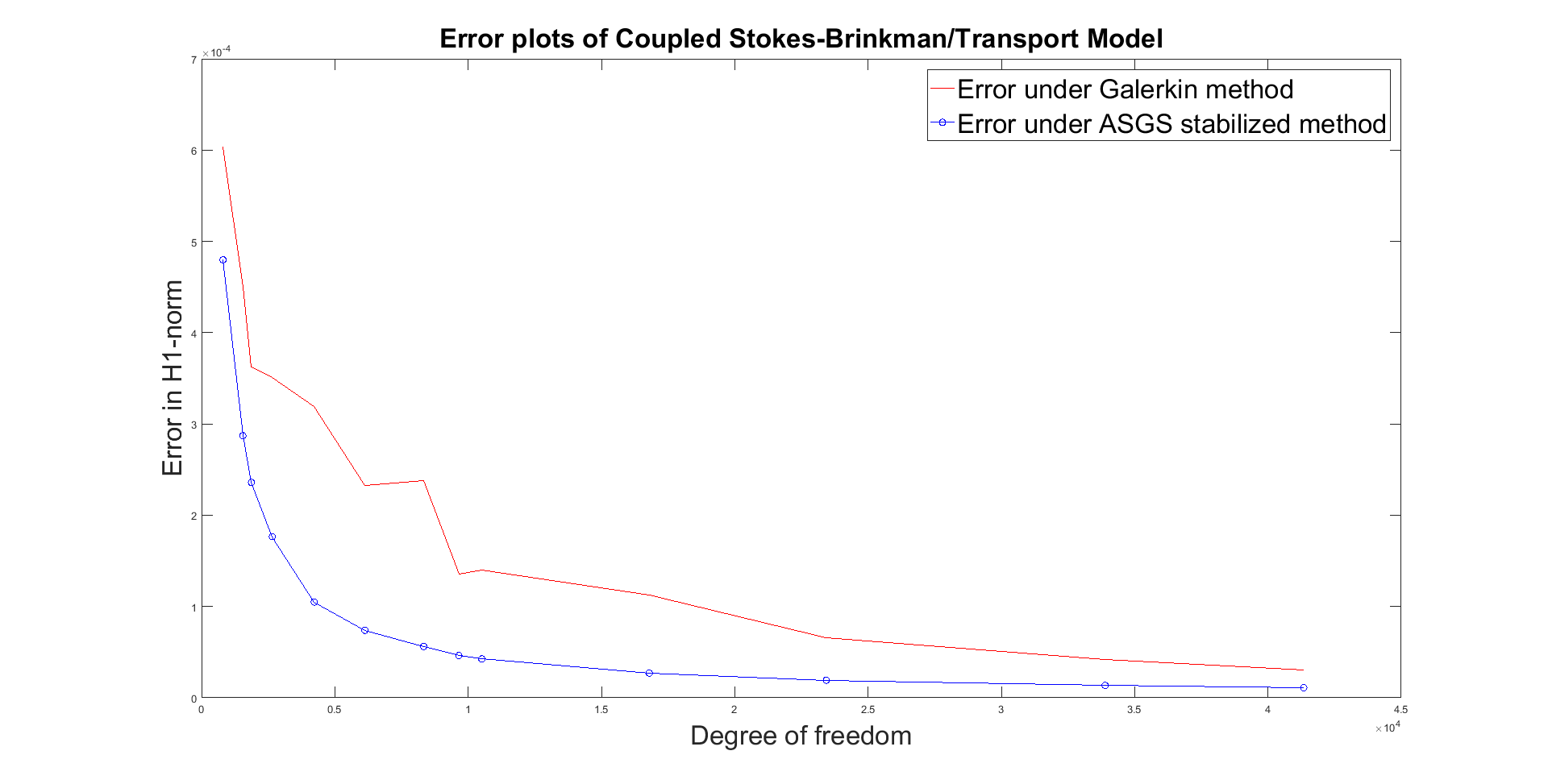}  
\end{minipage}
\begin{minipage}{.8\textwidth}
  \centering
  \includegraphics[width=\textwidth]{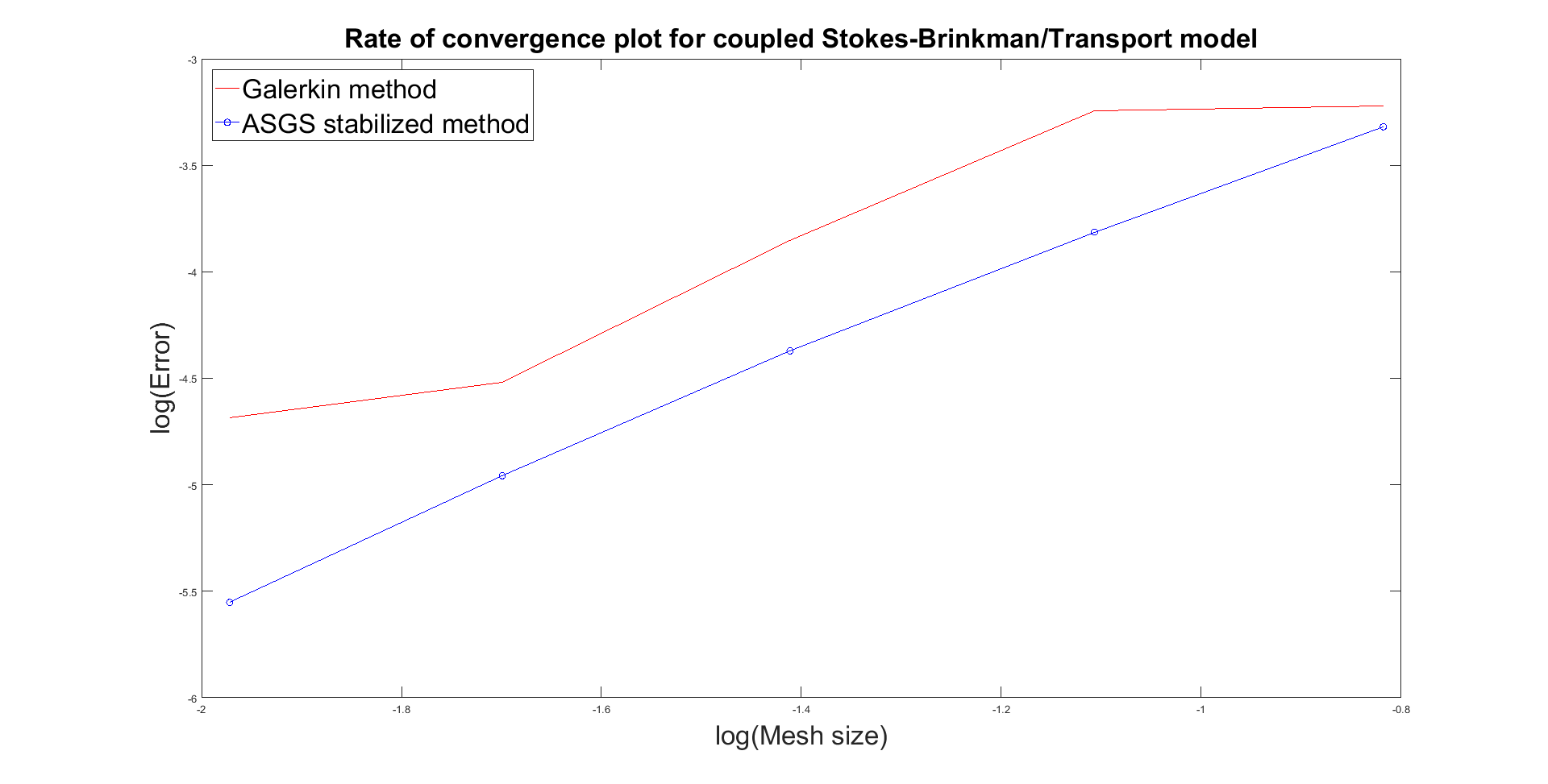}  
\end{minipage}
\caption{Comparison of solution plots at an arbitrary point (0.5,0.5) (upper); Error plot in $H1$-norm with respect to degrees of freedom (middle); Order of convergence plot (lower) for Coupled Stokes-Brinkman/Transport Model with interface conditions}
\end{figure}

\begin{figure}
\centering
\begin{minipage}{.4\textwidth}
  \centering
  \includegraphics[width=\textwidth]{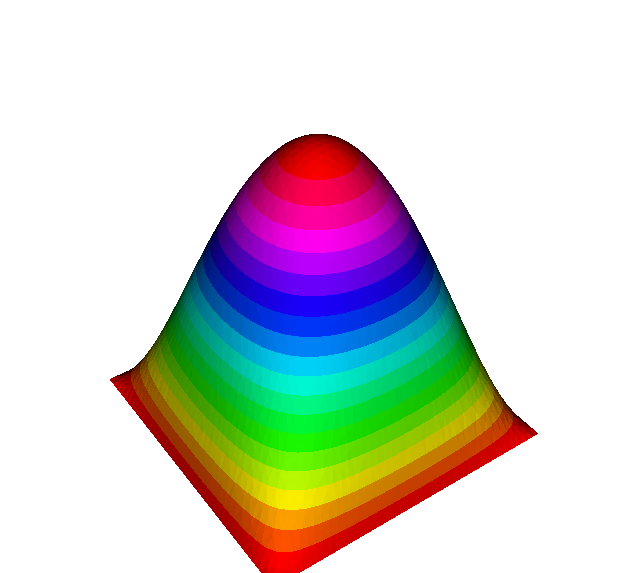} 
  \caption{Exact solution for 40 $\times$ 40 grid points}
\end{minipage}
\begin{minipage}{.4\textwidth}
  \centering
  \includegraphics[width=\textwidth]{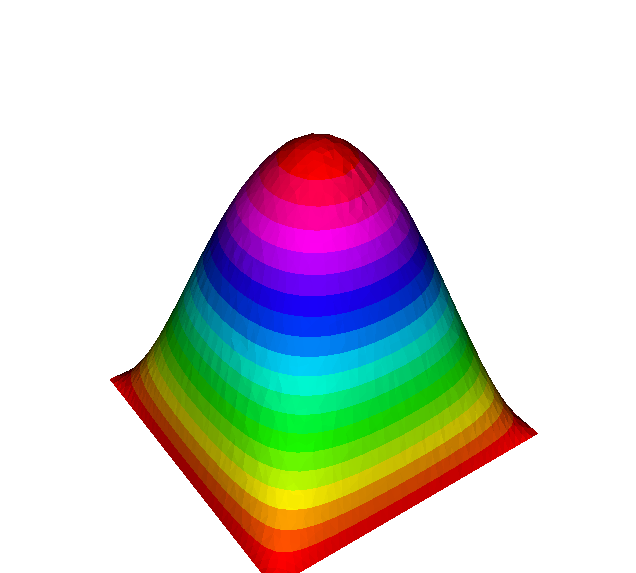}  
  \caption{ASGS solution for 40 $\times$ 40 grid points}
\end{minipage}
\end{figure}

\begin{remark}
The tables are showing that error under $ASGS$ method at each mesh size is turned out to be lesser than that of Galerkin method and for all three cases the order of convergence under $ASGS$ method is 2, which justifies theoretically established result.
\end{remark}

\begin{remark}
It is clear from figure 6,7, and 8 that $ASGS$ solution comparatively more fast converges to exact solution whereas the Galerkin solution converges slowly in first two cases and diverges in third case representing coupled Stokes-Brikman/Transport model.
\end{remark}

\begin{remark}
The error plots in $H1$-norm and order of convergence plot under Galerkin and $ASGS$ methods in figure 6,7,8 establish the more efficiency of the stabilized method in compared to Galerkin method. It shows that error under $ASGS$ method is much lesser than that of Galerkin method at the same mesh size and both are decreasing for finer mesh.
\end{remark}

\begin{remark}
Figure 9 and 10 present three dimensional view of exact solution and approximated solution derived under stabilized method. It is easily seen that the approximated solution is very much alike to the exact one.
\end{remark}

\section{Conclusion}
The paper presents $ASGS$ stabilized finite element analysis of two different aspects of Stokes-Brinkman fluid flow model strongly coupled with unsteady $VADR$ transport equation; one is unified way of considering the model and another is coupling system of equations through interface conditions. Whereas this paper in one hand elaborately derive both apriori and aposteriori error estimates, on other hand it highlights the way to prove existence and uniqueness of the solution of variational formulation. It is essential to mention that the norm employed for error estimation consists of the full norms corresponding to each variable belonging to their respective spaces. Therefore it provides a wholesome information about convergence of the method. Theoretically the rate of convergence for apriori error estimation turns out to be $O(h+h^2+dt^2)$   and for aposteriori it is $O(h^2+dt^2)$  for backward Euler time discretization method. In numerical experiment section three cases viz. coupled Stokes/Transport model, coupled Brinkman/Transport model, coupled Stokes-Brinkman/Transport model with interface conditions, have been considered to cover all the different aspects of the model and in all of the three cases stabilized $ASGS$ method presents better performance with respect to standard Galerkin method.\vspace{1cm}\\
{\large \textbf{Acknowledgement}} \vspace{2mm}\\
This work has been supported by grant from Innovation in Science Pursuit for Inspired Research (INSPIRE) programme sponsored and managed by the Department of Science and Technology(DST), Ministry of Science and Technology, Govt.of India.

\end{document}